\documentclass[a4paper,11pt]{amsart}

\usepackage[utf8]{inputenc}
\usepackage{lmodern}
\usepackage{url}
\usepackage{amsmath, amsthm, amssymb, amscd, accents, bm}
\usepackage{mathtools}
\usepackage{mdwlist}
\usepackage{paralist}
\usepackage{graphicx}
\usepackage{epstopdf}
\usepackage{datetime}
\usepackage{afterpage}
\usepackage[dvipsnames]{xcolor}
\usepackage{hyperref}
\usepackage[ruled,vlined]{algorithm2e}
\usepackage{caption}
\usepackage{subcaption}
\usepackage{comment}
\mathtoolsset{showonlyrefs}
\usepackage[sort,nocompress]{cite}
\usepackage{todonotes}
\usepackage[foot]{amsaddr}
\usetikzlibrary{matrix}

\SetCommentSty{mycommfont}

\SetKwInput{KwInput}{Input}                
\SetKwInput{KwOutput}{Output}              
\SetKwInOut{Parameter}{Parameters}

\DeclarePairedDelimiter{\ceil}{\lceil}{\rceil}

\newcommand\rurl[1]{%
  \href{https://#1}{\nolinkurl{#1}}%
}

\newtheorem{definition}{Definition}[section]
\newtheorem{theorem}[definition]{Theorem}
\newtheorem{lemma}[definition]{Lemma}
\newtheorem{corollary}[definition]{Corollary}
\newtheorem{proposition}[definition]{Proposition}

\newtheorem{example}{Example}[section]
\newtheorem{remark}[definition]{Remark}

\newtheorem*{theorem*}{Theorem}
\newtheorem*{lemma*}{Lemma}

\def\N{{\mathbb N}}
\def\Z{{\mathbb Z}}
\def\R{{\mathbb R}}

\def\C{{\mathbb C}}

\renewcommand{\Re}{\operatorname{Re}}
\renewcommand{\Im}{\operatorname{Im}}

\newcommand{\E}{{\mathcal{E}}}
\newcommand{\Tcal}{{\mathcal{T}}}
\newcommand{\G}{\mathcal{G}}
\newcommand{\spect}{\mathcal{S}}
\newcommand{\B}{\mathcal{B}}

\newcommand{\lspan}{{\mathrm{span} \,}}
\newcommand{\supp}{{\mathrm{supp}}}

\DeclareMathOperator{\diag}{diag}
\DeclareMathOperator{\trace}{tr}
\DeclareMathOperator{\argmin}{argmin}
\DeclareMathOperator{\dist}{dist}

\DeclareMathOperator{\sinc}{sinc}

\newcommand{\V}{\mathcal{V}}

\newcommand{\re}{{\mathrm{Re} \,}}
\newcommand{\im}{{\mathrm{Im} \,}}

\newcommand{\eps}{\varepsilon}

\newcommand{\ft}{{\mathcal{F}}}

\makeatletter
\newcommand\thankssymb[1]{\textsuperscript{\@fnsymbol{#1}}}

\makeatletter
\def\@makefnmark{%
  \leavevmode
  \raise.9ex\hbox{\fontsize\sf@size\z@\normalfont\tiny\@thefnmark}}

\makeatletter
\def\bign#1{\mathclose{\hbox{$\left#1\vbox to8.5\p@{}\right.\n@space$}}\mathopen{}}

\usepackage{eqparbox}
\usepackage{makebox}

\title{Gabor phase retrieval via semidefinite programming}

\author{Philippe Jaming}
\email{philippe.jaming@math.u-bordeaux.fr}
\address{Univ. Bordeaux, CNRS, Bordeaux INP, IMB, UMR 5251, F-33400 Talence, France}

\author{Martin Rathmair}
\email{martin.rathmair@univie.ac.at}
\address{Faculty of Mathematics, University of Vienna, Oskar-Morgenstern-Platz
1, 1090 Vienna, Austria}
\date{October 2023}

\begin{document}

\maketitle
\begin{abstract}
We consider the problem of reconstructing a function $f\in L^2(\R)$ given phase-less samples of its Gabor transform, which is defined by 
$$
\G f(x,\omega) \coloneqq 2^{\frac14} \int_\R f(t) e^{-\pi (t-x)^2} e^{-2\pi i y t}\,\mbox{d}t,\quad (x,y)\in\R^2.
$$
More precisely, given sampling positions $\Omega\subseteq \R^2$ the task is to reconstruct $f$ (up to global phase) from measurements $\{|\G f(\omega)|: \,\omega\in\Omega\}$. 
This non-linear inverse problem is known to suffer from severe ill-posedness. As for any other phase retrieval problem, constructive recovery is a notoriously delicate affair due to the lack of convexity. One of the fundamental insights in this line of research 
is that the connectivity of the measurements is both necessary and sufficient for reconstruction of phase information to be theoretically possible.\\
In this article we propose a reconstruction algorithm which is based on solving two convex problems and, as such, amenable to numerical analysis. We show, empirically as well as analytically, that the scheme accurately reconstructs from noisy data within the connected regime.
Moreover, to emphasize the practicability of the algorithm we argue that both convex problems can actually be reformulated as semi-definite programs for which efficient solvers are readily available.\\
The approach is based on ideas from complex analysis, Gabor frame theory as well as matrix completion.

\vspace{1.5em}

\end{abstract}

\section{Introduction}
Phase retrieval problems are an important and common class of problems in the physical sciences 
with applications ranging from diffraction imaging \cite{shechtman:imaging} over audio \cite{griffinlim} to quantum mechanics \cite{Pauli1933}. 
In mathematical terms, phase retrieval is concerned with the reconstruction from phaseless linear measurements.
More formally, given $T:X\to X'$  a linear operator between two complex function spaces $X$ and $X'$, consider the map 
$$
\mathcal{A}:f\mapsto |Tf|^2 ,\quad f\in X
$$
where $|Tf|$ denotes the pointwise modulus of $Tf$, i.e. $|Tf|(\omega)=|Tf(\omega)|$. 
Solving the phase retrieval problem amounts to solving the inverse problem associated with the forward operator $\mathcal{A}$.
Clearly, $\mathcal{A}f=\mathcal{A}h$ if $h =cf$ with $|c|=1$.
To remove this obvious source of ambiguity one does not distinguish between functions which coincide up to a multiplicative factor of unit modulus,
$$
f \sim h \quad \Longleftrightarrow \quad \exists c\in\C,\,|c|=1:\quad h=cf.
$$
In a nutshell, the reconstruction task is to come up with the left inverse $\mathfrak{L}$ of $\mathcal{A}$, that is, a map satisfying 
$$
\forall f\in X:\quad  \mathfrak{L}(|Tf|^2) \sim f.
$$
It is important to note that it is by no means clear that the left inverse exists. 
For many instances of phase retrieval problems it is already difficult to decide the question of uniqueness, i.e., whether $\mathcal{A}$ is injective on $X\slash\sim$ or not.
Typically, there are three fundamental aspects studied within the scope of phase retrieval: uniqueness, stability and actual reconstruction.
In terms of the left inverse, these categories correspond to existence of $\mathfrak{L}$, continuity of $\mathfrak{L}$ and construction of $\mathfrak{L}$, respectively.

Ideally, given a concrete phase retrieval problem one would like to have a good understanding of all three aspects. Uniqueness and stability address the issue of well-posedness. Algorithmic methods to reconstruct are of fundamental importance for practical applications. In addition -- as for any other numerical method -- performance guarantees and a-priori estimates for the error are highly desirable, as is the capability of the method to deal with noisy input data.\\

Generally speaking, the algorithmic solution of phase retrieval problems is a notoriously delicate matter. 
Different variations of algorithms based on alternating projections have been around for decades,
such as the methods by Gerchberg-Saxton \cite{Gerchberg1972APA} and by Fienup \cite{Fienup:82} to name just two. 
These schemes have the advantage of being uncomplicated to implement. However, the absence of convexity makes them difficult to analyse, which is a major drawback when it comes to establishing convergence guarantees.
A rather recent trend is based on the idea of lifting the problem into a higher dimensional space and solving a convex reformulation (or relaxation) there.
For a number of algorithms of this type, results have been established which guarantee the accuracy of the method in a sufficiently randomized setting.
We mention the 'PhaseLift' algorithm (Cand\`es, Strohmer and Voroninski \cite{phaselift}) and 
the 'PhaseCut' algorithm (Waldspurger, d'Aspremont and Mallat \cite{phasecut}) as two of the earliest contributions into this direction.
For a comprehensive discussion of numerical aspects of phase retrieval as well as a detailed overview of the literature we refer to the recent survey \cite{survey:numericspr}.

\subsection{Aim of this paper \& related results}
In this paper we consider a particular instance of a phase retrieval problem, namely where the linear operator is the short-time Fourier transform with Gaussian window.
Precisely, this means that $T=\G$, the Gabor transform defined by
$$
\G f(x,y) = 2^{1/4} \int_\R f(t) e^{-\pi(t-x)^2} e^{-2\pi i y t}\,\mbox{d}t,\quad f\in L^2(\R),\,(x,y)\in\Omega,
$$
with $\Omega\subseteq \R^2$.
Evaluations of the Gabor transform correspond to correlation between $f$ and time-frequency shifts of a Gaussian
$$
\G f(x,y) = \langle f, \pi(\lambda) \varphi\rangle_{L^2(\R)},\quad \lambda=(x,y)
$$
where $\varphi\coloneqq 2^{1/4} e^{-\pi\cdot^2}$ and $\pi(\lambda)$ denotes a time-frequency shift, i.e.,
$$
(\pi(\lambda)g)(t) = e^{2\pi i yt} g(t-x).
$$
The squared modulus of the Gabor transform is called the (Gabor) spectrogram and denoted by $\spect f=|\G f|^2$.
The set $\Omega$ represents the locations where phaseless information is given.
If $\Omega = \R^2$ (in fact, $\Omega$ open suffices) it is well known that any $f\in L^2(\R)\slash\sim$ is uniquely determined by $\{\spect f(\omega), \omega\in\Omega\}$. 
Contrary to that, the uniqueness question for discrete sets $\Omega$ is already quite subtle. While sampling on a lattice (irrespective of its density!) is not enough \cite{alaifari:counterexamples, liehr:discbarriers}, the union of three suitable perturbations of a sufficiently dense lattice yields uniqueness \cite{grohs2023phase}.\\
Phase retrieval is inherently unstable in infinite dimensions \cite{cahill16,alaifari:banach} which aggravates matters even more. 
In the above setting this means that the left inverse of $f\mapsto \spect f$ (presuming that it exists) cannot be a well-behaved smooth function. 
Sources of instabilities are functions $f=f_1+f_2$ consisting of two (or more) components: we say that $f_1$ and $f_2$ are components if their respective Gabor transforms are essentially supported on disjoint domains (in particular, $\G f_1 \overline{\G f_2}\approx 0$). In this case we can flip the sign in front of either component and get similar phaseless observations:
\begin{align*}
|\G (f_1-f_2)|^2 &= |\G f_1|^2-2\Re(\G f_1 \overline{\G f_2}) + |\G f_2|^2 \\
&\approx |\G f_1|^2+2\Re(\G f_1 \overline{\G f_2}) + |\G f_2|^2  = |\G f|^2.
\end{align*}
On a positive note, functions of aforementioned type are the only source of instability; that is, if the spectrogram is connected, then phase information is stably determined by the phase-less observation \cite{grohs:stablegaborsc, grohs:stablegabormulti}.\\

In terms of reconstruction from phaseless Gabor measurements, to the best of our knowledge there is only one result into this direction at this moment in time. For signals residing in the shift-invariant space generated by a Gaussian, Grohs and Liehr \cite{grohs2022stable} establish an explicit inversion formula and prove that its discretization leads 
to a stable reconstruction method under suitable connectedness conditions.\\
 Escudero et al. \cite{escudero2022efficient} consider the problem of locating the positions of the zeros of the Bargmann transform, 
 which is closely related to our problem as Gabor and Bargmann transform coincide up to normalization (and due to the fact that the Bargmann transform is an entire function, and thus essentially determined by  its roots).\\
Finite-dimensional variants of the STFT phase retrieval problem have also been investigated extensively \cite{eldar15, Jaganathan16, prusa17, bendory18, Alaifari2021}. \\

As outlined above absence of uniqueness and uniform stability is a real issue for the problem at hand. Furthermore, any practical implementation is restricted to process a finite amount of data (in particular, $\Omega$ finite) while the problem itself is infinite dimensional.
Thus, it is inevitable to compromise on the original objective of actually constructing the left inverse $\mathfrak{L}$ of 
$$\mathcal{A}:f\mapsto   \spect f = |\G f|^2.$$
It is therefore the aim of this article to identify a map $\tilde{\mathfrak{L}}:\R^\Omega\to L^2(\R)$
which serves as an approximate left inverse of $\mathcal{A}$.

\subsection{Contribution}

In this paper we propose an algorithmic solution to compute an approximate left inverse $\tilde{\mathfrak{L}}$.
More specifically, the algorithm takes spectrogram samples as input, i.e., 
$$
\spect f(\omega) = |\G f(\omega)|^2,\quad \omega\in\Omega
$$
with $\Omega\subseteq \R^2$ finite, and aims to reconstruct $f$ up to a multiplicative constant of unit modulus.
Our method is based on solving two convex programs (CP) both of which can actually be formulated as semidefinite programs (SDP). 
For the reader who is not so familiar with convex optimization, the prototype of a SDP has the form
\begin{equation}
\begin{aligned}
\min  \quad & \langle X,A_0\rangle_F\\
\textrm{s.t.} \quad & \langle X,A_k\rangle_F = b_k,\quad k=1,\ldots,M\\
& X\succeq 0
\end{aligned}
\end{equation}
with $(A_k)_{k=0}^M$ a family of Hermitian matrices, $\langle A,B\rangle_F= \trace(B^H A)$ is the Frobenius inner product and
$(b_k)_{k=1}^M\subseteq \R$. 
We refer to the survey article by Vandenberghe and Boyd \cite{vandenberghe:sdp} on semidefinite programming and highlight the following snippet from the abstract of the aforementioned article to partly explain the popularity of SDP.
\begin{quote}
    "{\em Although semidefinite programs are much more general than linear programs, they are not much harder to solve.}"
\end{quote}

Further, the case where the constraints are inequalities $\langle X,A_k\rangle_F \geq b_k$ also fits in this framework.
We point out that it often makes sense to soften the constraints by replacing 
$\langle X,A_k\rangle_F = b_k$ by $|\langle X,A_k\rangle_F-b_k|\le \varepsilon$ with $\varepsilon$ a positive tolerance parameter, for example when the system is over-determined or in the presence of noise. 
A SDP of the above type can be easily set up and solved using the CVXPY package in Python.

\smallskip
The reconstruction scheme is accompanied by numerical analysis. 
Our results guarantee that the proposed method is accurate 
 within the stable regime (that is, if the spectrogram exhibits sufficient connectivity) and in the presence of noise, presuming that sufficient data is provided (i.e., $\Omega$ is rich enough).
Thus, our algorithm joins the rank of provably convergent
semi-definite programming-based methods for phase retrieval problems such as the famous 'PhaseLift' and 'PhaseCut' algorithms. We point out that -- in contrast to the aforementioned contributions -- our results hold in a purely deterministic context.

\subsection{Main ideas} 
Before providing some explanation of our approach we introduce the number 
$$
\mathfrak{a}\coloneqq \frac1{\sqrt2} = 0.7071\ldots.
$$
The square lattice with mesh size $\mathfrak{a}$ will be denoted by $\Lambda=\mathfrak{a}\Z^2$. It is important to keep in mind that in practice $\Lambda$ needs to be finite so we deal with finite dimensional objects. In this introductory part this aspect will however be neglected for the sake of simplicity of the argument as will be the potential influence of noise.
\subsubsection*{Reduction to discrete phase recovery problem}
Given $f\in L^2(\R)$ we introduce a sequence $v=(v_\lambda)_{\lambda\in\Lambda} \in \C^\Lambda$ by virtue of 
$$
v_\lambda \coloneqq \G f(\lambda), \quad \lambda\in\Lambda.
$$
Even though the index set $\Lambda\subseteq\R^2$ has a two-dimensional appearance we think of it as a one-dimensional vector.
As it is well-known from Gabor frame theory, the infinite vector $v$ carries all the relevant information of the function $f$ to be reconstructed \cite{daubechies:framesinbargmann, lyubarskii:framesinbargmann, seip:density1, seip:density2}. More explicitly, the relationship between $v$ and $f$ can be described in terms of an expansion formula involiving the so-called dual window $\psi\in L^2(\R)$:
\begin{equation}\label{eq:dualwindowrec}
f = \sum_{\lambda=(a,b)\in\Lambda} v_\lambda \pi(\lambda)\psi
\end{equation}
Thus, it suffices to recover the vector $v$ up to a global phase factor.
\subsubsection*{Lifting} 
The term 'lifting' refers to the idea of embedding the vector space $\C^\Lambda$ into the matrix space $\C^{\Lambda\times\Lambda}$. 
We consider the infinite matrix $v\otimes \overline{v}\in \C^{\Lambda\times \Lambda}$, which is defined by
$$
(v\otimes\overline{v})_{\lambda,\lambda'} \coloneqq v_\lambda\overline{v_{\lambda'}} = \G f(\lambda)\overline{\G f(\lambda')},\quad \lambda,\lambda'\in\Lambda.
$$
as an alternative representation of the vector $v$. 
The matrix entry $(v\otimes \overline{v})_{\lambda,\lambda'}$ contains the information of relative phase change between the locations $\lambda$ and $\lambda'$.
Note that $v\otimes\overline{v}$ determines $v$ uniquely up to a multiplicative constant of unit modulus -- $\lspan \{v\}$ is the solitary nontrivial eigenspace of the matrix $v\otimes\overline{v}$. Thus, this approach neatly fits into the scope of phase retrieval.
The lifting trick is well-known within the phase retrieval community and has proven to be quite useful, for instance to turn a given phase retrieval problem (which is non-convex by nature) into a convex problem by considering a relaxation in the matrix space.

\subsubsection*{Matrix estimator}
We are left with the problem of determining the entries of the matrix $v\otimes\overline{v}$ given spectrogram data $(\spect f(\omega))_{\omega\in\Omega}$.
In the best case (that is, if $\Omega\supseteq \Lambda$) we can directly observe the diagonal entries of the matrix as
$$
(v\otimes \overline{v})_{\lambda,\lambda} = \spect f(\lambda),
$$
while no off-diagonal entries are provided. 
To go beyond the diagonal we construct a predictor function $V$ based on holomorphically extending the spectrogram. In Section \ref{subsec:evalopmotivation} we prove that there exist $L:\R^2\times \R^2\to \C^2$ and $Q:\R^2\times\R^2 \to \C$ such that 
\begin{equation}\label{eq:reltensorholext}
\G f(p+u) \overline{\G f(p)} = F(L(p,u)) \cdot e^{Q(p,u)},\quad p,u\in\R^2
\end{equation} 
where $F:\C^2\to \C$ is the unique entire extension of $\spect f$ (that is, $F\in\mathcal{O}(\C^2)$ and $F\big|_{\R^2}=\spect f$).
Assuming that $F$ is known, it follows from \eqref{eq:reltensorholext} that 
$$
V(\lambda,\lambda') = F(L(\lambda',\lambda-\lambda')) \cdot e^{Q(\lambda',\lambda-\lambda')}
$$
correctly predicts the entries of the matrix $v\otimes\overline{v}$. 
In practice it is not viable to identify $F$ exactly. 
To obtain a feasible variant of the idea we will pick a finite-dimensional Ansatz space $\{F_A,\,A\succeq 0\}$ whose members are parameterized by positive definite matrices of a fixed size. Further, we formulate a semi-definite program over $\{A\succeq 0\}$, the convex cone of positive definite matrices such that its solution $A_*$ gives rise to an entire function $F_{A_*}$ which serves as an approximation for $F$. The resulting predictor is then
$$
\tilde{V}(\lambda,\lambda') = F_{A_*}(L(\lambda',\lambda-\lambda')) \cdot e^{Q(\lambda',\lambda-\lambda')}
$$
Analyzing the accuracy of the estimator reveals that $\tilde{V}(\lambda,\lambda')$ is a reliable approximation for $(v\otimes\overline{v})_{\lambda,\lambda'}$ if 
$$
(\lambda,\lambda') \in\mathcal{P}\coloneqq \{(\lambda,\lambda')\in \Lambda\times\Lambda,\, |\lambda-\lambda'|\le r\},
$$
with $r>0$ a suitable threshold parameter, whereas accuracy cannot be guaranteed when $|\lambda-\lambda'|>r$. Hence, $\tilde{V}$ can only be used to predict relatively few matrix entries.

\subsubsection*{Matrix completion}
That matrix completion techniques can be quite useful in the context of solving phase retrieval problems is well documented \cite{candes:prmc}.
For the problem at hand, to infer the remaining entries of the matrix $v\otimes \overline{v}$ we exploit the structural knowledge we have on that matrix: $v\otimes\overline{v}$ has rank one and is positive semi-definite. We attack the matrix completion problem by solving the convex relaxation 
\begin{quote}
    Find $Y\succeq 0$ s.t. $|Y_{\lambda,\lambda'}-\tilde{V}(\lambda,\lambda')| \le \tau$, for all $(\lambda,\lambda')\in\mathcal{P}$,
\end{quote}
with $\tau>0$ a suitable threshold. Quite conveniently, the relaxation can be formulated in terms of a SDP.
We pick up a result due to Demanet and Jugnon \cite{demanet17} (and establish a slight modification) to show that the convex relaxation is guaranteed to yield an accurate result provided that the data exhibit sufficient connectivity. 
The precise quantitative concept to measure connectivity is the spectral gap of the (vertex-weighted) graph Laplacian associated to $f$ (see Definition \ref{def:sagsp}).

\begin{figure}
\begin{tikzpicture}
  \matrix (m) [matrix of math nodes,row sep=4em,column sep=5em,minimum width=3em]
  {
     f\in L^2(\R)  & \spect f  \in [0,\infty)^\Omega \\
     v\in \C^\Lambda & v\otimes\overline{v}\in\C^{\Lambda\times\Lambda} \\};
  \path[-stealth]
    (m-2-1) edge [dashed,->] node [left] {dual window} (m-1-1)
    (m-1-1) edge node [above] {$\mathcal{A}$} (m-1-2)
    (m-2-2) edge [dashed,->] node [below] {leading eigenvector} (m-2-1)
    (m-1-2) edge [dashed,->] node [right] {estimator \& completion} (m-2-2);
\end{tikzpicture}
\caption{Schematic decomposition into the respective substeps in the construction of the approximate left inverse $\tilde{\mathfrak{L}}$ of $\mathcal{A}$.}
\end{figure}
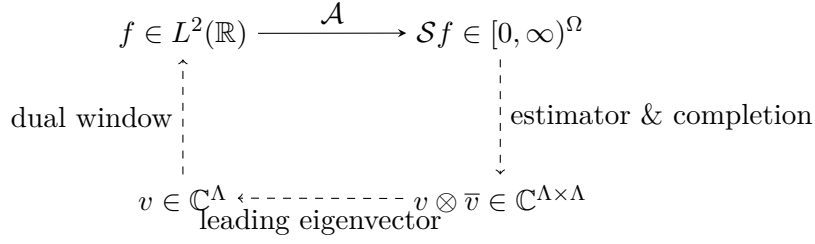

\subsection{Outline}
The remainder of this article is organized as follows.
Section \ref{sec:results} contains the precise definition of the reconstruction scheme as well as statements concerning the accuracy of the algorithm, which constitute the main results of this article.
In Section \ref{sec:numericalexp} we present some numerical experiments in order to verify the capabilities of the proposed algorithmic solution empirically.
Section \ref{sec:prelims} forms a collection of required preliminary material.
In Section \ref{sec:recfromincomplete} we discuss an inexact version of the dual window reconstruction formula \eqref{eq:dualwindowrec}.
Section \ref{sec:phasechanges} is concerned with a detailed analysis of the estimator $\tilde{V}$.
Finally, Section \ref{sec:proofthm1} contains the proof of the first main theorem, Theorem \ref{thm:main}.

\section{Algorithm and Results}\label{sec:results}

\subsection{The reconstruction scheme} This section is concerned with presenting the algorithmic solution of the phase retrieval problem at hand. 
As the method is quite involved we need to introduce a few objects and quantities first.\\
Throughout $\mathfrak{a}=\frac1{\sqrt2}$. The signal to be reconstructed is represented by $f\in L^2(\R)$. $\Gamma\subseteq \mathfrak{a}\Z^2$ denotes a finite subset of the lattice.
Furthermore, we denote the cone of positive definite matrices by $\mathfrak{A}_+(\Gamma)=\{A\in\C^{\Gamma\times\Gamma}:\, A\succeq 0\}$.
\begin{definition}[Ansatz function]
Let $\mathcal{J}=\begin{psmallmatrix}
    0&1\\
    -1&0
\end{psmallmatrix}$.
    Given $\lambda,\mu\in\R^2$, let
    $$
    \Phi_{\lambda,\mu}(z) := C(\lambda,\mu)
     e^{i\pi z^T \mathcal{J} (\lambda-\mu)} \cdot 
    e^{-\pi \big( z-\frac{\lambda+\mu}2\big)^2},
    \quad z\in\C^2,
    $$
    where $C(\lambda,\mu):= \exp\left\{ -\frac\pi4 \big|\lambda-\mu\big|^2 +\pi i (\lambda_1\lambda_2-\mu_1\mu_2)\right\}$.
    To every 
    $A\in \C^{\Gamma\times \Gamma}$ we associate the entire Ansatz function $F_A$ defined by
    $$
    F_A(z):=\sum_{\lambda,\mu\in\Gamma} A_{\lambda,\mu}\cdot \Phi_{\lambda,\mu}(z),\quad z\in\C^2.
    $$
\end{definition}
The prediction process involves evaluating a certain  entire function of two complex variables.
\begin{definition}[Evaluation operator]\label{def:evalop}
Let 
\begin{align*}
    L: &\begin{cases}
        \R^2 \times \R^2 &\rightarrow \C^2\\
        (p,u) &\mapsto p + \frac12
    \begin{psmallmatrix}
        1 & -i\\
        i & 1
    \end{psmallmatrix}
    u
    \end{cases}
    \\
    Q: &\begin{cases}
        \R^2 \times \R^2 &\rightarrow \C\\
        (p,u) &\mapsto -\frac\pi2 |u|^2-\pi i(2p_1+u_1)u_2
    \end{cases}
\end{align*}
    Given $G\in\mathcal{O}(\C^2)$, the \emph{evaluation} of $G$ is then defined by 
    $$
    \E[G](p,u) := G(L(p,u))\cdot e^{Q(p,u)}.
    $$
\end{definition}

We shall denote the canonical dual window of the Gabor frame $(\pi(\lambda)\varphi)_{\lambda\in\mathfrak{a}\Z^2}$ by $\psi$. Further, we point out that an explicit formula for $\psi$ is available, see equation \ref{eq:formulapsi}.\\
With this we are ready to formulate the  reconstruction scheme which is handed noisy spectrogram samples 
$
\sigma_\omega \approx \spect f(\omega), \omega\in\Omega
$
with $\Omega\subseteq \R^2$ finite and aims to reconstruct $f$ up to a global phase factor.

\begin{algorithm*}[H]\label{alg1}
\DontPrintSemicolon
  \Parameter{$r,\varepsilon,\varepsilon' >0$ and $\Lambda,\Gamma\subseteq \mathfrak{a}\Z^2$ finite}
  \KwInput{$\sigma\in \R^\Omega$}
  \KwOutput{$f_\ast \in L^2(\R)$}
  
  \underline{Step 1:} Solve the CP
  \begin{equation}\label{eq:entextcp}
\begin{aligned}
\min_{A\in\mathfrak{A}_+(\Gamma)} \quad & \max_{\lambda\in\Gamma} A_{\lambda,\lambda}\\
\textrm{s.t.} \quad & |F_A(\lambda) -\sigma_\lambda| \le \varepsilon, \quad \lambda\in \Omega,
\end{aligned}
\end{equation}
and denote its solution by $A_*$.

\underline{Step 2:} Set $\mathcal{P}\coloneqq \{(\lambda,\lambda')\in\Lambda\times\Lambda:\, |\lambda-\lambda'|\le r\}$ and define coefficients
\begin{equation}
    T_{\lambda',\lambda}:= \E[F_{A_*}](\lambda,\lambda'-\lambda),\quad (\lambda, \lambda') \in \mathcal{P}.
\end{equation}
Solve the CP 
\begin{equation}\label{eq:convfeasprob}
\text{find } Y\in\mathfrak{A}_+(\Lambda), \text{ s.t. } |Y_p - T_p|\le \varepsilon', \quad p\in\mathcal{P}.
\end{equation}

\underline{Step 3:} Extract $v\in \C^\Lambda$ an eigenvector corresponding to the largest eigenvalue of $Y$, and set
$$
f_* := \sqrt{\trace(Y)} \sum_{\lambda\in \Lambda} v_\lambda\cdot \pi(\lambda)\psi.
$$
\caption{Reconstruction procedure}
\end{algorithm*}

\begin{remark}[CPs are SDPs]
Both, the CP in step $1$ as well as the CP in step $2$ can actually be recast into a SDP.

\smallskip

Indeed, to see this for the CP in step $1$ we refer to 
Lemma \ref{lem:matrixevaluationFA}, where we explicitly construct for any given point $z=(x,y)\in\R^2$, a hermitian matrix $W_z\in \mathfrak{A}_+(\Gamma)$ with the property
$$
F_A(z) = \langle A, W_z\rangle_F, \quad A\in \mathfrak{A}_+(\Gamma).
$$
Furthermore, we add a single element '$\infty$' to the index set $\Gamma$ and denote $\tilde{\Gamma}=\Gamma \cup \{\infty\}$. 
With this, we introduce a slack variable $\mu\in\R$ by considering
$$
\tilde{A} = 
\left(
\begin{array}{c|c}
    A & \ast\\
    \hline
    \ast &\mu
\end{array}
\right)
\in \mathfrak{A}_+(\tilde{\Gamma}).
$$
It is not difficult to see that $A$ is a solution of \eqref{eq:entextcp} if and only if $\tilde{A}$ is a solution of the SDP
\begin{equation}
    \begin{aligned}
        \min_{\tilde{A}\in \mathfrak{A}_+(\tilde{\Gamma})} \mu &= \min_{\tilde{A}\in \mathfrak{A}_+(\tilde{\Gamma})} \Big\langle \tilde{A},
        \left(
        \begin{array}{c|c}
        0 &0\\
        \hline
        0 &1
        \end{array}
        \right)
        \Big\rangle_F\\
        \text{s.t. } \, &\left| \Big\langle \tilde{A}, 
        \left(
        \begin{array}{c|c}
        W_\lambda & 0\\
        \hline
        0 & 0
        \end{array}
        \right)
        \Big\rangle_F - \sigma_\lambda \right| \le \varepsilon, \quad \lambda\in\Omega\\
         & \langle \tilde{A}, 
        D_\gamma \rangle
         \le 0, \quad \gamma\in\Gamma
    \end{aligned}
\end{equation}
where the diagonal matrix $D_\gamma$ is given by 
$$
D_\gamma(\alpha,\beta) = \begin{cases}
    -1,\quad &\alpha=\beta=\infty\\
    1,  &\alpha=\beta=\gamma\\
    0,  &\text{otw.}
\end{cases}
$$
For the corresponding statement with regards to the CP in step $2$, we refer to Lemma \ref{lem:matrixcomplsdp}, where we explicitly construct,
given an arbitrary point $p\in \Lambda\times\Lambda$, a pair of hermitian matrices $E_p^r$ and $E_p^i$ with the property that for each 
$A\in \mathfrak{A}_+(\Lambda)$, 
$$
\langle A, E_p^r\rangle_F = \Re(A_p), \quad \text{and} \quad \langle A, E_p^i\rangle_F = \Im(A_p).
$$
Therefore, every solution of the SDP 
\begin{equation}
    \begin{aligned}
    \min_{Y\in\mathfrak{A}_+(\Lambda)} &\langle Y, 0\rangle_F\\
    \text{s.t.} \quad & |\langle Y, E_p^r \rangle_F - \Re T_p| \le \frac{\varepsilon'}2, \quad p\in \mathcal{P}\\
    & |\langle Y, E_p^i \rangle_F - \Im T_p| \le \frac{\varepsilon'}2, \quad p\in \mathcal{P}
    \end{aligned}
\end{equation}
is also a solution of the CP \eqref{eq:convfeasprob}.
Finally, we point out that off-the-shelf solvers (in particular, the CVXPY package which was used to produce the experiments carried out in Section \ref{sec:numericalexp}) deal with these aspects automatically and internally. For example, to solve the CP \eqref{eq:entextcp} we simply declared $\max_{\lambda\in\Gamma} A_{\lambda,\lambda}$ as the objective function to be minimized.
\end{remark}

\subsection{Reconstruction of Gabor coefficients}
The intermediate objective of the algorithm is concerned with the reconstruction of the coefficients 
$\{\G f(\lambda),\lambda\in\Lambda\}$. 
Before we present the result which guarantees the accuracy of this first component of the method we again have to settle some terminology.\\

First we attach a graph to the provided spectrogram samples in order to specify the relevant quantity of connectivity.  
\begin{definition}[Signal associated graph and spectral gap]\label{def:sagsp}
    Given a triple consisting of $f\in L^2(\R)$, a finite set     
    $\Lambda\subseteq \mathfrak{a}\Z^2$ and $r>0$,  
    the signal associated graph $G=(V,E,\alpha)$ is the vertex weighted graph with vertex set $V=\Lambda$, edge set $E$ defined by 
    $$
    (u,v)\in E \quad \Leftrightarrow \quad 0<|u-v|<r,
    $$
    and vertex weights $\alpha_v=\spect f(v)$, $v\in V$.
    Let $0=\lambda_1\le \lambda_2 \le \ldots$ denote the  eigenvalues of the Laplacian
    \footnote{we recall the definition of the Laplacian of a vertex weighted graph in Section \ref{sec:vwg}}
    of $G$. The \emph{spectral gap} is defined as
    $$
    \lambda_2(f,\Lambda,r) := \lambda_2.
    $$
\end{definition}

\begin{remark}[Spectral gap and connectivity]
    We recall in Lemma \ref{lem:graphconnection} that $G$ is connected if and only if there is an actual spectral gap, i.e. if $\lambda_2$ is strictly positive. This suggests to conceive the spectral gap $\lambda_2$ as a concept to quantify the connectivity of $G$: the larger the spectral gap of the Laplacian, the more connected is $G$.\\
    A formal approval of this intuition has been established by Cheeger \cite{Cheeger:lowerbound} in the context of Riemannian geometry by relating the spectral gap to a geometric invariant coined the {\em Cheeger constant}. Subsequently, respective results have also been encountered in the world of graph theory \cite{chung:spgrth}.
\end{remark}

In the remainder we will stick to index sets $\Lambda,\Gamma,\Omega$ of a rather simple form.
Namely, we assume that these sets arise as the intersection between a centered rectangle with a square lattice. 
Precisely, given numbers $T,S,R,s>0$ we set 
    \begin{align}
         \Lambda &= ([-T,T]\times [-S,S]) \cap \mathfrak{a}\Z^2, \label{def:Lambda} \\
        \Omega &= ([-T-R,T+R]\times [-S-R,S+R]) \cap s\Z^2, \label{def:Omega}\\
        \Gamma &= ([-T-2R,T+2R]\times [-S-2R,S+2R])\cap \mathfrak{a}\Z^2. \label{def:Gamma}
    \end{align}
The parameters need to be chosen reasonably for the reconstruction method to succeed. 
For this purpose we introduce the following technical notion.
\begin{definition}
    Let $f\in L^2(\R)$. 
    Given parameters $T,S,R,r,s,\varepsilon,\varepsilon>0$ let $\Lambda,\Gamma,\Omega$ be defined as in \eqref{def:Lambda}, \eqref{def:Omega} and \eqref{def:Gamma}, respectively.\\ 
    We say that the parameters are \emph{well calibrated} if all of the following relations hold true:
    \begin{align}
        \varepsilon &\le \left[\frac{e^{-\frac{17\pi}{32}r^2}}{1.33\times 10^5} \times \min\left\{\frac{\|\G f\|_{\ell^2(\Lambda)}^2}{|\Lambda|^2}, \frac{ \lambda_2(f,\Lambda,r)}{192 r^2 }\right\} \right]^2, \label{eq:condepsilon}\\
        \varepsilon'&=(3.1\times 10^4) \sqrt{\varepsilon} e^{\frac{17\pi}{32}r^2}, \label{eq:condepsilonprime}\\
         s &\le \frac{0.3}{\sqrt{ \ln \left(\frac2{3\varepsilon} \right)}}, \label{eq:conds}\\
        R &\ge \max\left\{ 2.1+0.9 \sqrt{\ln \left(\frac1\varepsilon\right)}, 
        \frac{r+s^{-1}}2
        \right\} \label{eq:condR}
    \end{align}
\end{definition}

\begin{remark}[Number of required samples]
    The parameters $T,S,r$ are considered to be fixed numbers. Typically $r\in [0,2]$, and expressions involving $r$ in \eqref{eq:condepsilon} - \eqref{eq:condR} may just be regarded to be constants.
    The parameter $\varepsilon$ plays the role of the desired error margin.
   From a qualitative perspective (neglecting numerical constants) we get that the required number of samples in the ``well-calibrated regime'' is roughly (asymptotically for  $\varepsilon$ close to $0$)
   \begin{align*}
       |\Omega| & = \Big| s\Z^2 \cap ([-T-R,T+R] \cap [-S-R,S+R])\Big| \\
       &\asymp \frac{(T+R)(S+R)}{s^2}\\
       &\asymp \ln \left(\frac2{3\varepsilon}\right) \cdot \left( T + \sqrt{\ln \left(\frac1\varepsilon \right)} \right) \cdot \left( S + \sqrt{\ln \left(\frac1\varepsilon \right)} \right)
   \end{align*}
   Hence, $|\Omega| = \mathcal{O}(\ln^2 \frac1\varepsilon)$ as $\varepsilon\to 0$.
\end{remark}

The first main result states that the proposed method recovers the samples $(\G f(\lambda))_{\lambda\in\Lambda}$ accurately under appropriate assumptions and reads as follows.
\begin{theorem}\label{thm:main}
    Let $f\in L^2(\R)$ be such that $\|\spect f\|_{L^\infty}\le 1$.\footnote{the assumption that $\|\spect f\|_{L^\infty}\le 1$ is purely there for aesthetic reasons in order to keep bounds and implicit constants as simple as possible}
    Suppose that $T,S,R,r,s,\varepsilon,\varepsilon'$ are well calibrated.
    Let $\sigma\in \R^{\Omega}$ be such that 
    $$
    \| \sigma- \spect f\|_{\ell^\infty(\Omega)} \le \frac\varepsilon2.
    $$
    Then it holds that both convex problems in Algorithm \ref{alg1} are feasible.
    Moreover, the top eigenvector $v\in \C^\Lambda$  satisfies that 
    \begin{multline}\label{eq:bounddelift}
    \min_{\theta\in\R} \left|\sqrt{\trace(Y)} v - e^{i\theta} (\G f(\lambda))_{\lambda\in\Lambda} \right| \\
    \le 177 
    e^{0.84 r^2} \cdot \left(1+20 r \sqrt{\frac{\|\G f\|_{\ell^2(\Lambda)}^2}{\lambda_2(f,\Lambda,r)}} \right)\cdot\sqrt[4]\varepsilon.
    \end{multline}
\end{theorem}
The proof of Theorem \ref{thm:main} can be found in Section \ref{sec:proofthm1}. 
\begin{remark}[On the choice of $r$]
    According to Theorem \ref{thm:main} the accuracy of the prediction is driven by  the quantity
    \begin{equation}
        C_{stab}(f,r)\coloneqq  e^{0.84r^2}\left( 1+ 20r\sqrt{\frac{\|\G f\|_{\ell^2(\Lambda)}^2}{\lambda_2(f,\Lambda,r)}} \right)
    \end{equation}
    The statement is empty if $\lambda_2(f,\Lambda,r)=0$, i.e., if the associated graph is not connected.
    For a generic  $f\in L^2(\R)$, we have that $\G f$ does not vanish on the finite set $\Lambda$. Thus -- provided that $r\ge \mathfrak{a}$ -- in general (for generic $f$) we have that the graph is indeed connected.\\
    Furthermore, the performance of the reconstruction method may crucially depend on the choice of $r$.
    Given $f$, the optimal\footnote{in the sense that the upper bound provided by Theorem \ref{thm:main} is then the tightest possible} choice is  
    $$r=r_*(f)\coloneqq \argmin_{r>0} C_{stab}(f,r).$$
    It is important to note that $r_*(f)$ can be computed if samples $(\spect f(\lambda))_{\lambda\in\Lambda}$ are given: the Laplacian (and therefore also the spectral gap $\lambda_2(f,\Lambda,r)$) depends solely on $(\spect f(\lambda))_{\lambda\in\Lambda}$ and moreover as $r\mapsto\lambda_2(f,\Lambda,r)$ is piecewise constant, it suffices to check a finite number of candidates $r$ (cf. Figure \ref{fig:diffrs} for the three smallest choices).
\end{remark}

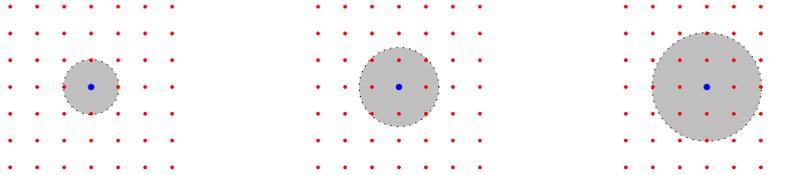
\begin{figure}[ht]
\centering
\begin{subfigure}{0.31\linewidth}
\centering
\begin{tikzpicture}[scale=0.5]
\draw [fill=lightgray,dotted] (0,0) circle (0.73);
\foreach \x in {-3,...,3} \foreach \y in {-3,...,3} \draw[red,fill=red] (\x*0.71,\y*0.71) circle [radius=1pt];
\draw[blue, fill=blue] (0,0) circle [radius=2pt];
\end{tikzpicture}
\end{subfigure}
\begin{subfigure}{0.31\linewidth}
\centering
\begin{tikzpicture}[scale=0.5]
\draw [fill=lightgray,dotted] (0,0) circle (1.05);
\foreach \x in {-3,...,3} \foreach \y in {-3,...,3} \draw[red,fill=red] (\x*0.71,\y*0.71) circle [radius=1pt];
\draw[blue, fill=blue] (0,0) circle [radius=2pt];
\end{tikzpicture}
\end{subfigure}
\begin{subfigure}{0.31\linewidth}
\centering
\begin{tikzpicture}[scale=0.5]
\draw [fill=lightgray,dotted] (0,0) circle (1.44);
\foreach \x in {-3,...,3} \foreach \y in {-3,...,3} \draw[red,fill=red] (\x*0.71,\y*0.71) circle [radius=1pt];
\draw[blue, fill=blue] (0,0) circle [radius=2pt];
\end{tikzpicture}
\end{subfigure}
\caption{Three different choices for the parameter $r$. For the values $r=\mathfrak{a}=\frac1{\sqrt2}$ (left), $r=1$ (center) and $r=\sqrt2$ (right) the blue vertex in the middle has $4,8$ and $12$ neighbors, respectively.}
\label{fig:diffrs}
\end{figure}

\begin{figure}[ht]
\centering
\includegraphics[width=\textwidth]{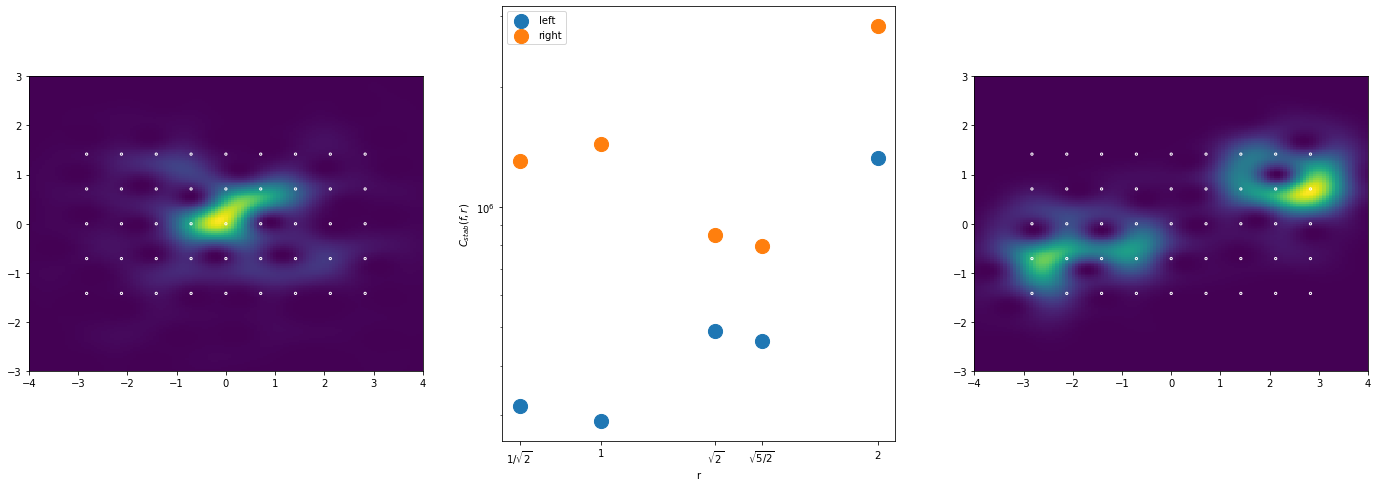}
\caption{We compare two different functions: one having a spectrogram which is connected (left) and the other one consisting of two connected components (right).
The white points indicate the finite lattice $\Lambda\subseteq \mathfrak{a}\Z^2$. The graph in the center shows the dependence of $C_{stab}(f,r)$ on the parameter $r$ for each of the functions. These computations suggest that the performance of the algorithm will benefit from choosing the parameter $r$ to be relatively large ($r=\sqrt{5/2}$) in the disconnected case, while it is favorable to pick $r=1$ in the connected case.  }
\label{fig:ropt}
\end{figure}

\begin{remark}[Comparison with existing stability results]
Next we put the above result into perspective with earlier stability results for the Gabor phase retrieval problem. More specifically, with the outcomes of \cite{grohs:stablegaborsc,grohs:stablegabormulti}, where it has been shown that on a domain $\Omega\subseteq \R^2$ on which $|\G f|$ is connected, phase information is stably encoded in the phase-less data $|\G f|$.
Even though inequality \eqref{eq:bounddelift} is of discrete type, qualitatively it is very much reminiscent of the aforementioned results:
While the results on the continuous level quantify connectivity in terms of the Poincaré constant (which is known to coincide with the spectral gap of the Laplacian if objects are correctly defined), we  here employ the spectral gap of the corresponding graph Laplacian.
To state the obvious, the substantial advantage of the novel results of this article is that we have a  mean to actually recover phase information.
\end{remark}

\subsection{Reconstruction from incomplete Gabor coefficients}
Given $\Lambda\subseteq \mathfrak{a}\Z^2$ we formally define
an operator $\mathcal{R}_\Lambda:\C^\Lambda\to L^2(\R)$ by 
\begin{equation}\label{eq:recdualwindow}
\mathcal{R}_\Lambda(c) = \sum_{\lambda\in\Lambda} c_\lambda \pi(\lambda)\psi, \quad c=(c_\lambda)_{\lambda\in\Lambda}
\end{equation}
where $\psi$ denotes the dual window of the Gabor frame $(\pi(\lambda)\varphi)_{\lambda\in\mathfrak{a}\Z^2}$.
If $\Lambda=\mathfrak{a}\Z^2$ and $c=(\G f(\lambda))_{\lambda\in\Lambda}$ the sum in \eqref{eq:recdualwindow} converges unconditionally in $L^2(\R)$ and exactly reconstructs $f$, that is, $f=\mathcal{R}_\Lambda(c)$, cf. Section \ref{sec:gaussgabor}.\\

In practice, one only has a finite number of possibly noisy samples at one's disposal. It is the purpose of the next result to control the reconstruction error in this inexact setting. 
As before, given $T,S>0$ let us denote
$$
\Lambda=  ([-T,T]\times[-S,S])\cap \mathfrak{a}\Z^2.
$$
Given the fact that only a limited amount of samples is available (particularly, $\Lambda\subseteq [-T,T]\times \mathbb{R}$) one cannot expect that $f$ is to be accurately reconstructed outside the interval $[-T,T]$. Moreover, given that there is no information about large frequencies ($>S$) it is to be expected that the reconstruction is accurate only if the function to be reconstructed is relatively smooth.\\

Before we state the result we introduce two quantities.
Given $f\in L^2(\R)$ we use 
$$
\eta(f)\coloneqq 
\sup_{x\in\R} \|f\cdot T_x\varphi^{\frac12}\|_{L^2}
=
2^{\frac18} \sup_{x\in\R}  \left( \int_\R |f(t)|^2 e^{-\pi(t-x)^2}\,\mbox{d}t\right)^{1/2},
$$
to capture the maximal local $L^2$-energy of $f$.
Given $S>0$, let
\begin{equation*}
\kappa_S(f) \coloneqq \sup_{x\in\R} \left(  \int_{|\omega|>S} |\G f(x,\omega)|^2\,\mbox{d}\omega\right)^{1/2} =\sup_{x\in\R} \|\ft\{f e^{-\pi(\cdot-x)^2}\}\|_{L^2([-S,S]^c)}.
\end{equation*}
The quantity $\kappa_S(f)$ may be understood as a measure of smoothness of $f$: If $\kappa_S(f)$ is negligibly small, then all Gaussian localizations $\{f e^{-\pi(\cdot-x)^2}, x\in\R\}$ are close to bandlimited (with bandwidth $2S$), and vice versa.\\

\smallskip
The second main result reads as follows.
\begin{theorem}\label{thm:main2}
    Let $T,S\in \mathfrak{a}\N$, let $0<\tau<T$.
    For all $f\in L^2(\R)$ and for all $c\in\C^\Gamma$ it holds that 
      \begin{multline}
        \|f-\mathcal{R}_\Lambda(c)\|_{L^2(-\tau,\tau)} \le \\ 
        6.82
        \left(
    \|\G f- c\|_{\ell^2(\Lambda)} +\sqrt{T+1} \cdot \kappa_S(f) + 
    \sqrt{\tau+1} e^{-\frac\pi{\sqrt2} (T-\tau)}\cdot \eta(f)
    \right).
    \end{multline}
\end{theorem}
For the proof of Theorem \ref{thm:main2} we refer to Section \ref{sec:recfromincomplete}.
There are three different terms contributing to the upper bound on the reconstruction error.
\begin{itemize}[--]
    \item The first term is simply the noise on the provided samples.
    \item  Recall that $\Lambda\subseteq [-T,T]\times [-S,S]$. This means that the reconstruction method only draws on Gabor samples located in the strip $\R\times [-S,S]$ while neglecting information from samples outside the strip. Morally, if there is only little contribution from samples outside the strip (that is, if $\kappa_S(f)$ is small) then neglecting these samples does not affect the accuracy too much.
    \item The decisive factor in the third term is the exponential $e^{-\frac\pi2 (T-\tau)}$. Since we want to accurately reconstruct on the interval $(-\tau,\tau)$ we require that $T>\tau$; clearly, the reconstruction accuracy will benefit from being fed more information (i.e., increasing $T$).
    The estimate shows that the improvement on the corresponding error term is exponential with respect to the offset $T-\tau$.
\end{itemize}
\subsection{Reconstruction from spectrogram samples}
The content of the main result is to clarify under which conditions and in which sense $f_*$, the output of Algorithm \ref{alg1} is a useful estimate for $f$. 

\begin{theorem}\label{thm:main3}
    Let $T,S\in\mathfrak{a}\N$ and let $0<\tau<T$. Furthermore, suppose the assumptions of Theorem \ref{thm:main}.
    Then it holds that $f_*$, the output of Algorithm \ref{alg1} satisfies that 
    \begin{multline}\label{eq:estmain3}
        \min_{\theta\in\R} \|f_*-e^{i\theta}f\|_{L^2(-\tau,\tau)} \\
        \le 18 \Bigg[ 177 
    e^{0.84 r^2} \cdot \left(1+20 r \sqrt{\frac{\|\G f\|_{\ell^2(\Lambda)}^2}{\lambda_2(f,\Lambda,r)}} \right)\cdot\sqrt[4]\varepsilon\\
    + (2T+2) \kappa_S(f)
    + 2\sqrt{S}(\tau+1)e^{-\frac\pi2(T-\tau)}
        \Bigg].
    \end{multline}
\end{theorem}
The result is essentially a direct consequence of Theorem \ref{thm:main} and Theorem \ref{thm:main2} and the proof will be provided right away at this stage.
\begin{proof}
    Let $c\in\C^\Lambda$ be defined by 
    $$
    c_\lambda\coloneqq \sqrt{\trace (Y)} v_\lambda,\quad \lambda\in\Lambda.
    $$
    As per Theorem \ref{thm:main} there exists $\theta_0\in\R$ such that 
    \begin{equation}\label{eq:asstthm1rewritten}
        |c-e^{i\theta_0}(\G f(\lambda))_{\lambda\in\Lambda}| \le 177 
    e^{0.84 r^2} \cdot \left(1+20 r \sqrt{\frac{\|\G f\|_{\ell^2(\Lambda)}^2}{\lambda_2(f,\Lambda,r)}} \right)\cdot\sqrt[4]\varepsilon.
    \end{equation}
    We continue by estimating
    \begin{align*}
        \min_{\theta\in\R}\|f_*-e^{i\theta}f\|_L^2(-\tau,\tau) &= \min_{\theta\in\R}\|\mathcal{R}_\Lambda(c)-e^{i\theta}f\|_{L^2(-\tau,\tau)}\\
        &\le \|\mathcal{R}_\Lambda(c)-e^{i\theta_0}f\|_{L^2(-\tau,\tau)}\\
        &= \|\mathcal{R}_\Lambda(e^{-i\theta_0}c)-f\|_{L^2(-\tau,\tau)}.
    \end{align*}
    As per Theorem \ref{thm:main2} we have that the right hand side is bounded from above by 
    \begin{equation}
        18
        \left(
    \|\G f- e^{-i\theta_0}c\|_{\ell^2(\Lambda)} +\sqrt{T+1} \cdot \kappa_S(f) + 
    (\tau+1) e^{-\frac\pi2 (T-\tau)}\cdot \eta(f)
    \right).
    \end{equation}
    
    We claim that 
    \begin{equation}\label{eq:boundeta}
        \eta(f) \le 2 \sqrt{S}+ \kappa_S(f).
    \end{equation}
    With this and with \eqref{eq:asstthm1rewritten} we then get that 
    \begin{multline}
        \min_{\theta\in\R} \|f_*-e^{i\theta}f\|_{L^2(-\tau,\tau)}
        \le 18 \Bigg[
        177 
    e^{0.84 r^2} \cdot \left(1+20 r \sqrt{\frac{\|\G f\|_{\ell^2(\Lambda)}^2}{\lambda_2(f,\Lambda,r)}} \right)\cdot\sqrt[4]\varepsilon\\
    + (\sqrt{T+1} +(\tau+1)e^{-\frac\pi2(T-\tau)})\cdot \kappa_S(f)
    + 2\sqrt{S}(\tau+1)e^{-\frac\pi2(T-\tau)}
        \Bigg]
    \end{multline}
    Notice that 
    $$
    \sqrt{T+1}+(\tau+1)e^{-\frac\pi2 (T-\tau)}\le 2T+2,
    $$
    which implies the desired estimate \eqref{eq:estmain3}.\\
    
    To establish \eqref{eq:boundeta}
    let $v\in\R$ arbitrary.
    Since $\G f(v,\cdot)=\ft\{f e^{-\pi(\cdot-v)^2}\}$ it follows from Plancherel's formula that 
    \begin{equation}\label{eq:estfcon}
    \|f e^{-\pi(\cdot-v)^2}\|_{L^2}^2 = \|\G f(x,\cdot)\|_{L^2}^2
    \le 2S + \kappa_S(f)^2,
    \end{equation}
    where we used that  $\|\G f\|_{L^\infty}\le 1$ by assumption. Since
    $$
    \int_\R e^{-2\pi (t-u)^2} e^{-2\pi u^2}\,\mbox{d}u = \frac12 e^{-\pi t^2},
    $$
    we then get for any $x\in\R$ (by applying \eqref{eq:estfcon} to $v=x+u$) that
    \begin{align*}
        \int_\R |f(t)|^2 e^{-\pi(t-x)^2}\,\mbox{d}t 
        &= 2 \int_\R |f(t)|^2 \left( \int_\R e^{-2\pi(t-x-u)^2-2\pi u^2}\,\mbox{d}u \right) \,\mbox{d}t\\
        &= 2 \int_\R \left(\int_\R |f(t)|^2 e^{-2\pi(t-x-u)^2}\,\mbox{d}t \right) e^{-2\pi u^2}\,\mbox{d}u\\
        &\le (2S+\kappa_S(f)^2) \cdot 2 \int_\R e^{-2\pi u^2}\,\mbox{d}u\\
        &= \sqrt{2} (2S+\kappa_S(f)^2)
    \end{align*}
    As $x$ was arbitrary, taking square roots implies \eqref{eq:boundeta}, and we are done.
\end{proof}

\section{Numerical simulations}\label{sec:numericalexp}
First of all, the purpose of the present section is to demonstrate empirically the aptitude of the proposed reconstruction algorithm. 
Furthermore, we aim to develop some sense for the choice of parameters of the scheme.

\smallskip

    We consider the reconstruction of certain random signals.
    The precise random model is the following: 
    Given $a>0$ and $\Xi \subseteq a\Z^2$ finite, 
    we pick a vector $(c_\lambda)_{\lambda\in\Xi} \in \C^\Xi$ at random, with each component chosen independently and according to the uniform distribution on the complex unit disk $\mathbb{D}\subseteq \C$. The resulting function is then given by    
    $$
    f = \sum_{\lambda \in\Xi} c_\lambda \cdot \pi(\lambda)\varphi.
    $$ 
    The algorithm takes noisy spectrogram measures as input.
    Given a noise level $\nu\ge 0$, the noisy sampled data fed into the algorithm are given by 
    $$
    \sigma_\lambda = \spect f(\lambda) + \eta_\lambda,\quad \lambda\in\Omega,
    $$
    where $(\eta_\lambda)_{\lambda\in\Omega}$ are independently and identically distributed random variables according to uniform distribution on the interval $[-\nu,\nu]$.
    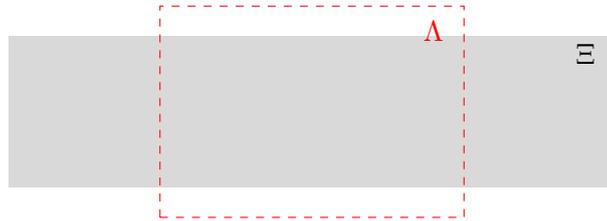
\begin{figure}[h!]
     \begin{tikzpicture}
\path[fill = gray!30] (-1,0) rectangle (7,2);
\node[above] at (6.6,1.5) {$\Xi $};
\draw[dashed, color = red] (1,-0.4) rectangle (5,2.4);
\node[above, color = red] at (4.6, 1.8) {$\Lambda$};
\end{tikzpicture}
\caption{In order to make sure that the reconstructed coefficients $(\G f(\lambda)_{\lambda\in\Lambda}$ carry all the relevant information of $f$ in an interval around zero, $\Xi$ is chosen in such a way that in the vertical direction it does not exceed the zone occupied by $\Lambda$. This means if $\Xi =a\Z^2\cap ([-T',T']\times [-S',S'])$ and $\Lambda = \mathfrak{a}\Z^2\cap ([-T,T]\times[-S,S])$, then $S\ge S'$.
}
    \end{figure}

\begin{figure}
\begin{subfigure}{0.48\textwidth}
\includegraphics[width=\textwidth]{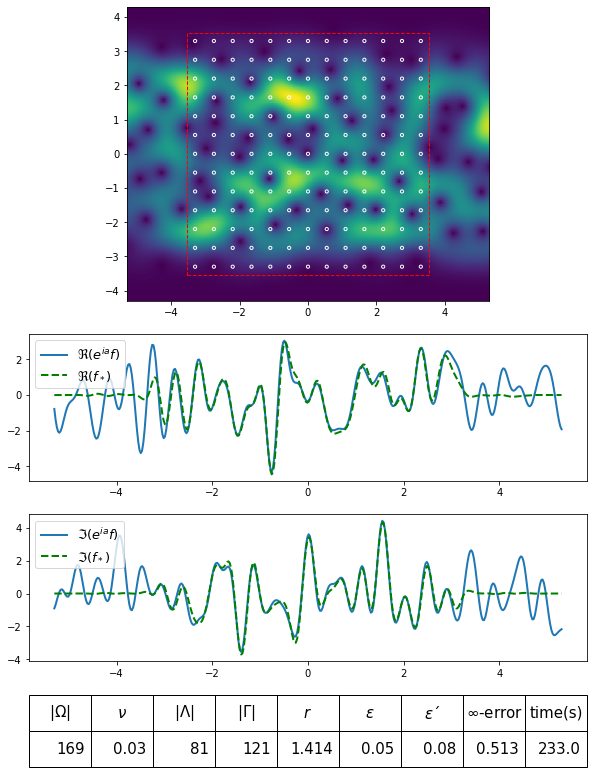}
\end{subfigure}
\hfill
\begin{subfigure}{0.48\textwidth}
\includegraphics[width=\textwidth]{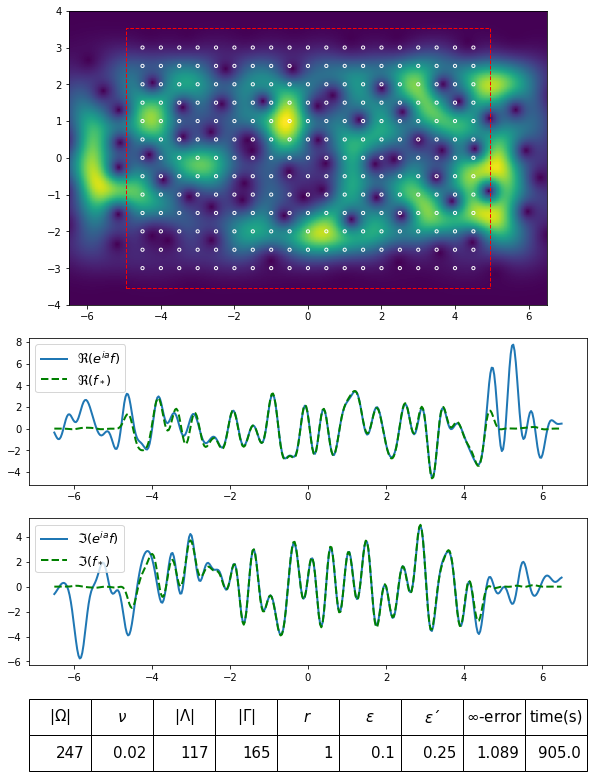}
\end{subfigure}
\caption{ The top plot shows $|\G f|$ for a randomly generated $f$.
The white points indicate the set $\Omega$, i.e., the position of the samples fed into the scheme.
The red rectangle indicates the zone in which samples of $\G f$ will be reconstructed:  all points of $\mathfrak{a}\Z^2\setminus \Lambda$ are on the red dashed line or outside of the rectangle.
The two plots below compare real and imaginary part of $f$ and its reconstruction $f_*$, respectively.
In order to make these two functions comparable $f$ is multiplied by an appropriate constant phase factor first. 
The table at the bottom summarizes a couple of relevant specifications of the experiment carried out. 
}\label{fig:randomfctsrec}
\end{figure}

The main difference between the two experiments (Figure \ref{fig:randomfctsrec} is the interval on which $f$ is reconstructed. 
On the left, reconstruction is fairly accurate on $[-3,3]$, while on the right hand side, we obtain a good approximation for the larger interval $[-4.5, 4.5]$.
This comes at the price of having to solve a problem of higher complexity. The bottleneck of the reconstruction scheme lies in identifying $A_*$, the parametrizing matrix of the (approximate) entire extension of the spectrogram $\spect f$ (step $1$ of Algorithm \ref{alg1}).
In this step, one needs to solve a SDP where the variable is a matrix of dimension $|\Gamma| \times |\Gamma|$.
Hence, a moderate enlargement of the set $\Lambda$ results in substantial extra computational expenses. Comparing the two present examples we register roughly a quadruplication of computation time.

\begin{remark}[Infeasibility issues]
The experimental results displayed in this section are very much a product of a trial and error process regarding the selection of suitable parameters.
For several instances the algorithm would terminate early as either of the convex problems handed to the solver is found to be infeasible. From a practical viewpoint it would be highly desirable to have a variant of the scheme which is robust w.r.t. such infeasibility issues.
For example, a viable pathway could be to incorporate the constraints into the objective function instead of using hard constraints.
As the focal point of this paper lies on the numerical analysis of the scheme we do not aim to further expand on this topic at this stage.
\end{remark}
\section{Terminology and Prerequisites}\label{sec:prelims}
\subsection{General notation}
Throughout we will understand $\R^d$ as a subspace of $\C^d$.
Given a vector $p\in\C^d$ its euclidean length is denoted by $|p|$ and we use the notation $p^2=p\cdot p = \sum_{k=1}^d p_k^2$. Real and imaginary parts are taken componentwise, i.e., 
$$
\Re(p)=(\Re(p_1),\ldots,\Re(p_d))^T,\quad\Im(p)=(\Im(p_1),\ldots,\Im(p_d))^T.
$$
Given a function $F:\R^d\to \C$, we use the tensor notation
$$\Tcal_u[F](p) = F(p+u) \overline{F(p)},\quad p,u\in\R^d.$$

\smallskip
Given a matrix $A$ we denote
$$
\|A\|_{\max} = \max_{k,\ell} |A_{k,\ell}|.
$$
If $\Lambda$ is a finite index set with $N=|\Lambda|$ its cardinality, we can identify 
$\C^{\Lambda\times\Lambda}$ with $\C^{N\times N}$, the vector space of $N\times N$ square matrices.
Linear algebra concepts such as matrix products, matrix-vector multiplication, trace, positive (semi)-definiteness, eigenvalues and eigenvectors are well-defined in $\C^{\Lambda\times\Lambda}$ and will be denoted by the common notation.
Moreover, we introduce the Frobenius inner product on $\C^{\Lambda\times\Lambda}$ by 
$$
\langle X,Y\rangle_F =\sum_{\lambda,\lambda'\in\Lambda} X_{\lambda,\lambda'} \overline{Y_{\lambda,\lambda'}} = \trace (Y^H X).
$$
Recall that $\langle X, Y \rangle_F$ is real-valued (resp. non-negative) if $X,Y$ are hermitian (resp. positive definite).
Given vectors $a,b\in\C^\Lambda$, we use tensor notation $a\otimes b$ to denote the outer vector product, that is,
$$
(a \otimes b)_{\lambda,{\lambda'}} = a_\lambda b_{\lambda'}, \quad \lambda,\lambda'\in\Lambda.
$$
The following simple lemma will be handy:

\begin{lemma}\label{lem:matrixcomplsdp}
    Let $k,\ell\in \{1,\ldots,N\}$.
    and let $e_k, e_\ell$ denote the $k$-th and $\ell$-th canonical basis vector in $\C^N$.
    Then, for every $X\in\C^{N\times N}$ hermitian, 
    $$
    \Re(X_{k,\ell}) = \langle X, \frac12 (e_k\otimes e_\ell + e_\ell \otimes e_k)\rangle_F,
    $$
    and
    $$\Im (X_{k,\ell}) = \langle X, \frac1{2i}(e_\ell \otimes e_k - e_k\otimes e_\ell)\rangle_F.
    $$
\end{lemma}

\begin{proof}
    Using that $X$ is hermitian implies that 
    \begin{multline}
        \Re(X_{k,\ell}) = \frac12 (X_{k,\ell} + \overline{X_{k,\ell}}) = \frac12 ( X_{k,\ell} + X_{\ell,k}) 
        = \frac12 (e_k^T X e_\ell + e_\ell^T X e_k) \\
        = \frac{1}{2}\bigl( \trace([e_\ell \otimes e_k] X) + \trace ([e_k \otimes e_\ell]X)\bigr)
    \end{multline}
    The second identity is similar.
\end{proof}
\smallskip

For $F\in\mathcal{O}(\C)$ an entire function we will employ the notation $F^*$ to denote the function defined by
$F^*(z)= \overline{F(\bar{z})}$. Note that $F^*\in \mathcal{O}(\C)$.
Given $G\in\mathcal{O}(\C^d)$, we define the maximum modulus function by
$$
M_G(r) := \sup_{\substack{z\in\C^d\\ |\Im z|=r}} |G(z)|, \quad r\ge 0.
$$

The Fourier transform of $f\in L^1(\R^d)$ is defined via the integral
$$
\ft f(\omega) = \hat{f}(\omega) = \int\limits_{\R^d} f(t) e^{-2\pi i \omega\cdot t}\,\mbox{d}t,
$$
and extended to $L^2(\R^d)$ in the usual way.
We denote the shift and the modulation operator by $T$ and $M$, respectively. That is
$$
T_x f(t)= f(t-x), \quad M_\omega f(t) = e^{2\pi i \omega\cdot t}f(t), \quad f\in L^2(\R^d).
$$
Time-frequency shifts are denoted by $\pi(z)f=M_\omega T_x f$, where $z=(x,\omega)$.

\smallskip
The Jacobi $\vartheta_3$ function which is defined by 
$$
\vartheta_3(z,q) = \sum_{k\in\Z} q^{k^2}e^{2ikz}, \quad q\in (0,1), \, z\in\C
$$
will be useful on a number of occasions. 
All values of this function used here have been computed using
the {\tt  EllipticTheta[3,z,q]} function on {\tt www.wolframalpha.com}, in particular, we will use
$$
\vartheta_3(0,e^{-\pi})
=1.08643...\quad,\quad
\vartheta_3(0,e^{-\pi/2})=1.41949...
$$
and $\vartheta_3(0,e^{-\pi/4})^2=4.00005...$
One application is estimating the sum of equidistant shifts of Gaussians.

\begin{lemma}\label{lem:sumgaussshifts}
    Let $b>0$. For all $t\in\R$ it holds that 
    $$
    \sum_{k\in \Z} e^{-\frac\pi{b}(t-k)^2} \le \sqrt{b} \vartheta_3(0,e^{-b\pi}).
    $$
\end{lemma}
\begin{proof}
    We denote $g(x) = e^{-\frac{\pi}{b} x^2}$ and note that $\hat{g}(\omega) = \sqrt{b} e^{-b\pi\omega^2}$.
    Applying Poisson summation formula allows us to rewrite
$$
\sum_{k\in\Z} e^{-\frac\pi{b}(t-k)^2} = \sum_{k\in\Z} T_t g(k) 
= \sum_{k\in\Z} M_{-t} \widehat{g}(k) 
= \sqrt{b} \sum_{k\in\Z} e^{-2\pi i tk } e^{-b\pi k^2}.
$$
The triangle inequality then gives
$$
\sum_{k\in\Z} e^{-\frac\pi{b}(t-k)^2}
        \le \sqrt{b}\sum_{k\in\Z} e^{-b\pi k^2} =  \sqrt{b} \vartheta_3(0, e^{-b\pi}),
$$
as claimed.
\end{proof}

\subsection{Time frequency analysis}
The central object of this paper is the following.
\begin{definition}[Short-time Fourier transform]
The \emph{short-time Fourier transform}(STFT) of $f\in L^2(\R^d)$ with window function $g\in L^2(\R^d)$
is defined by
$$
\V_g f(x,\omega) = \langle f, M_\omega T_x g\rangle_{L^2(\R^d)}.
$$
Moreover, if 
$$
g=\varphi:= 2^{d/4}e^{-\pi \cdot^2},
$$
the Gaussian, we use the notation $\G f = \V_\varphi f$ and call $\G f$ the Gabor transform.
The squared modulus of the Gabor transform is called the \emph{spectrogram} and will be denoted by $\spect f$, that is, $\spect f(x,\omega)=|\G f(x,\omega)|^2$.
\end{definition}
We devote the remainder of this section to collect a couple of properties and statements involving the STFT and refer to Gröchenig's textbook \cite{gro:foundtfanalysis} for more details.\\
Note that, from Cauchy-Schwarz, we get $\|\spect f\|_{L^\infty}\leq\|f\|_{L^2}^2$.
Given $\lambda=(a,b)$,
$$
\G [\pi(\lambda)f](x,\omega) =e^{-2\pi i a (\omega-b)} \cdot  \G f (x-a,\omega-b),
$$
in particular, $|\G [\pi(\lambda)f](p)| = |\G f(p-\lambda)|$.
The Gabor transform of the Gaussian $\varphi $ is given by 
\begin{equation}\label{eq:Gphi}
\G\varphi(x,\omega) = e^{-\frac\pi2 (x^2+\omega^2+2ix\omega)}
\end{equation}
so that, for $z=(x,\omega)$,
\begin{equation}\label{eq:Gphitrans}
\G [\pi(\lambda)\varphi](x,\omega)=
e^{-\pi i(x+a)(\omega-b)} e^{-\frac\pi2 |z-\lambda|^2}.
\end{equation}

Later on, we will require the formula
\begin{equation}\label{eq:ftofprodstft}
\ft \left(\V_g f \overline{\V_\gamma h}\right) (s,t) = \V_h f(-t,s) \overline{\V_\gamma g(-t,s)}, \quad f,g,h,\gamma\in L^2(\R^d).
\end{equation}
which may be found in \cite{jamingcras} and in \cite{groech:mystery}.

The Bargmann transform denoted by $\mathcal{B}$ is defined as 
$$
\mathcal{B}f(z) = 2^{1/4} \int_\R f(t) e^{2\pi t z-\pi t^2- \frac{\pi}2 z^2}\,\mbox{d}t.
$$
It is an isometry from $L^2(\R)$ to the Fock space of entire functions,
$$
\mathcal{F}^2=\{f\,:\C\to\C\mbox{ entire s.t. }\|f\|_{\ft^2}^2:=\int_{\C}|f(z)|^2e^{-\pi|z|^2}\,\mbox{d}z<+\infty\}.
$$
Bargmann and Gabor transforms are intimitely related by virtue of the identity 
\begin{equation}\label{eq:relgaborbargmann}
\G f(x,-y) = e^{\pi i xy} \mathcal{B}f(z) e^{-\frac\pi2 |z|^2}, \quad z=x+iy\in\C.
\end{equation}

The Hermite functions are defined by 
$$
h_k(t) = c_k (-1)^k e^{\pi t^2} \left(\frac{\mbox{d}}{\mbox{d}t} \right)^k \left(e^{-2\pi t^2}\right), \quad k\in\N,
$$
where $c_k>0$ is chosen in such a way that $\|h_k\|_{L^2}=1$. 
Then, $(h_k)_{k\in\N}$ forms an orthonormal basis of $L^2(\R)$.
The Bargmann transform maps the Hermite functions to monomials of the corresponding degree, that is, 
\begin{equation}
    \label{eq:bargmanpol}
\B h_k(z) = \sqrt{\frac{\pi^k}{k!}} z^k, \quad k\in\N.
\end{equation}
Equation \eqref{eq:ftofprodstft} implies that Fourier transform of the spectrogram has Gaussian decay.
The following lemma reveals that this is still the case after applying a Gaussian cut-off to the spectrogram.
\begin{lemma}\label{lem:fourierdecayDexp}
    Let $f\in L^2(\R)$, then it holds for all $\tau\in\R^2$ that
    \begin{equation}\label{eq:estspectgexp}
    \big| \ft\big( \spect f \cdot e^{-\frac{\pi}{8}|\cdot-\tau|^2}\big)(\xi)\big| \le 8 \|\spect f\|_{L^\infty(\R^2)} \cdot e^{-\frac{38\pi}{81} \xi^2}, \quad \xi\in\R^2.
        \end{equation}
\end{lemma}

\begin{proof}
    Since we can always consider $\pi(-\tau)f$ instead of $f$ we may assume w.l.o.g. that $\tau=0$. First we write $f$
    in the Hermite basis $f=\sum_{k\geq 0}\langle f,h_k\rangle h_k$ so that
    \begin{eqnarray*}
        \spect f(z)&=&\left|\B f(z)e^{-\frac{\pi}{2}|z|^2}\right|^2=\left|\sum_{k\geq 0}\langle f,h_k\rangle \B h_k e^{-\frac{\pi}{2}|z|^2}\right|^2\\
    &=&\left|\sum_{k\geq 0}\langle f,h_k\rangle \sqrt{\frac{\pi^k}{k!}}z^ke^{-\frac{\pi}{2}|z|^2}\right|^2
    \end{eqnarray*}
    with \eqref{eq:bargmanpol}.
    With $q=\frac{2\sqrt2}3<1$, let us denote
    $ 
    \gamma = \sum_{k\ge 0} q^k \langle f,h_k\rangle h_k
    $
    and rewrite 
    \begin{align}\label{eq:spectfgamma}
        \spect f(z)  e^{-\frac\pi8 |z|^2} &= \left|\sum_{k=0}^\infty \langle f,h_k\rangle \sqrt{\frac{\pi^k}{k!}} z^k e^{-\frac{9\pi}{16} |z|^2} \right|^2\\
        &= \left|\sum_{k=0}^\infty \langle f,h_k\rangle \sqrt{\frac{\pi^k}{k!}} z^k e^{-\frac{\pi}2 |z/q|^2} \right|^2\\
        &= \left|\sum_{k=0}^\infty q^k \langle f,h_k\rangle \sqrt{\frac{\pi^k}{k!}} (z/q)^k e^{-\frac{\pi}2 |z/q|^2} \right|^2 = \spect \gamma(z/q).
    \end{align}
    Thus, we get with $\xi'=(-\xi_2,\xi_1)$ that the Fourier transform of the product is equal to
    \begin{equation}
       \ft\big( \spect f\cdot e^{-\frac\pi8 |\cdot|^2}\big) (\xi)
       = q^2 \widehat{\spect \gamma}(q\xi) = q^2 \cdot \V_\gamma \gamma(q\xi') \cdot \V_\varphi\varphi(q\xi'),
    \end{equation}
    where we made use of \eqref{eq:ftofprodstft} to derive the last equality.
    Since $\V_\varphi\varphi$ is a Gaussian (see  \eqref{eq:Gphi}) it follows that 
    \begin{equation}\label{eq:spectfgamma_est1}
    |\ft\big( \spect f\cdot e^{-\frac\pi8 |\cdot|^2}\big) (\xi)| \le |\V_\gamma \gamma(q\xi')| \cdot \frac89 e^{-\frac{4\pi}9 |\xi|^2}.
    \end{equation}
    It remains to establish a pointwise bound for $\V_\gamma \gamma$.
    Note that it follows from \eqref{eq:spectfgamma} (replace $z$ by $qz$ and take roots on both sides) that 
    $$
    |\G \gamma(z)| = |\G f(qz)| e^{-\frac\pi{18}|z|^2}.  
    $$
    Since the Gabor transform is unitary we can use this to estimate for $p\in\R^2$
    \begin{multline}
        |\V_\gamma \gamma(p)|  = |\langle \gamma, \pi(p)\gamma\rangle| = 
        |\langle \G \gamma, \G [\pi(p)\gamma]\rangle_{L^2(\R^2)}| \\
        \le \int_{\R^2} |\G \gamma(z)| \cdot |\G \gamma(z-p)|\,\mbox{d}z
        \le \|\spect f\|_{L^\infty} \int_{\R^2} e^{-\frac\pi{18} (|z|^2+|z-p|^2)}\,\mbox{d}z.
    \end{multline}
    Since the integral expression 
    $$
    \rho(p):= \int_{\R^2} e^{-\frac\pi{18} (|z|^2+|z-p|^2)}\,\mbox{d}z
    $$
    only depends on $|p|$ we can assume that $p=(|p|,0)^T$ and compute
    $$
    \rho(p) = \left(\int_\R e^{-\frac\pi{18}(x^2+(x-|p|)^2} \right) \cdot \left(\int_\R e^{-\frac\pi9 \omega^2}\,\mbox{d}\omega \right) 
    = 9e^{-\frac{\pi}{36}|p|^2}.
    $$
    In particular, we get that 
    $$
    |\V_\gamma \gamma(q\xi') | \le 9 \|\spect f\|_{L^\infty} e^{-\frac{2\pi}{81}|\xi|^2},
    $$
    which together with \eqref{eq:spectfgamma_est1} implies the claim.
\end{proof}

We will need the following localization property of Gabor transforms:

\begin{lemma}\label{lem:gabordisccont}
    It holds for all $r>0$, $z\in\R^2$ and $f\in L^2(\R)$  that 
    $$
    |\G f(z)| \le \frac{\sqrt{e^{\pi r^2}-1}}{\pi r^2} \|\G f\|_{L^2(B_r(z))}.
    $$
\end{lemma}

\begin{proof}
    We may assume w.l.o.g. that $z=0$. Using the mean value property for the Bargmann transform we get that 
    \begin{align*}
    \G f(0) = \B f(0) &= \frac1{\pi r^2} \int\limits_{\{|z|<r\}} \B f(z) \,\mbox{d}z \\
    &= \frac1{\pi r^2} \iint\limits_{B_r(0)} \G f(x,-y) e^{\frac\pi2 (x^2+\omega^2)} e^{-\pi ixy} \,\mbox{d}x\mbox{d}y.
    \end{align*}
    Cauchy-Schwarz implies that 
    $$
    |\G f(0)| \le \frac1{\pi r^2} \|\G f\|_{L^2(B_r(0))} \cdot \sqrt{e^{\pi r^2}-1},
    $$
    and we are done.
\end{proof}

\subsection{Gaussian Gabor frames}\label{sec:gaussgabor}
This paragraph is concerned with discreti\-zation properties of the STFT.
We refer to textbooks of Gröchenig \cite{gro:foundtfanalysis} and Christensen \cite{christensen:intro} for a comprehensive account of this topic including the collection of statements below.\\

A family of vectors $(f_i)_{i\in I}\subseteq \mathcal{H}$ with $\mathcal{H}$ a Hilbert space is called a \emph{frame} if there exist constants $0<A\le B$ such that 
$$
A \|f\|_{\mathcal{H}}^2 \le \sum_{i\in I} |\langle f,f_i\rangle_{\mathcal{H}}|^2 \le B\|f\|_{\mathcal{H}}^2, \quad f\in \mathcal{H}.
$$
The maximal $A$ and the minimal $B$ to satisfy the above condition are called frame constants.
Given a frame $(f_i)_{i\in I}$ the frame operator defined by 
$$
S: \mathcal{H}\rightarrow \mathcal{H}, \, Sf=\sum_{i\in I} \langle f,f_i\rangle_{\mathcal{H}} f_i
$$
is well-defined, bounded, invertible, self-adjoint and positive. Moreover, the family 
$(S^{-1}f_i)_{i\in I}$ forms a frame with frame constants $B^{-1}, A^{-1}$. 
The frame $(S^{-1}f_i)_{i\in I}$ is called the dual frame.\\

We consider function systems of the form $\Phi=(\pi(\lambda)\varphi)_{\lambda\in \Lambda_{a,b}}$, with $\varphi$ the Gaussian, $\varphi(x)=2^{1/4}e^{-\pi x^2}$,
$\Lambda_{a,b}=a\Z\times b\Z$ a lattice, $a,b>0$ discretization parameters.
In this case, $\Phi$ forms a frame for $L^2(\R)$ if and only if $ab<1$.
By the definition of the Gabor transform the frame inequality can be written in the following way
$$
A\|f\|_{L^2(\R)}^2 \le \sum_{\lambda\in\Lambda} |\G f(\lambda)|^2 \le B\|f\|_{L^2(\R)}^2, \quad f\in L^2(\R). 
$$
The canonical dual of the frame $\Phi$ has Gabor structure as well and is given by $(\pi(\lambda)\gamma)_{\lambda\in\Lambda_{a,b}}$ with 
$\gamma = S^{-1}\varphi$ and  the reconstruction identity
\begin{equation}\label{eq:dualwindowrec}
f = \sum_{\lambda\in \Lambda} \langle f, \pi(\lambda) \varphi\rangle_{L^2(\R)} \pi(\lambda) \gamma
= \sum_{\lambda\in\Lambda} \G f(\lambda) \pi(\lambda) \gamma
, \quad f\in L^2(\R),
\end{equation}
holds true, 
where the sum converges unconditionally in $L^2(\R)$.

For the case of integer redundancy -- which means that the parameters of the underlying lattice $a\Z\times b\Z$ satisfy that $(ab)^{-1} \in \{2,3,4,\ldots\}$ -- Janssen computed several Gabor frame related quantities explicitly \cite[section 6]{janssen:someweyl}.
In the sequel, we will only consider the case 
$$
a=b=\mathfrak{a}=\dfrac{1}{\sqrt{2}},\qquad\Lambda=\Lambda_{\mathfrak{a}}=\dfrac{1}{\sqrt{2}}\Z^2
$$
which falls in this setting.
In principle, other configurations could also be used in our approach. However, the above is somehow the canonical choice in the sense that $\Lambda_\mathfrak{a}$ is the sparsest possible square lattice within the setting of Janssen's results.

For the present situation, Janssens's results imply the following.
The lower and upper frame bound of $(\pi(\lambda)\varphi)_{\lambda\in\Lambda}$ are given by
\begin{align*}
A &= 2 \left(\sum_{k\in\Z}(-1)^k e^{-\pi k^2} \right)^2 = 2 \vartheta_3(\frac\pi2,e^{-\pi})^2 = 1.66\ldots\\
B &= 2 \left(\sum_{k\in\Z}  e^{-\pi k^2} \right)^2 =  2 \vartheta_3(0,e^{-\pi})^2 = 2.36\ldots
\end{align*}
respectively.
The dual window is given by the formula
\begin{equation}\label{eq:formulapsi}
\psi(t) = \frac1{2\vartheta_3(\sqrt2 \pi t, e^{-\pi})} \sum_{k\in\Z} c_k \varphi(t-\sqrt2 k)
\end{equation}
with the coefficients given by 
\begin{equation}\label{def:coeffsc_k}
    c_k =  
    \frac{\sum_{m=0}^\infty (-1)^{k+m} e^{-\pi(m+\frac12)(2|k|+m+\frac12)}}{\sum_{n=-\infty}^\infty (-1)^n (n+\frac12)e^{-\pi(n+\frac12)^2}}.
\end{equation}
We will require several explicit bounds on the dual window function.
\begin{lemma}\label{lem:bdspsi}
    The canonical dual $\psi$ satisfies the bounds 
    $$
    \|\psi\|_{L^2(\R)} \le 0.6 \quad \text{and}\quad |\psi(t)| \le e^{-\frac{\pi |t|}{\sqrt2}}, \, t\in\R
    $$
    while $\|\G\psi\|_{L^1(\R^2)}\leq  1.23$.
\end{lemma}

\begin{remark}
Recall that $\|\G\psi\|_{L^1(\R^2)}$ is the norm of $\psi$ in the so-called {\em Feichtinger algebra} $\mathcal{S}_0$ which is also the {\em modulation space} $M_1$.
The fact that $\|\G\psi\|_{L^1(\R^2)}<+\infty$ follows from \cite[Theorem 1.2]{GL}. We here compute an explicit bound.
\end{remark}

\begin{proof}
    Since the operator norm of the inverse of the frame operator $S^{-1}$ is bounded from above by the reciprocal of the lower frame bound, we get
    $$
    \|\psi\|_{L^2(\R)} = \|S^{-1}\varphi\|_{L^2(\R)}\le \frac1{A} \|\varphi\|_{L^2(\R)} = \frac1{A}<0.6.
    $$
    For the pointwise bound we first estimate for $x\in \R$
    \begin{multline}
    \vartheta_3(x,e^{-\pi}) = 1 + 2 \sum_{k=1}^\infty e^{-\pi k^2} \cos(2kx) \ge 1 - 2\sum_{k=1}^\infty e^{-\pi k^2}\\
    =2 - \sum_{k\in\Z} e^{-\pi k^2} = 2 - \vartheta_3(0,e^{-\pi}) = 0.913\ldots
    \end{multline}
    With this we get that the first factor in the formula for $\psi$ is bounded by
    $$
    \frac{1}{2\vartheta_3\big(\sqrt2 \pi t,e^{-\pi} \big)} < 0.55
    $$
    Next, we deal with the sum in the denominator of $c_k$.
    Note that if $(a_n)_{n\in\N}$ is a non-increasing sequence of positive numbers, then 
    $\sum_{n=0}^\infty (-1)^n a_n \ge a_0- a_1$.
    Since $n\mapsto (n+\frac12)e^{-\pi(n+\frac12)^2}$ is monotonically decreasing for $n\ge 0$ we get a lower bound for the denominator in \eqref{def:coeffsc_k}:
    \begin{align*}
        \kappa :=\sum\limits_{n=-\infty}^\infty (-1)^n (n+\frac12) e^{-\pi(n+\frac12)^2}
            &= 2 \sum\limits_{n=0}^\infty (-1)^n (n+\frac12) e^{-\pi(n+\frac12)^2}\\
            &\ge 2 \left( \frac{e^{-\frac\pi4}}2 - \frac{3 e^{-\frac{9\pi}4}}2 \right) 
            >0.45.
    \end{align*}
    Let us denote the numerator of $c_k$ by
    $$
    \gamma_k := (-1)^k \sum_{m=0}^\infty (-1)^m e^{-\pi(m+\frac12)(2|k|+m+\frac12)}.
    $$
    A similar argument as above implies that 
    \begin{equation}
        0< (-1)^k \gamma_k = \sum_{m=0}^\infty (-1)^m e^{-\pi (m+\frac12)(2|k|+m+\frac12)} \le e^{-\pi(|k|+\frac14)}.
    \end{equation}
    With this we get that 
    \begin{equation}
        |\psi(t)| \le 0.55 \cdot \sum_{k\in\Z} |c_k| \varphi(t-\sqrt2 k) 
        \le \frac{0.55}{0.45} e^{-\frac\pi4} 2^{\frac14} \cdot  \sum_{k\in\Z} e^{-\pi|k|} e^{-\pi (t-\sqrt2 k)^2}.
    \end{equation}
    The product of the constants in front of the sum is bounded above by $0.67$.
    Clearly, the sum is even with respect to $t$. 
   Therefore, we can replace $t$ by $|t|$  and further estimate
    \begin{align*}
        \sum_{k\in\Z} e^{-\pi|k|} e^{-\pi (t-\sqrt2 k)^2} &= e^{-\frac{\pi|t|}{\sqrt2}} \cdot \sum_{k\in\Z} \exp \left\{ -\pi \left(|k|-\frac{|t|}{\sqrt2} + (|t|-\sqrt2 k)^2 \right)\right\}\\
        &\le  e^{-\frac{\pi|t|}{\sqrt2}}  \cdot \sum_{k\in\Z} \exp \left\{ -\pi \left(k-\frac{|t|}{\sqrt2} + (|t|-\sqrt2 k)^2 \right)\right\}\\
        &= e^{-\frac{\pi|t|}{\sqrt2}} \cdot \sum_{k\in\Z} e^{-\pi p\left(\frac{|t|}{\sqrt2}-k \right)},
    \end{align*}
    where $p(x) = 2x^2-x= 2 \left(x-\frac14 \right)^2 -\frac18$.
    The function $\sum_{k\in\Z} e^{-\pi p(\cdot-k)}$ coincides up to a translation with 
    $$g(x)=e^{\frac\pi8}\sum_{k\in\Z} e^{-2\pi(k-x)^2}.$$
    Using Poisson summation formula we get that
    \begin{equation}
         g(x) = \frac{e^{\frac\pi8}}{\sqrt2} \sum_{\ell\in\Z} e^{-2\pi i \ell x} e^{-\frac\pi2 \ell^2} \le 
         \frac{e^{\frac\pi8}}{\sqrt2} \sum_{\ell\in\Z}e^{-\frac\pi2 \ell^2} =  \frac{e^{\frac\pi8}}{\sqrt2} \vartheta_3(0, e^{-\frac\pi2}) = 1.48\dots
    \end{equation}
    Since $1.49\cdot 0.67 <1$ we arrive at the desired bound.

    \smallskip

    It remains to prove the $L^1$-bound of $\G\psi$. First note that, as 
    $$|\G[\pi(\lambda)\varphi]|=e^{-\frac{\pi}{2} |\cdot - \lambda|^2},$$
    we get
    $\|\G[\pi(\lambda)\varphi]\|_{L^1(\R^2)}=\|\G\varphi\|_{L^1(\R^2)}=2$.
    It follows from \cite[formulas (6.10) and (6.11)]{janssen:someweyl}
    that 
    \begin{equation}
        \frac1{\vartheta_3(\sqrt{2}\pi t, e^{-\pi})} = \frac1{\kappa} \sum_{\ell=-\infty}^\infty \gamma_\ell e^{2 \sqrt2 \pi i \ell t}.
    \end{equation}
    Thus, we have that 
        \begin{equation}
            \psi(t) = \frac1{\vartheta_3(\sqrt2 \pi t, e^{-\pi})} \sum_{k\in\Z} c_k \varphi(t-\sqrt2 k)
            = \frac1{2\nu^2} \sum_{\ell\in\Z} \sum_{k\in\Z} \gamma_\ell \gamma_k \,\pi 
            \binom{\sqrt2 k}{2^{3/2}\ell}\varphi
        \end{equation}
    and applying $\G$ gives 
    \begin{equation}
        \G \psi = \frac1{2\nu^2} \sum_{\ell\in\Z} \sum_{k\in\Z} \gamma_\ell \gamma_k \, \G \left[\pi 
            \binom{\sqrt2 k}{2^{3/2}\ell}\varphi \right].
    \end{equation}
    Making use of the above estimates we get that
    \begin{equation}
        \|\G \psi\|_{L^1(\R^2)} \le \frac1{2\cdot 0.45^2} 
        \sum_{\ell\in\Z} \sum_{k\in \Z} 2 e^{-\pi(|k|+|\ell|+\frac12)}
    =  \frac{e^{-\frac\pi2}}{0.45^2}
    \cdot \left(1 + \frac{2 e^{-\pi}}{1-e^{-\pi}} \right)^2
    \end{equation}
    which is $<1.23$, and we are done.
\end{proof}

Finally, we will also need two simple lemmas. The first one deals with Bessel-bounds
for the systems $\{\pi(\lambda)\psi\}_{\lambda\in\Lambda}$ and $\{\pi(\lambda)\varphi\}_{\lambda\in\Lambda}$. 

\begin{lemma}\label{lem:rieszbdpsi}
    For all $c=(c_\lambda)_{\lambda\in\Lambda} \in \ell^2(\Lambda)$ it holds that 
    $$
        \left\|\sum_{\lambda\in\Lambda} c_\lambda \pi(\lambda) \varphi\right\|_{L^2(\R)} \le 4.25 \|c\|_{\ell^2(\Lambda)}
 \ \mbox{and}\ 
    \left\|\sum_{\lambda\in\Lambda} c_\lambda \pi(\lambda)\psi\right\|_{L^2(\R)} \le 3.1 \|c\|_{\ell^2(\Lambda)}.
    $$
\end{lemma}

\begin{proof} As we will not use the first inequality, we will only give full detail for the second one.
We define the quantity
$$
B' := 2\sqrt2 \left( \sum_{k\in\Z} \|\psi\|_{L^\infty([k\mathfrak{a},(k+1)\mathfrak{a}))} \right)^2.
$$
With the help of Lemma \ref{lem:bdspsi} we get that the expression inside the brackets is bounded by
\begin{equation}
 \sum_{k=-\infty}^{-1} e^{-\frac{\pi}{\sqrt2} |(k+1)\mathfrak{a}|} + \sum_{k=0}^\infty e^{-\frac{\pi}{\sqrt2}|k\mathfrak{a}|}
= 2 \sum_{k=0}^\infty e^{-\frac\pi2 k} = \frac2{1-e^{-\frac\pi2}} = 2.52\dots,
\end{equation}
which implies that $B'\le 18.04$.
It follows from \cite[Proposition 11.5.2]{christensen:intro} that $B'$ is an upper frame bound for $(\pi(\lambda)\psi)_{\lambda\in\Lambda}$.
Moreover, \cite[Theorem 3.2.3]{christensen:intro} implies that 
$$
\|\sum_{\lambda\in\Lambda} c_\lambda \pi(\lambda)\psi\|_{L^2(\R)} \le \sqrt{B'} \|c\|_{\ell^2(\Lambda)},
$$
which finishes the proof of the second inequality.

To prove the first inequality, recall that the upper frame bound for the original frame $(\pi(\lambda)\varphi)_{\lambda\in\Lambda}$ is 
given by 
$$
B= 2 \vartheta_3(0,e^{-\pi})^2 < 1.54^2,
$$
so that we can proceed as above.
\end{proof}

Finally, we also have:

\begin{lemma}\label{lem:besselTg}
    Let $s>0$ and let $g(x):= e^{-2\pi x^2}$.
    Then it holds that $(T_{s\ell}g)_{\ell\in\Z}$ forms a Bessel sequence with 
    $$
    \sum_{\ell\in\Z} \big|\langle f,T_{s\ell}g\rangle \big|^2 \le \frac{\vartheta_3(0,e^{-\pi s^2})}2  \|f\|_{L^2}^2, \quad f\in L^2(\R).
    $$
\end{lemma}

\begin{proof}
    By  \cite[Proposition 3.5.4]{christensen:intro}, it suffices to show that for all $k\in\Z$ it holds that 
    $$
    \sum_{\ell\in \Z} |\langle T_{s\ell}g, T_{sk}g\rangle| \le \frac{\vartheta_3(0,e^{-\pi s^2})}2.
    $$
    Since for all $a,b\in\R$ we have that 
    $$
    \langle T_a g, T_b g\rangle = \langle g, T_{b-a} g\rangle = \int_\R e^{-2\pi t^2} e^{-2\pi (t-b+a)^2}\,\mbox{d}t = \frac12 e^{-\pi(b-a)^2},
    $$
    we indeed get that 
    \begin{equation}
        \sum_{\ell\in \Z} |\langle T_{s\ell}g, T_{sk}g\rangle| =\frac12  \sum_{\ell\in \Z} e^{-\pi s^2(k-\ell)^2} = \frac12  \sum_{\ell\in \Z} e^{-\pi s^2 \ell^2}
        =\frac{\vartheta_3(0,e^{-\pi s^2})}2
    \end{equation}
    as claimed.
\end{proof}

\subsection{Oversampling formula}
We denote the space of functions of bandwidth (at most) $b>0$ by $PW_b(\R^d)$, that is,
$$
PW_b(\R^d) = \Big\{f\in L^2(\R^d): \supp(\hat{f})\subseteq [-\frac{b}2,\frac{b}2]^d \Big\}.
$$
The cardinal sine is defined by 
$$
\sinc(x) =\begin{cases}
    \frac{\sin(\pi x)}{\pi x}, & x\neq 0\\
    1, & x=0.
\end{cases}
$$
For several variables the $\sinc$ function is defined by tensorization, i.e., 
$$
\sinc(x) = \sinc(x_1)\cdot \ldots \cdot \sinc(x_d),\quad x=(x_1,\ldots,x_d)\in \R^d.
$$
The famous sampling theorem attributed to Shannon, Whittaker and Kotelnikov asserts that 
a bandlimited function can be reconstructed given samples on the lattice $\frac1{b}\Z^d$ by virtue of the formula 
\begin{equation}
    f(x) = \sum_{k\in \Z^d} f\left( \frac{k}{b}\right) \sinc\left(b\left(x-\frac{k}{b}\right) \right),\quad f\in PW_b(\R^d).
\end{equation}
The series converges rather slowly due to the slow decay of the $\sinc$, which can be a disadvantage especially for numerical purposes.
By oversampling the function one can get an improvement on the convergence.
\begin{lemma}\label{lem:oversampling}
    Let $b>0$ and $0<h<1/b$. 
    Moreover, let $\eta\in L^1(\R^d)$ be a non-negative function such that 
    $$
    \supp(\eta) \subseteq [-1,1]^d \quad \text{and}\quad \int_{\R^d}\eta(\xi)\,\mbox{d}\xi = 1.
    $$
    Then it holds for all $f\in PW_b(\R^d)$ that 
    $$
    f(x) = \left(\frac{1+hb}2\right)^d \cdot \sum_{k\in\Z^d} f(hk) \widehat{\eta}\left(q(hk-x)\right) \sinc\left(2c(hk-x) \right)
    $$
    where $c=\frac14(\frac1{h}+b)$ and $q=\frac14(\frac{1}{h}-b)$.
\end{lemma}

We prove this lemma, both for the sake of completeness, and since some of its ingredients are useful in the
proof of the next lemma.

\begin{proof}
Since $\supp(\hat{f}) \subseteq [-\frac{b}2,\frac{b}2]^d\subseteq [-\frac1{2h},\frac1{2h}]^d$ we have that
$$
\hat{f}(\xi) = h^d \sum_{k\in\Z} f(hk) e^{-2\pi ihk\xi}.
$$
Suppose that $\psi\in C(\R^d)$ is such that $\psi(\xi)=1$ for $\xi\in [-\frac{b}2,\frac{b}2]^d$ and 
$\psi(\xi)=0$ for $\xi\notin [-\frac1{2h},\frac1{2h}]^d$, then it further holds that 
\begin{equation}\label{eq:expfhatxi}
\hat{f}(\xi) = \hat{f}(\xi)\cdot \psi(\xi) = h^d \sum_{k\in\Z^d} f(hk) e^{-2\pi ihk\xi}\psi(\xi), \quad \xi \in\R^d
\end{equation}
where we used that $\hat{f}$ is supported in $[-\frac{b}2,\frac{b}2]^d$.

Next, we specify $\psi$. With $\eta_q=\frac1{q^d} \eta(x/q)$ we define
$$
\psi(\xi):=(\eta_q \ast \chi_{[-c,c]^d})(\xi) = \int_{[-c,c]^d} \eta_q(t-\xi)\,\mbox{d}t,
$$
and observe that 
\begin{itemize}[--]
\item $\psi\in C(\R^d)$ and non-negative,
\item since $\supp(\eta_q)\subseteq[-q,q]^d$ and since $q+c\le \frac{1}{2h}$ we have that 
$$
\supp(\psi) \subseteq [-q,q]^d + [-c,c]^d = [-(q+c),q+c]^d \subseteq [-\frac{1}{2h},\frac1{2h}]^d;
$$
in particular, it holds that $\psi(\xi)=0$ for $\xi \notin [-\frac{1}{2h},\frac1{2h}]^d$.
\item since $c = \frac{b}2+q$ it holds for all 
 $\xi\in [-\frac{b}2,\frac{b}2]^d$ that
$$[-c,c]^d+\xi\supseteq [-q,q]^d \supseteq \supp(\eta_q), $$ which implies that 
$$
\psi(\xi)= \int_{[-c,c]^d} \eta_q(t-\xi) \,\mbox{d}t =\int_{[-c,c]^d+\xi} \eta_q(t) \,\mbox{d}t = \int_{\R^d}\eta_q(t)\,\mbox{d}t  = 1
$$ for all $\xi\in [-\frac{b}2,\frac{b}2]^d$.
\item the Fourier transform of $\psi$ is given by 
$$
\widehat{\psi}(x) = \widehat{\eta_q}(x) \widehat{\chi_{[-c,c]^d}}(x) = \widehat{\eta}(qx) \cdot (2c)^d\sinc(2cx).
$$
\end{itemize}
We have therefore verified the conditions from above, and applying the inverse Fourier transform to \eqref{eq:expfhatxi} gives that 
\begin{multline}
    f(x) = h^d \sum_{k\in\Z^d} f(hk) \left(\int_{\R^d} \psi(\xi)e^{2\pi i(x-hk)\xi} \,\mbox{d}\xi \right)
    = h^d \sum_{k\in \Z^d} f(hk) \widehat{\psi}(hk-x) \\
    = (2c h)^d \cdot \sum_{k\in\Z^d} f(hk) \widehat{\eta}\left(q(hk-x)\right) \sinc\left(2c(hk-x) \right).
\end{multline}
The statement follows now from $2ch=\frac12(1+hb)$.
\end{proof}

We will also need an approximate sampling result for almost bandlimited functions.

\begin{lemma}\label{lem:samplinginf}
    Let $s>0$, let $F\in L^2(\R^2)$ and suppose that there exist $K,C>0$ such that 
    $    |\widehat{F}(\xi)| \le K e^{-C|\xi|^2}$, $\xi\in\R^2$.
    Then it holds that 
    $$
    \|F\|_{L^\infty(\R^2)} \le 73 \cdot \|F\|_{\ell^\infty(s\Z^2)} + 233 \cdot \frac{K}{C} e^{-\frac{C}{16s^2}}.
    $$
\end{lemma}

\begin{proof}
    With $b=\frac1{2s}$, we introduce the function 
    $$
    F_b := \ft^{-1}(\widehat{F}\cdot \chi_{[-\frac{b}2,\frac{b}2]^2}),
    $$
    which is the projection of $F$ onto $PW_b(\R^2)$.
    Furthermore, let $h=\frac1{2b}=s$ and 
    $
    \eta=\chi_{[-\frac12,\frac12]^2} \ast \chi_{[-\frac12,\frac12]^2}.
    $ 
    Notice that $\eta\in C(\R^2)$ is non-negative and $\displaystyle\int_{\R^2} \eta(x)\mbox{d}x=1$.
    We apply Lemma \ref{lem:oversampling} to $F_b$: 
    By taking into account that 
    $$
    c=\frac14\left(\frac1{h}+b\right) = \frac34 b \quad \text{and}\quad q = \frac14\left(\frac1{h}-b\right) = \frac14 b
    $$
    it follows -- since $1+hb=\frac32$ -- that
    $$
    F_b(x) = \frac9{16} \sum_{k\in\Z^2} F_b \left(\frac{k}{2b} \right) \cdot \widehat{\eta}\left(\frac{k}8-\frac{bx}4 \right) \cdot \sinc\left(\frac34 k-\frac{3bx}{2} \right).
    $$
    Since $\widehat{\eta}(\xi) =\sinc^2(\xi)$ and $s=\frac1{2b}$ we can further rewrite the above identity as 
    \begin{align*}
    F_b(x) &= \frac9{16} \sum_{k\in\Z^2} F_b \left(sk \right) \cdot \sinc^2 \left(\frac{x-sk}{4s}\right) \cdot \sinc\left( \frac{3(x-sk)}{4s}\right)\\
    &= \sum_{k\in\Z^2} F_b(sk) \rho(x-sk),
    \end{align*}
    where $\rho(x):=\frac{9}{16}  \cdot \sinc^2\left( \frac{x}{4s}\right) \cdot \sinc\left(\frac{3x}{4s}\right).$
    We apply the triangle inequality and get 
    $$
        \|F\|_{L^\infty} \le \|F_b\|_{L^\infty} + \|F-F_b\|_{L^\infty}.
    $$
    We make use of the assumption on the Fourier decay and get that the approximation error is bounded by
    \begin{multline}
        \|F-F_b\|_{L^\infty} \le \|\ft(F-F_b)\|_{L^1} = \int_{\R^2\setminus[-\frac{b}2,\frac{b}2]^2} |\widehat{F}(\xi)|\,\mbox{d}\xi \\
        \le K \int_{\R^2\setminus[-\frac{b}2,\frac{b}2]^2} e^{-C|\xi|^2} \,\mbox{d}\xi 
        \le K \int_{\frac{b}2}^\infty e^{-C r^2} 2\pi r\,\mbox{d}r = \frac{\pi K}{C} e^{-\frac{C b^2}4}.
    \end{multline}
    Thus, we can bound
    \begin{multline}
        \|F_b\|_{L^\infty} \le \|\sum_{k\in\Z^2} F(sk)\rho(\cdot-sk)\|_{L^\infty} + \|\sum_{k\in\Z^2} [F(sk)-F_b(sk)] \rho(\cdot-sk)\|_{L^\infty}\\
        \le \left(\|F\|_{\ell^\infty(s\Z^2)} +  \frac{\pi K}{C} e^{-\frac{C b^2}4} \right) \cdot \|\sum_{k\in\Z^2} T_{sk}|\rho|\|_{L^\infty}.
    \end{multline}
    It remains to bound the last term involving $\rho$. To do this note that the univariate $\sinc$ satisfies the pointwise bound $|\sinc(x)| \le \frac{1.2}{1+|x|}$.
    As a consequence, 
    \begin{equation}
    |\rho(x)| \le \frac{9}{16} \cdot \frac{1.2^3}{\left(1+\frac{|x_1|}{4s}\right)^3} \cdot \frac{1.2^3}{\left(1+\frac{|x_2|}{4s}\right)^3} 
    \le \frac{(4.4\cdot s)^6}{(4s+|x_1|)^3 (4s+|x_2|)^3}.
    \end{equation}
    This implies that
    \begin{equation}
        \sup_{x\in\R^2} \left(\sum_{k\in\Z^2} |\rho(x-sk)|\right) \le (4.4\cdot s)^6 \cdot \sup_{x\in\R} \left( \sum_{k\in\Z} \frac1{(4s+|x-sk|)^3}\right)^2.
    \end{equation}
    By periodicity the function inside the brackets attains its maximum in the interval $x\in[0,s]$.
    For such $x$ we have that 
    \begin{eqnarray*}
      \sum_{k\in\Z} \frac1{(4s+|x-sk|)^3} &\le& \frac1{4^3s^3} +2 \sum_{k=1}^\infty \frac1{(3+k)^3 s^3}\\
    &=& \frac1{s^3} \left(4^{-3} + 2\sum_{k=4}^\infty k^{-3} \right) < 0.1 \cdot s^{-3}.
    \end{eqnarray*}
    With this we obtain that 
    $$
    \left\|\sum_{k\in\Z^2} T_{sk}|\rho|\right\|_{L^\infty} \le 4.4^6 \cdot 0.01 <73.
    $$
    Finally, collecting the estimates leads to 
    $$
    \|F\|_{L^\infty} \le 73 \|F\|_{\ell^\infty(s\Z^2)} + 74\pi \frac{K}{C}e^{-\frac{Cb^2}4},
    $$
    which implies the claim.
\end{proof}

\subsection{Controlling holomorphic functions in terms of their boundary values}
We begin with recalling a classical result from complex analysis.

\begin{theorem}[Hadamard's three line theorem]\label{thm:hadamard3}
Suppose that $f$ is bounded on the strip $\{x+iy:\, 0\le y\le b\}$ and holomorphic in the interior. 
Let $m_f(t):= \sup_{x\in\R} |f(x+it)|$. Then it holds for all $t\in [0,b]$ that 
$$
m_f(t) \le m_f(0)^{1-\frac{t}{b}} \cdot m_f(b)^{\frac{t}{b}}.
$$
\end{theorem}

Recall that for $F$ holomorphic on $\C^d$ (or only on a tube $\{|\Im(y)|\le b\}$) we have defined
$$
M_F(t)=\sup \big\{ |F(z)|: \, z\in \C^d, \, |\im(z)|=t\big\}.
$$
When $d=1$, {\it i.e.} $F$ is a function of one complex variable, $M_F(t)=\max\bigl(m_F(t),m_F(-t)\bigr)$ and Hadamard's three line theorem is easily seen to be also
valid with $M_F$ replacing $m_F$. 
We will require a version of this fact for several complex variables for which we did not find a convenient reference:

\begin{lemma}\label{lem:hadamardscv}
Suppose that $F(z)$ is bounded on $\{x+iy\in\C^d: |\Im(y)|\le b\}$ and holomorphic in the interior. 
It holds for all $t\in [0,b]$ that 
$$
M_F(t) \le M_F(0)^{1-\frac{t}{b}} \, M_F(b)^{\frac{t}{b}}.
$$
\end{lemma}

\begin{remark}
    We are going to use this inequality in the following way: the function $F$ we will consider belongs to some class
    of holomorphic functions for which an upper bound of $M_F(t)$ is known. Then if $M_F(0)$ is small, $M_F(0)\leq \eps$,
    we obtain that $M_F(t)\leq M_F(b)^{\frac{t}{b}}\eps^{1-\frac{t}{b}}$ which shows propagation of smallness away from $\R^d$.
\end{remark}

\begin{proof}
    The statement is trivial if $t\in\{0,b\}$ is one of the endpoints.
    Let us fix $t\in (0,b)$  and pick an arbitrary $\zeta\in\C^d$ with the property that $|\im(\zeta)|=t$.
    It suffices to show that 
    $$
    |F(\zeta)|\le M_F(0)^{1-\frac{t}{b}} \cdot M_F(b)^{\frac{t}{b}}.
    $$
    Consider the function $f$ defined by
    $$
    f(z)=F\left( \frac{\im(\zeta)}{t} z + \re(\zeta)\right),\quad z\in\C
    $$
    which is bounded on the strip $\{z:\, 0 \le \im(z) \le b\}$, holomorphic in the interior and satisfies $f(it)=F(\zeta)$.
    Moreover, if $z=x+i0$ is real-valued we have that $|f(z)|\le M(0)$ and if $z=x+ib$ lies on the top boundary of the strip then
    $$
    \left|\im \left(\frac{\im(\zeta)}{t} z + \re(\zeta)\right)   \right| = b,
    $$
    which implies that $|f(x+ib)|\le M_F(b)$. 
    Thus, applying Theorem \ref{thm:hadamard3} on $f$ gives that
    $$
    |F(\zeta)|=|f(it)|\le M_F(0)^{1-\frac{t}{b}} \cdot M_F(b)^{\frac{t}{b}},
    $$
    as announced.
\end{proof}

\subsection{Vertex weighted graphs}\label{sec:vwg}

As opposed to the standard notion of a weighted graph where the edges are furnished with weights, the case where vertices are weighted seems to be scarcely considered. One exception in the literature is the paper by Chung and Langlands \cite{chung:combinatorial}. We follow the introductory part of their 
paper to get familiar with the concept of vertex-weighted graphs.

We begin with an undirected graph $G=(V,E)$ which we assume to have no self-loops.
If vertices $u,v$ are neighbors (that is, if $(u,v)\in E$) we will simply write $u\sim v$.
Let us further assume that to each vertex $v\in V$ there is an associated weight denoted by $\alpha_v\ge 0$. 
The Laplacian of the vertex-weighted graph $G=(V,E,\alpha)$ will be denoted by $\mathcal{L}^G$ and is defined by the matrix 
\begin{equation}
    \mathcal{L}^G(u,v) = 
    \begin{cases}
        \sum_{z\sim u} \alpha_z,\quad & u=v\\
        -\sqrt{\alpha_u \alpha_v}, \quad & u\sim v\\
        0, \quad & \text{otherwise}.
    \end{cases}
\end{equation}
To $G$ one then associates an unweighted graph $G'=(V,E')$ with $E'$ defined by
$$
(u,v)\in E' \quad\mbox{if and only if}  \quad (u,v)\in E \quad \text{and}\quad \alpha_u \alpha_v >0. 
$$
We will say that the vertex weighted-graph $G=(V,E,\alpha)$ is \emph{connected} if the unweighted graph $G'=(V,E')$ is connected in the usual sense.

\begin{lemma}\label{lem:graphconnection}
    Let $G=(V,E,\alpha)$ be a vertex-weighted graph, let $\mathcal{L}^G$ be its Laplacian, and let $N=|V|$.
    Then it holds that $\mathcal{L}^G$ is positive semi-definite with eigenvalues $0=\lambda_1\le \lambda_2\le \ldots\le \lambda_N$.
    Furthermore, it holds that 
    \begin{equation}
    \lambda_2 = \min_{\substack{g\in\C^N\\ \sum_j g_j\alpha_j=0}} \frac{\frac12 \sum\limits_{i=1}^N \sum\limits_{k\sim i} \alpha_i\alpha_k |g_i-g_k|^2}{\sum\limits_{i=1}^N \alpha_i |g_i|^2 }.
    \end{equation}
    In particular, $G$ is connected if and only if $\lambda_2>0$.
\end{lemma}

\begin{proof}
    The proof is mainly a reproduction of arguments from \cite{chung:combinatorial}.
    We assume w.l.o.g. that $V=\{1,\ldots,N\}$.
    Further, define a matrix $B$ on $V\times E$ by 
    $$
    B(j,e)= 
    \begin{cases}
        \sqrt{\alpha_i}, & e=(i,j), \,i\le j\\
        -\sqrt{\alpha_i}, & e=(i,j), \, i> j\\
        0 & \text{otherwise.}
    \end{cases}
    $$
    We claim that $BB^*= \mathcal{L}^G$. For entries in the diagonal we have that
    $$
    (BB^*)_{j,j} = \sum_{i\sim j, i\le j} \sqrt{\alpha_i}^2 + \sum_{i\sim j, i>j} (-\sqrt{\alpha_i})^2 =\mathcal{L}^G_{j,j}.
    $$
    For off-diagonal entries
    $
    (BB^*)_{j,k} = \sum_{e\in E} B(j,e)B(k,e)
    $
    note that by definition 
    $$
    e\in E: \, B(j,e)B(k,e)\neq 0  \quad \Rightarrow \quad e=(k,j).
    $$
    As a consequence, when $j\not=k$
    $$
    (BB^*)_{j,k} = \begin{cases}
        -\sqrt{\alpha_j \alpha_k}, & j\sim k\\
        0, &\text{otherwise.}
    \end{cases}
    $$
    and it follows that indeed $BB^*=\mathcal{L}^G$, which implies that $\mathcal{L}^G\succeq 0$.
    In particular, $\mathcal{L}^G$ has an orthonormal basis of eigenvectors and the corresponding
    eigenvalues are all non-negative, $0\le\lambda_1\le \lambda_2\le \ldots\le \lambda_N$.

    Next consider the vector $\psi=(\sqrt{\alpha_1},\ldots,\sqrt{\alpha_N})^T$ and notice that, for
    $k\in\{1,\ldots,N\}$ we have that
    \begin{equation}
        (\mathcal{L}^G \psi)_k= \sqrt{\alpha_k} \left(\sum_{i\sim k} \alpha_i \right)- 
        \sum_{i\sim k} \sqrt{\alpha_i \alpha_k} \cdot \sqrt{\alpha_i} = 0.
    \end{equation}
    So $\psi$ is an eigenvector for the eigenvalue $\lambda_1=0$.

    To prove the last part, let us assume that all the weights $\alpha_i$ are strictly positive for now.
    Let us denote $W=\diag(\alpha_i)$ and, let $L$ be defined by virtue of
    $$
    L_{i,j}= 
    \begin{cases}
        \sum_{k\sim i} \alpha_k, & j=i\\
        -\alpha_j, & i\sim j\\
        0 &\text{otherwise.}
    \end{cases}
    $$
    For $f\in\C^N$ and $i\in\{1,\ldots,N\}$ it holds that
    $$
    (Lf)_i =  (\sum_{k\sim i} \alpha_k)f_i - \sum_{j\sim i} \alpha_j f_j = \sum_{k\sim i} \alpha_k(f_i-f_k).
    $$
    Furthermore, it is easy to see that $L=W^{-1/2}\mathcal{L}^G W^{1/2}$.
    
    By expressing the spectral gap in terms of the Rayleigh coefficient, and by substituting $g=W^{-1/2}f$
    we get that 
    \begin{equation}
    \lambda_2= \min\limits_{\substack{f\in\C^N\\\langle f,\psi \rangle=0}} \frac{\langle f,\mathcal{L}^G f\rangle}{\langle f,f\rangle}
    = 
    \min\limits_{\substack{g\in\C^N\\ \sum_j g_j \alpha_j = 0}} \frac{\langle Wg,Lg\rangle}{\langle Wg,g\rangle}
    \end{equation}
    We further rewrite the numerator
    \begin{eqnarray*}
        \langle Wg, Lg\rangle &=& \sum_{i=1}^N \alpha_i g_i \cdot \left(\sum_{k\sim i} \alpha_k (\bar{g_i}-\bar{g_k}) \right)\\
        &=& 
        \sum_{i=1}^N \sum_{k\sim i}\alpha_i \alpha_k |g_i-g_k|^2 + \sum_{i=1}^N \sum_{k\sim i} \alpha_i\alpha_k g_k(\bar{g_i}-\bar{g_k})\\
       & &= \frac12 \sum_{i=1}^N \sum_{k\sim i}\alpha_i \alpha_k |g_i - g_k|^2\\
       &&+ \frac12 \sum_{i=1}^N \sum_{k\sim i}\alpha_i \alpha_k \left(|g_i|^2 -2\re(\bar{g_i}g_k) + |g_k|^2 + 2g_k\bar{g_i} -2|g_k|^2 \right)
    \end{eqnarray*}
    Let us denote the double sum in the last line by $S$. Since we know that the whole expression is non-negative, and thus in particular real-valued we get that 
    $$
    S=\frac{S+\overline{S}}{2}=\frac12 \sum_{i=1}^N \sum_{k\sim i} \alpha_i\alpha_k (|g_i|^2-|g_k|^2).
    $$
    Since each edge $e=(i,k)$ appears precisely twice, and with opposite orientation in the sum, it follows that $S=0$. With this we get that indeed 
    \begin{eqnarray*}
      \lambda_2 &=& \min_{\substack{g\in\C^N\\ \sum_j g_j\alpha_j=0}} \frac{\frac12 \sum\limits_{j=1}^N \sum\limits_{k\sim j} \alpha_j\alpha_k |g_j-g_k|^2}{\sum\limits_{j=1}^N \alpha_j |g_j|^2 }\\
      &=&\min\left\{\frac12 \sum\limits_{j=1}^N \sum\limits_{k\sim j} \alpha_j\alpha_k |g_j-g_k|^2\,:\right.\\
      &&\qquad\qquad\left. g\in\C^N\mbox{ with }\sum\limits_{j=1}^N \alpha_j |g_j|^2=1\mbox{ and }\sum\limits_{i=1}^N g_j\alpha_j=0\right\}
    \end{eqnarray*}
    The general case, where some of the weights are allowed to vanish then follows from a continuity argument.\\
    
    Let $G$ be connected so that $\alpha_j\alpha_k>0$ when $j\sim k$ and assume towards a contradiction that $\lambda_2=0$. Then it follows that a minimizer $g$ of the above problem satisfies
    $$
    \sum\limits_{k\sim j}\alpha_k |g_j-g_k|^2=0,\quad  j\in \{1,\ldots,N\}
    $$
    and must thus be a multiple of $(1,\ldots,1)^T$. This contradicts the constraint $\sum_j g_j\alpha_j=0$. Hence connectedness of $G$ implies $\lambda_2>0$.\\
    On the other hand, if $G$ is disconnected, then there exists a non-trivial component $\emptyset\subsetneq A \subsetneq V$ and one can construct an admissible $g$ by setting $g_i=1$ for $i\in A$ and $g_j=-b$ for $j\in V\setminus A$ where $b$ satisfies
    $$
    \sum_{i\in A} \alpha_i = b \sum_{i\in V\setminus A} \alpha_i,
    $$
    which achieves that the ratio vanishes and therefore, that $\lambda_2=0$.
\end{proof}

\subsection{Rank one matrix completion}\label{sec:matrixcompletion}
Matrix completion is concerned with determining a matrix $X$ given only some of its entries.
Clearly, this requires some prior knowledge on the structure of the matrix to be recovered.
We consider here the case where $X\succeq 0$ is positive semi-definite and has rank (at most) one.
That is, $X$ is of the form 
$$
X = \mathbf{x}\otimes \bar{\mathbf{x}} = \mathbf{x}\mathbf{x}^H,
$$
with $\mathbf{x}\in\C^d$ the generating vector. Note that $X(i,i)=|x_i|^2$ so we assume that we know the diagonal of $X$.
Let $E\subsetneq \{1,\ldots,d\}^2\setminus D$ ($D=\{(1,1),\ldots,(d,d)\}$) denote 
the index set of the off-diagonal information on the matrix that is available. The problem we are interested in is the following:
\begin{center}
    Given $\{X_{k,\ell}$, $(k,\ell)\in E\cup D\}$, find $X$ (or $\mathbf{x}$ up to a phase factor).
\end{center}

\smallskip

The purpose of this section is to present (a slight modification of) a result due to Demanet and Jugnon \cite{demanet17}, which relates the matrix recovery 
problem to a vertex weighted graph associated with the restriction of $X$ to $D\cup E$, and reveals that any positive semidefinite matrix $Y\succeq 0$ with 
$$
Y_{k,\ell} \approx X_{k,\ell}, \quad (k,\ell)\in E
$$
is already a good approximation for $X$ provided that the respective graph has good connectivity as quantified by its spectral gap.

\begin{theorem}\label{thm:demanetmod}
    Let $V=\{1,\ldots,d\}$, $D=\{(1,1),\ldots,(d,d)\}$ and $E\subset (V\times V)\setminus D$ so that 
    $(V,E)$ is an undirected graph with no self loops.

    For $x\in \C^d\setminus\{0\}$, let $\alpha = (|x_i|^2)_{i=1}^d$, let $(V,E,\alpha)$ be the resulting vertex weighted graph and let $\lambda_2$ denote the spectral gap of $(V,E,\alpha)$, which we assume to be $>0$.
    
    Let  $B=\sum_{(k,\ell)\in V^2} |\mathcal{L}^G_{k,\ell}|$, and let $\varepsilon >0$ such that 
    \begin{equation}\label{eq:asspteps}
        \varepsilon \le \min\left\{ \frac{|x|^2 }{d^2 }, \frac{|x|^2\lambda_2}{12 B}\right\}.
    \end{equation}
    Further, let $Y\in\C^{d\times d}$, $Y\succeq 0$ that satisfies 
    \begin{equation}\label{est:YXclose}
        |Y_{k,\ell} - x_k\overline{x_\ell}| \le \varepsilon, \quad (k,\ell)\in D\cup E.
    \end{equation}
    Then it holds that 
    $$
    \|Y-x\otimes \overline{x}\|_F \le \sqrt{\frac{3|x|^2 B}{\lambda_2}} \cdot \sqrt\varepsilon.
    $$
    Furthermore, $Y$ has a unique top eigenpair $(\eta,v)$ 
    and
    $$
    \min_{\theta\in\R} |\sqrt{\trace(Y)}v - e^{i\theta}x| \le \left( 1+2\sqrt6 \sqrt{\frac{B}{\lambda_2}}\right) \cdot \sqrt\varepsilon.
    $$
\end{theorem}

Note that the hypothesis of this theorem do not require to know $x$, only to know $x_k\overline{x_\ell}$ for $(k,\ell)\in D\cup E$. The theorem
states that if we find a matrix $Y\in\C^{d\times d}$, $Y\succeq 0$ whose $(k,\ell)$ entries with $(k,\ell)\in D\cup E$ approximate
well the data, then the top eigenvector provides a good approximation of $x$, up to a rescaling factor.

The result as well as its proof is a variation of Theorem 4 in \cite{demanet17}.
The modification consists in how we measure certain error terms. While we use $\ell^\infty$-norm in assumption \eqref{est:YXclose}, the original statement employs $\ell^1$-norm at this point.
\begin{proof}
    Throughout we denote $X=xx^*$.
    We introduce a matrix $L\in\C^{d\times d}$ by 
    $$
    L_{k,\ell} = 
    \begin{cases}
        \sum_{j\sim k} |x_j|^2, \quad & (k,\ell)\in D\\
        - x_k\overline{x_\ell}, \quad & (k,\ell)\in E\\
        0 \quad &\text{otherwise.}
    \end{cases}
    $$
    Let $\phi_k\in \R$ be such that $x_k = |x_k|e^{i\phi_k}$, $k\in V$ and let 
    $\Delta = \diag(e^{i\phi_1},\ldots,e^{i\phi_d}).$
    Note that 
    $\Delta\mathcal{L}^G \Delta^* = L$, that $L\succeq 0$ with same eigenvalues as the Laplacian $\mathcal{L}^G$, which we denote by $0=\lambda_1<\lambda_2\le \ldots\le \lambda_d$.
    Let $v_1,\ldots,v_d$ be the corresponding orthonormal basis of eigenvectors of $L$, so that 
    $$
    L = \sum_{k=1}^d \lambda_k v_k \otimes \overline{v_k}.
    $$
    We define for $k=1,\ldots,d$  numbers
    $
    c_k := \langle Y, v_k\otimes\overline{v_k}\rangle_F \ge 0
    $
    and notice that, $\langle Y,L\rangle_F = \sum_{k=1}^d \lambda_k c_k$, and     
    as $(v_k)_{k=1}^n$ is an o.n.b., that
    $$
    \sum_{k=1}^d c_k 
    = \langle Y, \sum_{k=1}^d v_k\otimes\overline{v_k}\rangle_F
    = \langle Y,I \rangle_F = \trace(Y).
    $$
    \\
    For all $k$ we have that 
    $$
    (L x)_k = \left(\sum_{j\sim k} |x_j|^2\right) x_k - \sum_{\ell\sim k} x_k \overline{x_\ell}x_\ell = 0, 
    $$
    i.e. $Lx =0$. Since by assumption $\lambda_2>0$ we conclude that $x=e^{i\theta}|x| v_1$ for suitable $\theta\in\R$. \\

    The idea of the proof is now to show that $(c_k)_{k=1}^d$ is 'front-heavy' in the sense that 
    $c_1\approx \trace(Y)$ and $c_2,\ldots,c_d\approx 0$, and deduce from this that $Y\approx X$.\\
    
    \smallskip
    On the one hand we have with \eqref{est:YXclose}  that
    $$\begin{aligned}
    \sum_{k=1}^d \lambda_k c_k &= \langle Y,L\rangle_F =\langle Y-X, L\rangle_F 
    \le \sum_{(k,\ell)\in V^2} |Y_{k,\ell}-x_k \overline{x_\ell}| \cdot |L_{k,\ell}| \\
     &= \sum_{(k,\ell)\in D\cup E} |Y_{k,\ell}-x_k \overline{x_\ell}| \cdot |L_{k,\ell}|   \le \varepsilon \sum_{(k,\ell)\in V^2} |L_{k,\ell}|\\
        &= \varepsilon \sum_{(k,\ell)\in V^2} |\mathcal{L}^G_{k,\ell}| = \varepsilon B.
    \end{aligned}
    $$
    Since $\lambda_1=0$ and $0<\lambda_2\le \lambda_3\le\ldots$ this implies that 
    $$
    \sum_{k=2}^d c_k \le \frac{\sum_{k=2}^d \lambda_k c_k}{\lambda_2} \le \frac{B\varepsilon}{\lambda_2}.
    $$
    Therefore, we have that 
    \begin{equation}
    c_1 \ge \trace(Y)- \frac{B\varepsilon}{\lambda_2}.
    \end{equation}
    Recall that, since $Y\succeq 0$, $\langle Y,Y\rangle_F = \trace(Y^2)\leq\trace(Y)^2$. With this,
    \begin{multline}
        \|Y - X \|_F^2 = \|Y - |x|^2 v_1\otimes\overline{v_1}\|_F^2 
        = \langle Y,Y\rangle_F - 2|x|^2 c_1 + |x|^4 \\
        \le \trace(Y)^2 - 2 \trace(X)c_1 + \trace(X)^2
        \le (\trace(Y)- \trace(X))^2 + \frac{2|x|^2 B}{\lambda_2}\varepsilon.
    \end{multline}
    It follows from \eqref{est:YXclose} that $\trace(Y-X)^2\le d^2\varepsilon^2$.
    Hence, we further have that 
    $$
    \|Y-X\|_F^2 \le \left(d^2\varepsilon + \frac{2 |x|^2 B}{\lambda_2} \right)\varepsilon \le
    \frac{3 |x|^2 B}{\lambda_2},
    $$
    where the final inequality follows from \eqref{eq:asspteps}.\\

    \smallskip

    To prove the second statement we will make use of \cite[Lemma 2]{demanet17}, which 
    implies that the top eigenpair is unique and that
    \begin{equation}\label{eq:topepapprox}
    \min_{\theta\in\R} \big| |x|v- e^{i\theta} x\big| \le \frac{2\sqrt2}{|x|} \|Y-xx^*\|_2
    \end{equation}
    provided that $\|Y-xx^*\|_2 < \frac{|x|^2}2$.
    By part one and assumption \eqref{eq:asspteps}, we have that
    $$
    \|Y-xx^*\|_2^2 \le \|Y-xx^*\|_F^2 \le \frac{3 |x|^2 B}{\lambda_2} \varepsilon \le \frac{|x|^4}4,
    $$
    and therefore, that \eqref{eq:topepapprox} holds. In particular,  there exists $\theta\in\R$ such that 
    $$
    \big| |x|v- e^{i\theta} x\big| \le 2 \sqrt{\frac{6B\varepsilon}{\lambda_2}}.
    $$
    With this we get that 
\begin{align}
    \left|\sqrt{\trace(Y)}v - e^{i\theta}x\right| 
    &\le \left|\sqrt{\trace(Y)} v- |x| v\right| + \left||x|v-e^{i\theta}x\right| 
    \\
    &\le \left|\sqrt{\trace(Y)} - |x|\right| + 2 \sqrt{\frac{6B\varepsilon}{\lambda_2}}
\end{align}
    Since the first term on the right hand side can be bounded by 
    \begin{equation}
        \left|\sqrt{\trace(Y)} - |x|\right| = \frac{\left|\trace(Y)-|x|^2\right|}{\sqrt{\trace(Y)} + |x|} \le \frac{\left|\trace(Y-X)\right|}{|x|} \le \frac{d\varepsilon}{|x|},
    \end{equation}
    we further get -- by making use of \eqref{eq:asspteps} -- that 
    $$
    \left|\sqrt{\trace(Y)}v - e^{i\theta}x\right| \le \frac{d\varepsilon}{|x|} +  2 \sqrt{\frac{6B\varepsilon}{\lambda_2}}
    \le \left(1+ 2\sqrt6 \sqrt{\frac{B}{\lambda_2}} \right)\sqrt\varepsilon,
    $$
    and we are done.
\end{proof}

Finally, we translate the above result into the context of our de-lifting task.

\begin{corollary}\label{cor:delifting}
    Let $f\in L^2(\R)$, let $\Lambda\subseteq \mathfrak{a}\Z^2$ finite and let $r>0$.
    Let $\mathcal{L}$ denote the Laplacian of the signal associated graph, 
    let $B=\sum_{(u,v)\in \Lambda^2} |\mathcal{L}_{u,v}|$, and let $D=\{(v,v),\,v\in \Lambda\}$.
    Further, suppose the following two assumptions:
    \begin{enumerate}[i)]
        \item $\varepsilon>0$ is such that 
        $$ 
        \varepsilon \le \min\left\{
        \frac1{|\Lambda|^2}, \frac{\lambda_2(f,\Lambda,r)}{12B}
        \right\} \cdot \|\G f\|_{\ell^2(\Lambda)}^2
        $$
        \item $Y\in \C^{\Lambda\times\Lambda}$, $Y\succeq 0$ satisfies 
        $$
        |Y_{u,v} - \G f(u)\overline{\G f(v)}| \le \varepsilon, \quad (u,v)\in D\cup E.
        $$
    \end{enumerate}
    Then, $Y$ has a unique top eigenvector $v\in \C^\Lambda$ and it holds that 
    $$
    \min_{\theta\in\R} \sum_{\lambda\in \Lambda} \big| \sqrt{\trace(Y)} v_\lambda - e^{i\theta}\G f(\lambda)\big|^2 \le \left( 1+2\sqrt6 \sqrt{\frac{B}{\lambda_2(f,\Lambda,r)}}\right) ^2  \varepsilon.
    $$
\end{corollary}

\section{An error estimate for the reconstruction from incomplete and noisy Gabor coefficients}\label{sec:recfromincomplete}

\subsection{The error estimate}
Recall that given $T,S>0$ we denote 
$\Lambda=([-T,T]\times[-S,S])\cap \mathfrak{a}\Z^2$, and that $\mathcal{R}_\Lambda:\C^\Lambda\to L^2(\R)$ is defined by 
$$
\mathcal{R}_\Lambda(c) = \sum_{\lambda\in\Lambda} c_\lambda \pi(\lambda)\psi,
$$
where $\psi$ is the canonical dual of $(\pi(\lambda)\varphi)_{\lambda\in\mathfrak{a}\Z^2}$.

In Section \ref{subse:proofofthmincsamples} we establish the following estimate.

\begin{proposition}\label{thm:recincsamples}
    Let $T,S\in\mathfrak{a}\N$, let $0<\tau<T$ and let 
$$
\Lambda' := ([-T,T]\times [-S,S]^c) \cap \mathfrak{a}\Z^2.
$$
    For all $f\in L^2(\R)$ and for all $c\in\C^\Lambda$ it holds that 
    \begin{multline}
    \|f-\mathcal{R}_\Lambda(c)\|_{L^2(-\tau,\tau)} \le 3.1
    \Bigg(
    \|\G f- c\|_{\ell^2(\Lambda)} + \|\G f\|_{\ell^2(\Lambda')}  \\
    +\sqrt{\tau+1} \cdot  e^{-\frac{\pi}{\sqrt2} (T-\tau)}\cdot  
    \|f \cdot T_x \varphi^{\frac12}\|_{L^2(\R)}
    \Bigg).
    \end{multline}
\end{proposition}

As a consequence we obtain Theorem \ref{thm:main2}.

\begin{proof}[Proof of Theorem \ref{thm:main2}]
    We just need to show that
    \begin{equation}\label{eq:estGfGammaprime}
    \|\G f\|_{\ell^2(\Lambda')} \le  2.2 \sqrt{T+1} \cdot \kappa_S(f),
    \end{equation}
    which together with Proposition \ref{thm:recincsamples} implies the result.
    Let $K,L\in\N$ be given by $T=\mathfrak{a}K$ and $S=\mathfrak{a}L$.
    Note that by applying Lemma \ref{lem:gabordisccont} with $r=\frac{\mathfrak{a}}2$ we get that, for all $\lambda \in\Lambda'$, 
    $$
    |\G f(\lambda)| \le 1.79  \|\G f\|_{L^2(B_r(\lambda))}.
    $$
    As a consequence, 
    $$
    \begin{aligned}
    \|\G f\|_{\ell^2(\Lambda')}^2 &= \sum_{k=-K}^K \sum_{\substack{\ell\in\Z\\ |\ell|\ge L+1}} |\G f(\mathfrak{a}k,\mathfrak{a}\ell)|^2 \\
    &\le 1.79^2 \sum_{k=-K}^K \sum_{\substack{\ell\in\Z\\ |\ell|\ge L+1}} \|\G f\|_{L^2(B_r(\mathfrak{a}k,\mathfrak{a}\ell))}^2.    
    \end{aligned}
    $$
    For fixed $k$, for every $|\ell|\ge L+1$ we have that 
    $$
    B_r(\mathfrak{a}k,\mathfrak{a}\ell) \subseteq [k\mathfrak{a}-r,k\mathfrak{a}+r]\times[-S,S]^c
    $$
    while for $\ell'\neq \ell$
    $$
    B_r(\mathfrak{a}k,\mathfrak{a}\ell)\cap B_r(\mathfrak{a}k,\mathfrak{a}\ell')=\emptyset.
    $$
    Therefore, for all $k$ the inner sum is bounded according to 
    $$
    \sum_{\substack{\ell\in\Z\\ |\ell|\ge L+1}} \|\G f\|_{L^2(B_r(\mathfrak{a}k,\mathfrak{a}\ell))}^2 \le 2r \kappa_S(f)^2 = \mathfrak{a} \cdot \kappa_S(f)^2,
    $$
    and we get that 
    $$
    \|\G f\|_{\ell^2(\Lambda')} \le 1.79\cdot \sqrt{2K+1} \cdot \sqrt{\mathfrak{a}} \kappa_S(f)b
    $$
    which implies \eqref{eq:estGfGammaprime}.
\end{proof}

We now turn to the proof of Theorem \ref{thm:recincsamples}. 

\subsection{Proof of Proposition \ref{thm:recincsamples}}\label{subse:proofofthmincsamples}
We will denote $I=(-\tau,\tau)$.
Recalling the reconstruction formula involving the dual window $\psi$, we decompose 
\begin{align*}
    f &= \sum_{\lambda\in\Lambda} \G f(\lambda)\cdot \pi(\lambda)\psi\\
    &= \sum_{\lambda\in\Lambda} \G f(\lambda)\cdot \pi(\lambda)\psi + \sum_{\lambda\in\Lambda'} \G f(\lambda)\cdot\pi (\lambda)\psi 
    + \sum_{\substack{k\in\Z\\ |k|> \frac{T}{\mathfrak{a}}}} g_k,
\end{align*}
where $g_k:= \sum_{\ell\in\Z} \G f(\mathfrak{a}k,\mathfrak{a}\ell) M_{\mathfrak{a}\ell}T_{\mathfrak{a}k}\psi
$.
Notice that each of the series converges unconditionally since it is a sub-series of an unconditionally convergent one.
With this, we can write 
$$
f-\mathcal{R}_\Lambda(c) = \mathcal{R}_\Lambda((\G f(\lambda)-c_\lambda)_{\lambda\in\Lambda}) + \sum_{\lambda\in\Lambda'} \G f(\lambda)\cdot\pi (\lambda)\psi + \sum_{\substack{k\in\Z\\ |k|\ge \sqrt2 T+1}} g_k,
$$
which implies that 
\begin{multline}
\|f-\mathcal{R}_\Lambda(c)\|_{L^2(I)} \le 
\big\|\mathcal{R}_\Lambda\big((\G f(\lambda)-c_\lambda)_{\lambda\in\Lambda}\big)\big\|_{L^2(I)}\\
+
\left\|\sum_{\lambda\in\Lambda'}\G f(\lambda)\cdot \pi(\lambda)\psi\right\|_{L^2(I)}
+ 
\sum_{\substack{k\in\Z\\ |k|\ge \sqrt2 T+1}} \|g_k\|_{L^2(I)}.
\end{multline}
We estimate each term on the right hand side separately.

It follows from Lemma \ref{lem:rieszbdpsi} that the first term is bounded by
\begin{align}
\big\|\mathcal{R}_\Lambda\big((\G f(\lambda)-c_\lambda)_{\lambda\in\Lambda}\big)\big\|_{L^2(I)}
&=\left
\|\sum_{\lambda\in\Lambda}(\G f(\lambda)-c_\lambda) \pi(\lambda)\psi\right\|_{L^2(\R)}\\
&\le 3.1 \left( \sum_{\lambda\in\Lambda} |\G f(\lambda)-c_\lambda|^2 \right)^\frac12.\label{errorterm1}
\end{align}

For the second term, Lemma \ref{lem:rieszbdpsi} directly gives
\begin{equation}\label{errorterm2}
\left\|\sum_{\lambda\in\Lambda'}\G f(\lambda)\cdot \pi(\lambda)\psi\right\|_{L^2(I)}
\le
3.1 \| \G f\|_{\ell^2(\Lambda')}.
\end{equation}

Let $k\in\Z$ be arbitrary but fixed, and let $\gamma_\ell := \G f(k\mathfrak{a},\ell\mathfrak{a})$ so that
$$
g_k= \sum_{\ell\in\Z} \gamma_\ell M_{\ell\mathfrak{a}} T_{k\mathfrak{a}}\psi=m \cdot  T_{k\mathfrak{a}}\psi,
$$
where $m$ is the $\dfrac{1}{\mathfrak{a}}$-periodic Fourier-series defined by 
$$
m(t) = \sum_{\ell\in\Z} \gamma_\ell e^{2\pi i \mathfrak{a}\ell t}.
$$
In particular, covering $I$ with $\lceil \mathfrak{a}|I|\rceil$ translates of $[0,1/{\mathfrak{a}}]$, we obtain
\begin{align}
\|g_k\|_{L^2(I)}^2&\leq \|m\|_{L^2(I)}^2\|T_{k\mathfrak{a}}\psi\|_{L^\infty(I)}^2\\
&\leq\ceil{\mathfrak{a} |I|}\|m\|_{L^2([0,1/{\mathfrak{a}}])}^2\|T_{k\mathfrak{a}}\psi\|_{L^\infty(I)}^2\\
&=\frac{\ceil{\mathfrak{a} |I|}}{\mathfrak{a}}\sum_{\ell\in\Z} |\gamma_\ell|^2 \cdot 
\|T_{k\mathfrak{a}}\psi\|_{L^\infty(I)}^2 \\
&\leq 2(\tau+1)\sum_{\ell\in\Z} |\gamma_\ell|^2 
\cdot \|T_{k\mathfrak{a}}\psi\|_{L^\infty(I)}^2
\label{eq:estgk}
\end{align}
where we have used Parseval's relation for the $L^2$-norm of $m$
and the fact that $\dfrac{\ceil{\mathfrak{a} |I|}}{\mathfrak{a}}\leq |I|+\dfrac{1}{\mathfrak{a}}=2\tau+\sqrt{2}$.

We will now estimate the $\ell^2$-sum over $\gamma_\ell$. To that end note that, 
\begin{eqnarray*}
\gamma_\ell &=& \langle f, M_{\ell\mathfrak{a}} T_{k\mathfrak{a}}\varphi\rangle_{L^2} 
= \langle f, e^{2\pi i \mathfrak{a}^2 k\ell} T_{k\mathfrak{a}} M_{\ell\mathfrak{a}} \varphi\rangle_{L^2}\\
&=& e^{-\pi i k\ell} \langle  T_{-k\mathfrak{a}} f, M_{\ell\mathfrak{a}} \varphi\rangle_{L^2}
=e^{-\pi i k\ell} \langle \varphi^{\frac12}\cdot T_{-k\mathfrak{a}} f, M_{\ell\mathfrak{a}} \varphi^{\frac12}\rangle_{L^2}\\
&=& e^{-\pi i k\ell} \langle \ft(\varphi^{\frac12} \cdot T_{-k\mathfrak{a}} f), \ft(M_{\ell\mathfrak{a}} \varphi^{\frac{1}{2}})\rangle_{L^2}.
\end{eqnarray*}
Now, since $\ft(\varphi^{\frac{1}{2}})= 2^{\frac{1}{8}} \varphi^2$, we obtain 
$\ft(M_{\ell\mathfrak{a}} \varphi^{\frac{1}{2}})=2^{\frac{1}{8}} T_{\ell\mathfrak{a}}\varphi^2$
thus
$$
\gamma_\ell 
= 2^\frac18 e^{-\pi i k\ell} \langle \ft(\varphi^{\frac{1}{2}} T_{-k\mathfrak{a}} f), T_{\ell\mathfrak{a}} \varphi^2 \rangle_{L^2}.
$$
Since $\varphi^2(x)=2^{\frac{1}{2}} e^{-2\pi x^2}$, it follows from Lemma \ref{lem:besselTg} that 
\begin{equation}
\begin{aligned}
\sum_{\ell\in\Z} |\gamma_\ell|^2 &\le& 2^{\frac{5}{4}} \frac{\vartheta_3(0,e^{-\frac{\pi}{2}})}{2}\big\|\ft(\varphi^{\frac{1}{2}} T_{-k\mathfrak{a}} f)\big\|_{L^2}^2\\
&=& 2^{\frac{1}{4}} \vartheta_3(0,e^{-\frac{\pi}{2}}) \|f\cdot T_{k\mathfrak{a}}\varphi^\frac12\|_{L^2}^2\label{eq:parseval}\\
&\le& 1.69 \cdot \sup_{x\in\R} \|f\cdot T_x\varphi^\frac12\|_{L^2}^2.
\end{aligned}
\end{equation}

On the other hand, by Lemma \ref{lem:bdspsi} we have that $|\psi(t)| \le e^{-\frac{\pi|t|}{\sqrt2}}$. Thus, if $\mathfrak{a}|k|>T>\tau$
$$
\|T_{k\mathfrak{a}}\psi\|_{L^\infty(I)} \le \exp\{-\frac{\pi}{\sqrt2} (\mathfrak{a}|k|-\tau)\}.
$$
Putting this inequality together with \eqref{eq:parseval} into \eqref{eq:estgk}, we obtain
\begin{equation}
    \label{eq:estnormgk}
\|g_k\|_{L^2(I)} \le 
\sqrt{2(\tau+1)} \cdot e^{-\frac\pi{\sqrt2}(\mathfrak{a}|k|-\tau)} \cdot \sqrt{1.69} \sup_{x\in\R} \|f\cdot T_x\varphi^{\frac12}\|_{L^2}.
\end{equation}
Further, we have that 
\begin{multline}
\sum_{\substack{k\in\Z\\ |k|\ge \sqrt2 T+1}} e^{-\frac{\pi}{\sqrt2} (\mathfrak{a}|k|-\tau)} = 2 e^{\frac{\pi \tau}{\sqrt2}} 
\sum_{k=\sqrt2 T+1}^{\infty} e^{-\frac{\pi}{2}k} 
= 2 e^{\frac{\pi \tau}{\sqrt2}} \cdot \frac{e^{-\frac\pi2(\sqrt2 T+1)}}{1-e^{-\frac\pi2}} \\
=
0.52\ldots \cdot 
e^{-\frac\pi{\sqrt2}(T-\tau)}.
\end{multline}
Thus, summing \eqref{eq:estnormgk}, we get the upper bound
\begin{equation}\label{errorterm3}
\sum_{\substack{k\in\Z\\ |k|\ge \sqrt2 T+1}} \|g_k\|_{L^2(I)} 
\le 
\sqrt{\tau+1} \cdot e^{-\frac{\pi}{\sqrt2}(T-\tau)} \cdot \sup_{x\in\R} \|f\cdot T_x\varphi^{\frac12}\|_{L^2}
\end{equation}
By adding up the three error terms \eqref{errorterm1}, \eqref{errorterm2} and \eqref{errorterm3} we obtain the result.

\section{Determining relative phase changes}\label{sec:phasechanges}
Suppose we are given (possibly noisy) samples of the spectrogram, that is 
$(\spect f(\lambda))_{\lambda\in\Omega}$ with $\Omega\subseteq \R^2$ a set of sampling positions.
The objective of the present section is to find a way to evaluate the tensor
$$
\Tcal_u[\G f](p)=\G f(p+u)\overline{\G f(p)}.
$$
If $u=0$, we again find the spectrogram, that is, $\Tcal_0[\G f](p) = \spect f(p)$.
However, the relevant case is to go beyond $u=0$ as this is where the information of relative phase changes between points $p$ and $p+u$ is stored.\\

\smallskip

We dedicate Section \ref{subsec:evalopmotivation} to shed some light on why $\E$ is defined the way it is and to outline the further strategy.
In Section \ref{subsec:analysisE} we analyze certain aspects regarding the continuity of $\E$ w.r.t. the argument $G$. 
In Section \ref{subsec:imlications} we discuss implications for the task of estimating relative phase changes.

\subsection{Motivation and strategy}\label{subsec:evalopmotivation}
The purpose of this paragraph is to clarify the role of the evaluation operator $\E$.
First we show that the spectrogram extends to an entire function of two variables.
\begin{lemma}\label{lem:extspectrogram}
For $f\in L^2(\R)$ let $F$ be defined by
$$
F(z) = \B f(z_1-iz_2) (\B f)^*(z_1+iz_2) e^{-\pi z^2}, \quad z = 
\begin{pmatrix}
    z_1\\ z_2
\end{pmatrix}
\in \C^2.
$$
Then $F$ is the entire extension of the spectrogram of $f$.
That is to say, $F\in\mathcal{O}(\C^2)$ and $F\big|_{\R^2}=\spect f$.

In the sequel, we will simply write $\spect f$ instead of $F$.
\end{lemma}

\begin{proof}
That $F$ is entire is obvious as it is a product of entire functions.
From the relation between Gabor and Bargmann transform \eqref{eq:relgaborbargmann}, it follows that for all $(x,y)\in\R^2$
\begin{align}
\G f(x,y) &= e^{-\pi i xy} \mathcal{B}f(x-iy) e^{-\frac\pi2 (x^2+y^2)},\label{eq:GBid1}\\
\overline{\G f(x,y)} &= e^{\pi i xy}  (\mathcal{B}f)^*(x+iy) e^{-\frac\pi2 (x^2+y^2)}.\label{eq:GBid2}
\end{align}
Thus, multiplying the two functions gives that 
$$
\spect f(x,y) = \B f(x-iy) (\B f)^*(x+iy) e^{-\pi(x^2+y^2)}.
$$
This implies that $F\big|_{\R^2}=\spect f$, and we are done.
\end{proof}

The next result describes how evaluations of the function $\spect f$  relate to the Gabor transform.

\begin{lemma}\label{lem:relFGprod}
Let $f\in L^2(\R)$ and let $\spect f$ be the holomorphic extension of the spectrogram of $f$ as given in Lemma \ref{lem:extspectrogram}.
For all $p,u \in\R^2$ it holds that 
$$
\Tcal_u [\G f](p) = \E[\spect f](p,u).
$$
\end{lemma}

\begin{proof}
    Let $L$ and $Q$ be defined as in Definition \ref{def:evalop}.
    With $p=\begin{pmatrix}
        x\\y
    \end{pmatrix}, 
    u=
    \begin{pmatrix}
        a\\b
    \end{pmatrix}$ we introduce 
    $\zeta:= L(p,u) \in \C^2$.
    Notice that 
    $$
    \zeta_1-i\zeta_2=  (x-iy) + (a-ib), \qquad \zeta_1+i\zeta_2 = x+iy
    $$
    and that 
    $$
    \zeta_1^2 + \zeta_2^2 =(\zeta_1-i\zeta_2)(\zeta_1+i\zeta_2) = |p|^2 + (x+iy)(a-ib).
    $$
    We use the identities \eqref{eq:GBid1} and \eqref{eq:GBid2} to evaluate
    \begin{align}
    \spect f(\zeta_1,\zeta_2) &= \B f(\zeta_1-i\zeta_2) (\B f)^*(\zeta_1+i\zeta_2) e^{-\pi(\zeta_1^2+\zeta_2^2)}\\
    &= \B f ((x+a)-i(y+b)) (\B f)^*(x-iy) e^{-\pi(|p|^2+(x+iy)(a-ib))}\\
    &= \G f(p+u) \overline{\G f(p)} \cdot e^{Q(p,u)}.
    \end{align}
    This implies the claim.
\end{proof}

Recall that the objective is to determine (or estimate) values of $\Tcal_u  [\G f](z)$.
Assume for a moment we were given $\spect f$, the holomorphic extension of the spectrogram, on all of $\C^2$.
In this case, one simply has to plug in $\spect f$ into $\E$ according to the preceding result.

In our situation however, the provided information is significantly weaker. First, we 
do not have direct access to $\spect f$ outside of $\R^2$.
While in theory, it is enough to know $\spect f$ on $\R^2$ (or any open subset thereof) to determine it on all of $\C^2$, to perform this extension in practice is already a challenging task.
To make things worse, we do not have access to $\spect f$ on all of $\R^2$, but only to samples $(\spect f(\lambda))_{\lambda\in\Omega}$ with $\Omega\subseteq \R^2$ finite. An additional difficulty is that those samples 
are corrupted by noise, but let us neglect this aspect for the moment.\\

As it turns out, $\spect f$ not only extends holomorphically to $\C^2$, but this
extension has also moderate growth with respect to the distance of a point $z\in\C^2$ to $\R^2$.
The following example demonstrates the importance of such growth in a one-dimensional setting.

\begin{example}
    Let us define a family of entire functions $F_\varepsilon\in\mathcal{O}(\C)$, $\varepsilon>0$ by virtue of
    $$
    F_\varepsilon(z) = \varepsilon e^{-\frac{iz}\varepsilon}, \quad z\in\C.
    $$
    Then, each $F_\varepsilon$ is $\varepsilon$-close to the zero function on $\R\subseteq \C$, that is
    $
    \|F_\varepsilon - 0\|_{L^\infty(\R)}=\varepsilon.
    $
    However, if $z=a+ib$ with $b>0$ we have that 
    $$
    |F_\varepsilon(z)-0| = \varepsilon e^{\frac{b}\varepsilon},
    $$
    which blows up if $\varepsilon \to 0$. If we consider $F_\varepsilon$ to be an estimator for the holomorphic extension of the zero function, we find that while approximation of the given data improves for $\varepsilon\to 0$ the performance of the estimator degenerates.

    \smallskip

    Instead, we may look for $F\in \mathcal{O}(\C)$ such that $M_F$ is finite on $[0,2b]$. 
    If $F$ is $\varepsilon$ close to $0$ on $\R$, i.e., $M_F(0)\le \varepsilon$ it follows from Hadamard's three line theorem that 
    $$
    |F(a+ib) - 0| \le M_F(b) \le \sqrt{M_F(0) M_F(2b)} \le \sqrt\varepsilon \cdot \sqrt{M_F(2b)}.
    $$
    This suggests that for $F$ to be a reliable estimator it is the combination of both, i) an approximation of the zero function well on $\R$ and ii) it has moderate growth behaviour as quantified by $M_F$.
\end{example}

We can now outline the approach we take for the problem at hand.
First we are going to construct a set of Ansatz functions $\mathcal{A}\subset\mathcal{O}(\C^2)$ obeying the following conditions
\begin{itemize}[--] 
    \item we want $\mathcal{A}$ to be a convex set as we want to formulate a convex problem.
         Furthermore, to make such a problem practically feasible, $\mathcal{A}$ must by parameterized by finite dimensional objects; in our case the cone of positive semi-definite matrices of a certain size will play the role of the parameter space.
    \item we need to make sure that $\mathcal{A}$ has sufficient expressive capabilities; that is, that it contains at least one function $F$ which meets both requirements. Namely, 
    i) approximating $\spect f$ well on $\Omega\subseteq \R^2$ and ii) possessing beneficial growth behaviour of the type discussed above. 
    \item we need to set up the convex problem in a way to guarantee that its solution does satisfy both of these requirements. 
\end{itemize}
Once we have identified $F\in\mathcal{A}$ with moderate growth and which approximates the samples well, say
$$
\sup_{\lambda\in\Omega} |\spect f(\lambda) - F(\lambda)| \le \varepsilon
$$
we want to use $\mathcal{E}[F](p,u)$ (an object which we can access) to 
estimate $\mathcal{E}[\spect f](p,u) = \Tcal_u[\G f](p)$ (an object which we want to access but cannot).
As $\E$ is linear, the resulting estimation error is given by 
$
|\E[\spect f - F](p,u)|.$
An application of Hadamard's three line theorem (Lemma \ref{lem:hadamardscv}) would then imply that 
\begin{equation}\label{eq:contEcal}
|\E[\spect f - F](p,u)| \le C(p,u) \|\spect f - F\|_{L^\infty(\R^2)}^{1/2}
\end{equation}
with $C(p,u)$ also depending on the growth behaviour of $\spect f - F$.
Unfortunately, we have no direct way to make sure that $F$ is close to $\spect f$ on all of $\R^2$, since we only have access to the values of the latter on a finite subset $\Omega\in\R^2$.
In the subsequent part, Section \ref{subsec:analysisE} we will employ sampling and cut-off arguments in order to establish an approximate version of \eqref{eq:contEcal}. That is, an estimate of the form 
$$
|\E[G](p,u)| \le C(p,u) \left( \|G\|_{\ell^\infty(\Omega)}^{1/2} + \epsilon(\Omega) \right)
$$
with the error term $\epsilon(\Omega)$ rapidly decaying as $\Omega$ becomes richer.

\subsection{Approximate continuity of the functional $\E$}\label{subsec:analysisE}
Throughout this section we shall employ the notation $\gamma(z) = e^{-\frac\pi8 z^2}$, $z\in\C^2$.
Moreover, we introduce the following subspace of entire functions
$$
\mathcal{O}^\infty(\C^d) := \{ G\in\mathcal{O}(\C^d):\, M_G(r) <\infty, \forall r\ge 0\}.
$$
In other words, $\mathcal{O}^\infty(\C^d) $ is the set of holomorphic functions that
are in $H^\infty(T_r)$ for every tube $T_r=\{x+iy\in\C^d\,:\ |\Im(y)|<r\}$.
For $G\in\mathcal{O}^\infty(\C^2)$, we
define a Gaussian cut-off at $\tau\in\R^2$ by
$$
G_\tau (z) := G(z) \gamma(z-\tau),\quad z\in\C^2.
$$
The term ``Gaussian cut-off'' of course refers to the restriction of $G_\tau$ to $\R^2$.
Note that this restriction is in $L^1(\R^2)$ so that its Fourier transform $\widehat{G_\tau}$ makes sense.
We are interested in functions with controlled smoothness of cut-off $G_\tau$ and introduce
\begin{equation}
    \label{eq:defvck}
    \mathcal{V}(C,K) := \{G\in\mathcal{O}^\infty(\C^2): \, |\widehat{G_\tau}| \le K e^{-C|\cdot|^2},\,\forall \tau\in\R^2\}.
\end{equation}
The main result of the present section is the following continuity estimate for the evaluation operator $\E$.

\begin{proposition}\label{prop:contestEvalop}
    Let $s,C,K>0$ 
    and let $\Omega\subseteq s\Z^2$.
    Suppose that $p,u\in\R^2$ are such that 
    $$
    \dist\left(p+\frac12 u, s\Z^2\setminus \Omega\right) \ge \sqrt{\frac{C}{2\pi}} s^{-1}.
    $$
    Then it holds for all $G\in \mathcal{V}(C,K)$ that 
    \begin{multline}
    |\mathcal{E}[G](p,u)| \le 8.6  \left( \|G\|_{\ell^\infty(\Omega)}^{1/2} + \left(\sqrt{M_G(0)} + 1.8 \sqrt{\frac{K}{C}} \right) e^{-\frac{C}{32 s^2}} \right)
    \\
    \times \sqrt{M_G(|u|)} e^{-\frac{15\pi}{32} u^2}.
    \end{multline}
\end{proposition}

\begin{proof}
    Let $\tau=p+\frac12 u$ and consider $G_\tau(z)= G(z)\gamma(z-\tau)$. \\
    
    First we rewrite $\E[G](p,u)$ in terms of $G_\tau$: Let $L,Q$ be defined as in Definition \ref{def:evalop}.
Note that $\Re Q(p,u) = -\frac\pi2 u^2$ and set 
$$
\zeta_*:= L(p,u) = p + \frac12 \begin{pmatrix}
    1 &-i\\ i&1
\end{pmatrix}
u.
$$
Note that $\zeta_*-\tau=\frac{i}2 \begin{pmatrix}
    0&-1\\1 &0
\end{pmatrix}u$.
Thus, 
\begin{align}\label{est:Econt1}
|\E[G](p,u)| &= |G(\zeta_*)| e^{-\frac\pi2 u^2} \\
&= |G_\tau(\zeta_*)| \cdot \exp\left\{\frac\pi8 \Re\{(\zeta_*-\tau)^2\} - \frac\pi2 u^2 \right\}\\
&= |G_\tau(\zeta_*)| \cdot \exp\left\{-\frac{17\pi}{32}u^2\right\}.
\end{align}
Next, we apply Hadamard's three line theorem (Lemma \ref{lem:hadamardscv}, to be precise) to $G_\tau$.
To that end, note that $G_\tau\in\mathcal{O}^\infty(\C^2)$ since $G\in\mathcal{O}^\infty(\C^2)$.
As $|\Im \zeta_*|=\frac{|u|}2$, we get that 
$$
|G_\tau(\zeta_*)| \le M_{G_\tau}(|u|/2) \le \sqrt{M_{G_\tau}(0) M_{G_\tau}(|u|)} 
\le 
\|G_\tau\|_{L^\infty(\R^2)}^{1/2} \sqrt{M_{G_\tau}(|u|)}.
$$
Moreover, as $|\gamma(\zeta-\tau)| = \exp\left\{-\frac\pi8 (\Re(\zeta)-\tau)^2 +\frac\pi8 (\Im \zeta)^2 \right\}$ we find that 
\begin{equation}\label{est:Econt2}
M_{G_\tau}(r) = \sup_{|\Im\zeta| = r} |G(\zeta) \gamma(\zeta-\tau)| \le e^{\frac\pi8 r^2} \cdot M_G(r).
\end{equation}

Next we want to replace the term $\|G_\tau\|_{L^\infty(\R^2)}$ by its discrete variant.
 Applying Lemma \ref{lem:samplinginf} gives that 
\begin{equation}
    \|G_\tau\|_{L^\infty(\R^2)} \le 73 \|G_\tau\|_{\ell^\infty(s\Z^2)} + 233 \frac{K}{C}e^{-\frac{C}{16s^2}}.
\end{equation}
Suppose that $\lambda_0\in \Omega^c:= s\Z^2\setminus \Omega$.
By assumption we have that 
$$
|\tau-\lambda_0| \ge \sqrt{\frac{C}{2\pi}} s^{-1},
$$
which implies that
$$
|G_\tau(\lambda_0)| = |G(\lambda_0)| e^{-\frac\pi8 |\lambda_0-\tau|^2} \le M_G(0) \cdot e^{-\frac{C}{16s^2}}.
$$
In particular, as $\|G_\tau \|_{\ell^\infty(s\Z^2)}\le \|G_\tau\|_{\ell^\infty(\Omega)} + \|G_\tau\|_{\ell^\infty(\Omega^c)}$, we have that 
\begin{equation}\label{est:Econt3}
    \|G_\tau\|_{L^\infty(\R^2)} \le 73 \left(\|G_\tau\|_{\ell^\infty(\Omega)} + \left( M_G(0)+ \frac{233 K}{73 C} \right) e^{-\frac{C}{16s^2}} \right).
\end{equation}

It remains to combine the estimates:
\begin{align*}
    &|\E[G](p,u)|\\
    & \stackrel{\eqref{est:Econt1}}{\le} |G_\tau(\zeta_*)| \cdot \exp\left\{-\frac{17\pi}{32} u^2 \right\}\\
    &\stackrel{\eqref{est:Econt2}}{\le} \sqrt{M_G(|u|)} \cdot \|G_\tau\|_{L^\infty(\R^2)}^{1/2} \cdot \exp\left\{-\frac{15\pi}{32} u^2 \right\}\\
    &\stackrel{\eqref{est:Econt3}}{\le} \sqrt{73} \sqrt{M_G(|u|)} \left(\|G_\tau\|_{\ell^\infty(\Omega)} + \left( M_G(0)+ \frac{233 K}{73 C} \right) e^{-\frac{C}{16s^2}} \right)^{1/2}\cdot e^{-\frac{15\pi}{32} u^2}\\
    &\le 8.6  \left( \|G\|_{\ell^\infty(\Omega)}^{1/2} + \left(\sqrt{M_G(0)} + 1.8 \sqrt{\frac{K}{C}} \right) e^{-\frac{C}{32 s^2}} \right)
    \cdot \sqrt{M_G(|u|)} e^{-\frac{15\pi}{32} u^2}
\end{align*}
This finishes the proof.
\end{proof}

\subsection{Implications}\label{subsec:imlications}

Recall the objective from the beginning of the section.

\smallskip

{\em Given samples $(\spect f(\lambda))_{\lambda\in\Omega}$ we want to estimate relative phase changes}
$$
\G f(p+u) \overline{\G f(p)} =\E[\spect f](p,u).
$$

\smallskip

The basic idea is -- since $\spect f$ is not directly available -- to identify a suitable dummy $F\in \mathcal{O}(\C^2)$ in place of $\spect f$, and use $\E[F]$ as an estimator for relative phase changes. This raises several questions.
Most importantly, how to actually identify $F$ and how accurate is the resulting estimator?
The remainder of this paragraph is aimed at resolving these questions. 
To do that we first introduce a finite-dimensional cone of Ansatz functions, then formulate an associated convex problem (CP) and finally argue that the solution of this CP gives rise to an accurate estimator. 

Throughout, let $\Gamma\subseteq\mathfrak{a}\Z^2$ denote a finite set.
The following lemma clarifies the relation between $F_A$ and the Gabor transform.
\begin{lemma}\label{lem:relansatzspectro}
    Let $N=|\Gamma|$ and let $a=(a_\lambda)_{\lambda\in\Gamma}\in \C^\Gamma$.
    The following properties hold:
    \begin{enumerate}[i)]
    \item $\Phi_{\lambda,\mu}$ is the entire extension of $\G[\pi(\lambda)\varphi]\overline{\G [\pi(\mu)\varphi]}$. That is, for all $p=(x,\omega)\in \R^2\subseteq \C^2$ it holds that 
    $$
    \Phi_{\lambda,\mu}(p) = \G[\pi(\lambda)\varphi](x,\omega) \cdot \overline{\G[\pi(\mu)\varphi](x,\omega)}.
    $$
    \item With $A=a\otimes\bar a$ and $f=\sum_{\lambda\in\Gamma}a_\lambda \pi(\lambda)\varphi$, we have that $F_A$ is the entire extension of $\spect f$. 
    That is, for all $p=(x,\omega)\in \R^2\subseteq \C^2$ it holds that 
    $$
    F_A(p) = \spect f(x,\omega).
    $$
    \item Suppose $A\in\mathfrak{A}_+(\Gamma)$ with eigenvalues
    $\alpha_1,\ldots,\alpha_N\ge 0$ and with corresponding orthonormal basis of eigenvectors $u_1,\ldots,u_N\in\C^\Gamma$. Let 
    $$
    f_k := \sum_{\lambda\in\Gamma} u_k(\lambda) \pi(\lambda)\varphi,\quad k\in\{1,\ldots,N\}.
    $$
    Then, $F_A$ is the entire extension of $\sum_{k=1}^N \alpha_k \spect f_k$. That is, for all $p=(x,\omega)\in\R^2\subseteq\C^2$ it holds that 
    $$
    F_A(p) = \sum_{k=1}^N \alpha_k \spect f_k(x,\omega).
    $$
    \end{enumerate}
\end{lemma}

\begin{proof} First note that the first statement implies the second one. For this, one simply writes $\spect f=\G f\cdot \overline{\G f}$
and uses linearity of $f\mapsto\G f$. The second statement then implies the third one by linearity
of $A\mapsto F_A$ once one writes $\displaystyle A=\sum_{k=1}^N \alpha_k u_k\otimes\overline{u_k}$.
Recalling \eqref{eq:Gphitrans}, for $\lambda=(a,b)$ we have that
    \begin{equation*}
        \G [\pi(\lambda)\varphi](x,\omega) = e^{\pi i ab} e^{-\pi ix\omega} e^{\pi i (xb-\omega a)} e^{-\frac\pi2 |z-\lambda|^2}.
    \end{equation*}
    An elementary computation shows that 
    $$|z-\lambda|^2+|z-\mu|^2 = 2 \left|z-\frac{\lambda+\mu}2 \right|^2 + \frac{|\lambda-\mu|^2}2.$$
    Hence, with $\mu=(a',b')$ we get that 
    \begin{align*}
        \big(\G [\pi(\lambda)\varphi] \cdot \overline{\G [\pi(\mu)\varphi]}\big) (z)
        &= 
        e^{\pi i(ab-a'b')} e^{\pi i [x (b-b')-\omega (a-a')]} e^{-\frac\pi2 (|z-\lambda|^2+|z-\mu|^2)}\\
        &= 
        e^{\pi i(ab-a'b')} e^{\pi i z \cdot \mathcal{J}(\lambda-\mu)} e^{\frac\pi4 |\lambda-\mu|^2} e^{-\pi \left| z-\frac{\lambda+\mu}2\right|^2} 
    \end{align*}
    which coincides with $\Phi_{\lambda,\mu}(z)$. 
\end{proof}

Later on we want to apply Proposition \ref{prop:contestEvalop} to $G=F_A-\spect f$.
Before that we show how in this case the relevant quantities appearing in the right hand side can be controlled in terms of the parameterizing matrix $A$.

\begin{lemma}\label{lem:AVCK}
    Let $A\in\mathfrak{A}_+(\Gamma)$. 
    Moreover, let 
    $$
    C = \frac{8\pi}{17}\quad \text{and}\quad K=272 \cdot \max_{\lambda\in\Gamma} A_{\lambda,\lambda}
    $$
    Then it holds that $F_A\in \mathcal{V}(C,K)$ and moreover that
    $$
    M_{F_A}(r) \le 64.1 e^{2\pi r^2}\cdot \max_{\lambda\in\Gamma} A_{\lambda,\lambda}, \quad r\ge 0.
    $$
\end{lemma}

\begin{proof}
    Committing a slight abuse of notation we denote the restrictions of $F_A$, $\gamma$ and $\Phi_{\lambda,\mu}$ to $\R^2$ by the same symbols.
    For the first statement, we need to show for arbitrary $\tau\in\R^2$ that 
    $$
    |\ft[F_A \cdot \gamma(\cdot-\tau)](\xi)| \le K e^{-C|\xi|^2},\,\quad \xi \in\R^2.
    $$
    Let $\lambda,\mu\in\R^2$ be arbitrary but fixed.
    We set
    $$
    p =p(\lambda,\mu)=\frac{4\lambda+4\mu+\tau}9 \quad \text{and} \quad q=q(\lambda,\mu)= \frac{\lambda+\mu-2\tau}6,
    $$
    so that, after an elementary computation,
    $$
    \big( x-\frac{\lambda+\mu}2 \big)^2 + \frac18 \big( x-\tau\big)^2 = 
    \frac98(x-p)^2 + q^2.
    $$
    With this we rewrite
    \begin{align*}
    \Phi_{\lambda,\mu}(x) \gamma(x-\tau) &= C(\lambda,\mu) e^{i\pi [\mathcal{J}(\lambda-\mu)]\cdot x}\cdot e^{-\pi\big(x-\frac{\lambda+\mu}2 \big)^2} \cdot e^{-\frac\pi8 (x-\tau)^2}\\
    &= C(\lambda,\mu) e^{-\pi q^2} \cdot \left(M_{\frac12 \mathcal{J}(\lambda-\mu)} T_p [e^{-\frac98 \pi\cdot^2}]\right)(x).
    \end{align*}
    As $\ft[e^{-\frac{9}8 \pi \cdot ^2}](\xi) = \frac89 e^{-\frac{8\pi}9 \xi^2}$, we get that the Fourier transform of the above function is given by
    \begin{equation}
    \ft[\Phi_{\lambda,\mu}(\cdot) \gamma(\cdot-\tau)] (\xi) = \frac89 C(\lambda,\mu) e^{-\pi q^2}  \cdot \left(T_{\frac12 \mathcal{J}(\lambda-\mu)} M_{-p} [e^{-\frac89 \pi \cdot^2}]\right)(\xi).
    \end{equation}
    With this, application of the triangle inequality yields 
    \begin{multline*}
    |\ft[F_A(\cdot)  \gamma(\cdot-\tau)](\xi)| \le \|A\|_{\max} \sum_{\lambda,\mu\in\Gamma} 
    |\ft[\Phi_{\lambda,\mu} \gamma(\cdot-\tau)](\xi)| \\
    \le \frac89 \|A\|_{\max} \sum_{\lambda,\mu\in\mathfrak{a}\Z^2}.
    e^{ -\frac\pi4 (\lambda-\mu)^2 - \frac{\pi}{36} (\lambda+\mu-2\tau)^2 -\frac{8\pi}9 \left(\xi - \frac12\mathcal{J}(\lambda-\mu) \right)^2}.
    \end{multline*}
    Note that $\mathcal{J}^2=-I$, that $\mathcal{J}^{-1}=-\mathcal{J}$ and that
    $$
    \begin{cases}
        \mathfrak{a}\Z^2 \times \mathfrak{a}\Z^2 \quad &\rightarrow \quad \mathfrak{a}\Z^2\times\mathfrak{a}\Z^2\\
        (\lambda,\mu)  \quad &\mapsto \quad (\lambda+\mu,\lambda-\mu)
    \end{cases}
    $$
    is injective (but not onto!). Thus, substituting $u=\lambda+\mu$ and $v=\lambda-\mu$ and $\xi' = \mathcal{J}\xi$ allows us to estimate the above sum from above by
    \begin{align}
        &\sum_{u,v\in\mathfrak{a}\Z^2} \exp\left\{-\frac\pi4 v^2 -\frac{\pi}{36} (u-2\tau)^2 -\frac{8\pi}9 (\xi-\frac12 \mathcal{J} v)^2\right\}\\
        =& \sum_{u,v\in\mathfrak{a}\Z^2} \exp\left\{-\frac\pi4 v^2 -\frac{\pi}{36} (u-2\tau)^2 -\frac{8\pi}9 (\xi'-\frac{v}2 )^2\right\}\\
        =& \left( \sum_{u\in\mathfrak{a}\Z^2} e^{-\frac{\pi}{36} (u-2\tau)^2}\right) 
        \left( \sum_{v\in\mathfrak{a}\Z^2} e^{-\frac\pi4 v^2 -\frac{8\pi}9 (\xi'-\frac{v}2)^2} \right).
        \label{eq:prodofgaussums}
    \end{align}
    As per Lemma \ref{lem:sumgaussshifts} we have that 
    \begin{equation*}
        \sum_{u\in\mathfrak{a}\Z^2} e^{-\frac{\pi}{36} (u-2\tau)^2} \le 
        \left(\sup_{t\in\R} \sum_{k\in\Z} e^{-\frac\pi{72} (k-t)^2 } \right)^2 
        \le 
        72\vartheta_3(0, e^{-72\pi})^2.
    \end{equation*}
    We rewrite the second sum in \eqref{eq:prodofgaussums} as 
    \begin{equation}
        \sum_{v\in\mathfrak{a}\Z^2} e^{-\frac\pi4 v^2 -\frac{8\pi}9 (\xi'-\frac{v}2)^2}
        = \prod_{\ell=1}^2 \sum_{k\in\Z} \exp\left\{-\frac\pi8 k^2 -\frac\pi9 \left(k - 2\sqrt2 \xi'_\ell\right)^2 \right\}.
    \end{equation}
    We consider for $t\in\R$ arbitrary 
    $$
    \begin{aligned}
       \sum_{k\in\Z} \exp\left\{-\frac\pi8 k^2 -\frac\pi9 \left(k - t\right)^2 \right\}
     &= e^{-\frac{\pi}{17} t^2} \sum_{k\in\Z} \exp\{ -\frac{17}{72}\pi (k-\frac{8}{17}t)^2\}\\
     &\le e^{-\frac\pi{17}t^2} \cdot  \sqrt{\frac{72}{17}} \vartheta_3(0, e^{-\frac{72}{17}\pi}),
    \end{aligned}
    $$
    where we once more made use of Lemma \ref{lem:sumgaussshifts}.
    Plugging in $t=2\sqrt2 \xi'_\ell$, $\ell\in\{1,2\}$, we obtain by combining the above estimates that 
    \begin{multline}
        |\ft[F_A(\cdot)  \gamma(\cdot-\tau)](\xi)| \\
        \le \frac89 \|A\|_{\max} \cdot \frac{72^2}{17} \cdot \vartheta_3(0, e^{-72\pi})^2 \cdot \vartheta_3(0, e^{-\frac{72}{17}\pi})^2 \cdot e^{-\frac{8\pi}{17}\xi^2 }.
    \end{multline}
    Recall that a positive definite matrix attains its maximum absolute value in the diagonal: $\|A\|_{\max} =\max_{\lambda\in\Gamma} A_{\lambda,\lambda}$. 
    Finally, as
    $$
    \frac{8}{9} \cdot \frac{72^2}{17} \cdot \vartheta_3(0, e^{-72\pi})^2 \cdot \vartheta_3(0, e^{-\frac{72}{17}\pi})^2<272,
    $$
   we find that  
    \begin{equation*}
    |\ft[F_A \cdot \gamma(\cdot-\tau)](\xi)| \le 272
    \left(\max_{\lambda\in\Gamma} A_{\lambda,\lambda}\right) e^{-\frac{8\pi}{17}\xi^2},
    \end{equation*}
    which proves the first statement.\\
    \smallskip
    
    To prove the second statement, let $\zeta\in\C^2$ such that $|\Im\zeta|=r$.
    We apply the triangle inequality to obtain 
    \begin{multline}
    |F_A(\zeta)|
    \le \|A\|_{\max} \cdot \sum_{\lambda,\mu\in\Gamma} |\Phi_{\lambda,\mu}(\zeta)| \\
    = \|A\|_{\max} \cdot \sum_{\lambda,\mu\in\Gamma} e^{- \frac\pi4 (\lambda-\mu)^2 -\pi \Im\zeta\cdot \mathcal{J}(\lambda-\mu)
    - \pi \left(\Re\zeta-\frac{\lambda+\mu}2\right)^2 
    + \pi (\Im \zeta)^2 }.
    \end{multline}
    We substitute again $v=\lambda-\mu$ and $u=\lambda+\mu$, and notice that 
    \begin{align*}
    -\frac\pi4 v^2-\pi \Im\zeta \cdot \mathcal{J}v+\pi (\Im \zeta)^2 &= -\pi \left(\Im \zeta + \frac12 \mathcal{J}v \right)^2 + 2\pi (\Im \zeta)^2 \\
    &\le -\pi \left(\Im \zeta + \frac12 \mathcal{J}v \right)^2 + 2\pi r^2.
    \end{align*}
    With this (and recalling that $(\lambda,\mu)\mapsto (\lambda+\mu,\lambda-\mu)$ is injective on $\mathfrak{a}\Z^2 \times \mathfrak{a}\Z^2$), 
    \begin{equation}\label{est:FAzeta}
        |F_A(\zeta)| \le \|A\|_{\max} \cdot e^{2\pi r^2} \cdot  \sum_{v\in \mathfrak{a}\Z^2} e^{-\pi \left( \Im\zeta + \frac12 \mathcal{J}v\right)^2}  \cdot 
        \sum_{u\in \mathfrak{a}\Z^2} e^{-\pi \left(\Re \zeta-\frac12 u \right)^2}.
    \end{equation}
    To bound the first sum on the right hand side we 
    proceed similarly as before and resort to a one dimensional sum: 
    \begin{multline}
        \sum_{v\in \mathfrak{a}\Z^2} e^{-\pi \left( \Im\zeta + \frac12 \mathcal{J}v\right)^2} 
        \le \left( \sup_{t\in\R} \sum_{k\in\Z} e^{-\pi(t+ \frac{k}{2\sqrt2})^2} \right)^2 = \left( \sup_{t\in\R} \sum_{k\in\Z} e^{-\frac{\pi}8(t+ k)^2} \right)^2\\
        \le 8 \vartheta_3(0,e^{-8\pi})^2,
    \end{multline}
    where the last inequality is again a consequence of Lemma \ref{lem:sumgaussshifts}.
    An analogous argument shows that the second sum in \eqref{est:FAzeta} is upper bounded by the same quantity. Hence,
    \begin{equation}
       |F_A(\zeta)| \le \|A\|_{\max} e^{2\pi r^2} \cdot 8^2 \vartheta_3(0,e^{-8\pi})^4 
       = \|A\|_{\max} e^{2\pi r^2} \cdot 64.0\ldots
    \end{equation}
    Since $\zeta$ with $|\Im \zeta|=r$ was arbitrary, we are done.
    \end{proof}

    A similar estimate holds true for the holomorphic extension of a spectrogram.

\begin{lemma}\label{lem:entextgrowth}
Let $f\in L^2(\R)$,  with $\spect f$ (the entire extension of) its spectrogram.
Then it holds that 
$$
M_{\spect f}(r) \le \|\spect f\|_{L^\infty(\R^2)} \cdot e^{2\pi r^2}, \quad r\ge 0.
$$
\end{lemma}

\begin{proof}
   From the definition of $\spect f$ in terms of $\B f$ in Lemma \ref{lem:extspectrogram}, we obtain and estimate
    \begin{multline}
        |\spect f(z)| =  |\B f(z_1-iz_2)| \cdot |\B f (\overline{z_1+i z_2})| \cdot e^{-\pi(\Re(z_1^2)+\Re(z_2)^2)}\\
        \le \|\G f\|_{L^\infty}^2 \cdot \exp\left\{\frac\pi2 |z_1-iz_2|^2 + \frac\pi2 |z_1+iz_2|^2 -\pi \bigl(|\Re(z)|^2 - |\Im(z)|^2\bigr)\right\}
    \end{multline}
expressing $\B f$ in terms of $\G f$ with \eqref{eq:relgaborbargmann}. But $\|\G f\|_{L^\infty}^2=\|\spect f\|_{L^\infty(\R^2)}$
and 
$$
    |z_1-iz_2|^2+|z_1+iz_2|^2=2|z_1|^2+2|z_2|^2=2|\Re(z)|^2+2|\Im(z)|^2.
$$
It follows that
$$
|\spect f(z_1,z_2)|        \le \|\spect f\|_{L^\infty(\R^2)}        e^{2\pi|\Im(z)|^2},
$$
which implies the claim.
\end{proof}

Evaluation of $F_A$ can be expressed in terms of a matrix inner product.
In view of Algorithm \ref{alg1} this fact is relevant in order to see that the CP in step $1$ can be rephrased as a SDP as well as for the practical implementation of step $1$.
\begin{lemma}\label{lem:matrixevaluationFA}
    Let $\Gamma\subseteq \mathfrak{a}\Z^2$ be finite, let $p\in\R^2\subseteq \C^2$ and 
    let $v\in \C^\Gamma$ be defined by 
    $$
    v_\lambda = \G [\pi(\lambda)\varphi](p)= \exp\left\{ -\frac\pi2(p-\lambda)^2 -\pi(x+a)(y-b)\right\}
    \quad 
    \lambda\in\Gamma.
    $$
    Then it holds for all $A\in\mathfrak{A}_+(\Gamma)$ that 
    $$
    F_A(p) = \trace( \bar{v}\otimes v \cdot  A) = \langle A, v \otimes \bar{v}\rangle_F.
    $$
\end{lemma}
\begin{proof}
    Since $A\mapsto F_A(p)$ is linear it suffices to consider the rank one case $A= a\otimes \bar{a}$.
    In this case we have that 
    \begin{multline}
        F_A(p) = \spect [\sum_{\lambda\in\Gamma} a_\lambda \pi(\lambda)\varphi](p)
        = 
        \big| \sum_{\lambda\in\Gamma} a_\lambda \G [\pi(\lambda)\varphi](p) \big|^2
        = \big| \sum_{\lambda\in\Gamma} a_\lambda v_\lambda \big|^2 \\
        = (a \cdot v)(\overline{a \cdot v})
        = \trace( a^H \bar{v} v^T a) = \trace([\bar{v}\otimes v ] [a\otimes \bar{a}]),
    \end{multline}
    which implies the claim.
\end{proof}

In the following, let $f\in L^2(\R)$ and let $s>0$.
Moreover, let $\Omega\subseteq s\Z^2$ and $\Gamma\subseteq \mathfrak{a}\Z^2$ be finite sets.
Given $p\in\R^2$, we use the notation $W_p=v\otimes \bar{v}$ with $v$ defined as in Lemma \ref{lem:matrixevaluationFA}.

\begin{definition}[Associated Convex Problem (ACP)]
Given a triple $(f,\Gamma,\Omega)$ as above, and a tolerance parameter $\varepsilon>0$ 
we define the {\em Admissible Set} as
$$
Adm_\varepsilon(f,\Gamma,\Omega) = \{A\in\mathfrak{A}_+(\Gamma): \, |\langle A,W_p\rangle_F - \spect f(p)| \le \varepsilon,\, p\in\Omega\}.
$$
The {\em Associated Convex Problem (ACP)} is then
\begin{equation}\label{eq:associatedsdp}
\min_{A\in Adm_\varepsilon(f,\Gamma,\Omega)}  \quad  \max_{\lambda\in\Gamma} A_{\lambda,\lambda}
\end{equation}
%
\end{definition}

Solving the ACP \eqref{eq:associatedsdp} amounts to finding among all admissible matrices the one such that the maximum of all diagonal entries is minimal. 
While the admissibility condition can be simply reformulated in terms of $F_A$, that is, 
$$
A \in Adm_\varepsilon(f,\Gamma,\Omega) \quad \Leftrightarrow \quad \|\spect f-F_A\|_{\ell^\infty(\Omega)} \le \varepsilon, 
$$
it is not yet clear why this particular objective function could be helpful. The next statement vindicates that this is indeed a good choice.
\begin{proposition}\label{prop:accuracytensorest}
    Let $f\in L^2(\R)$, let $\varepsilon\in (0,e^{-1})$ and let $s>0$ such that 
    \begin{equation}\label{eq:sofepsilon}
        s \le  \frac{\sqrt{38\pi}}{36} \left(\ln \frac1\varepsilon\right)^{-1/2}.
    \end{equation}
    Moreover, let $\Omega\subseteq s\Z^2$, let $\Gamma\subseteq \mathfrak{a}\Z^2$ and let 
    $p,u\in\R^2$ be such that 
    \begin{equation}\label{eq:puxi}
    p+\frac{1}{2} u\in \mathcal{A}(s,\Omega):=\{X\in\R^2\,:\    \dist\left(X, s\Z^2\setminus \Omega \right) \ge \frac{\sqrt{19}}9 s^{-1}\}.
    \end{equation}
    Suppose that $A\in Adm_\varepsilon(f,\Gamma,\Omega)$ and let 
    \begin{equation}
        c_0:= \|\spect f\|_{L^\infty(\R^2)} + 64.1\big(\max_{\lambda\in\Gamma} A_{\lambda,\lambda}\big).
    \end{equation}
    Then it holds that 
    \begin{equation}
        |\E[F_A](p,u) - \Tcal_u[\G f](p)| \le 8.6 (1+5.2 \sqrt{c_0})^2 \cdot \sqrt\varepsilon \cdot e^{\frac{17\pi}{32} u^2}.
    \end{equation}
\end{proposition}
\begin{center}
\begin{figure}[ht]
\begin{tikzpicture}
\draw [dotted] (0.85,0.85) rectangle (5.45,5.45);
\draw [fill=lightgray,dotted] (1.50,1.50) rectangle (4.8,4.8);
\foreach \x in {1,2,17,18} \foreach \y in {1,...,18} \draw[red,fill=red] (\x/3,\y/3) circle [radius=1pt];
\foreach \x in {3,...,16} \foreach \y in {1,2,17,18} \draw[red,fill=red] (\x/3,\y/3) circle [radius=1pt];
\foreach \x in {3,...,16} \foreach \y in {3,...,16} \draw[blue,fill=blue] (\x/3,\y/3) circle [radius=1pt];
\draw[<-] (5.45,5) -- (6.5,5);
\node at (6.9,5) {$\partial \mathcal{R}$};
\draw[<-] (4.8,4) -- (6.5,4);
\node at (7.15,4) {$\mathcal{A}(s,\Omega)$};
\end{tikzpicture}

\caption{Here $\Omega=s\Z^2\cap\mathcal{R}$ where $\mathcal{R}$ is a square.
The condition $\varepsilon<e^{-1}$ ensures that $\mathcal{A}(s,\xi)$ contains no point outside $\mathcal{R}$.}
\end{figure}
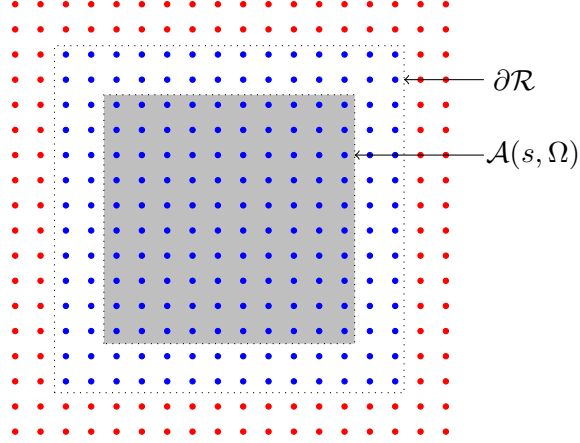
\end{center}

\begin{proof}
    We want to apply Proposition \ref{prop:contestEvalop} to $G=F_A-\spect f$.
    According to Lemma \ref{lem:fourierdecayDexp} and Lemma \ref{lem:entextgrowth}, respectively we have that 
    $$
    \spect f \in \mathcal{V}\left(\frac{38\pi}{81}, 8\|\spect f\|_{L^\infty(\R^2)} \right)
    \quad\text{and}\quad
    M_{\spect f}(r) \le \|\spect f\|_{L^\infty(\R^2)} e^{2\pi r^2}.
    $$
    As per Lemma \ref{lem:AVCK} we have that 
    \begin{equation}
        F_A \in \mathcal{V}\left(\frac{8\pi}{17}, 272 \max_{\lambda\in\Gamma} A_{\lambda,\lambda}
        \right) \quad\text{and} \quad 
        M_{F_A}(r) \le 64.1  (\max_{\lambda\in\Gamma} A_{\lambda,\lambda}) e^{2\pi r^2} .
    \end{equation}
    Note that 
    $$
    F_1 \in \mathcal{V}(C_1,K_1), \, F_2\in \mathcal{V}(C_2,K_2) \quad \Rightarrow \quad F_1-F_2 \in \mathcal{V}\big(\min\{C_1,C_2\}, K_1+K_2 \big).
    $$
    Thus, we get that 
    \begin{equation}
        G \in \mathcal{V}\left( \frac{38\pi}{81}, 8\|\spect f\|_{L^\infty} + 272 \max_{\lambda\in\Gamma} A_{\lambda,\lambda} \right) \subseteq \mathcal{V}\left( \frac{38\pi}{81}, 8c_0\right),
    \end{equation}
    and furthermore, we have that 
    \begin{equation}
        M_{G}(r) \le M_{F_A}(r)+ M_{\spect f}(r) \le \left(64.1 \max_{\lambda\in\Gamma}A_{\lambda,\lambda} + \|\spect f\|_{L^\infty} \right) e^{2\pi r^2} \le c_0 e^{2\pi r^2}.
    \end{equation}
    Note that with $C=\frac{38\pi}{81}$ we have that $\sqrt{\frac{C}{2\pi}}=\frac{\sqrt{19}}9$.
    Proposition \ref{prop:contestEvalop} therefore implies that 
    \begin{multline*}
         |\E[G](p,u)| \\
        \le  8.6  \left( \|G\|_{\ell^\infty(\Omega)}^{1/2} + \left(\sqrt{M_G(0)} + 1.8 \sqrt{\frac{81\cdot 8c_0}{38\pi}} \right) e^{-\frac{38\pi}{81\cdot32 s^2}} \right)
    \cdot \sqrt{M_G(|u|)} e^{-\frac{15\pi}{32} u^2}\\
    \le 8.6 \left( \sqrt\varepsilon + \sqrt{c_0} \left(1+1.8 \sqrt{\frac{81\cdot 8}{38\pi}} \right)  e^{-\frac{38\pi}{81\cdot32 s^2}}\right) \cdot \sqrt{c_0} e^{\frac{17\pi}{32}u^2}\\
    \le  8.6 \left(\sqrt\varepsilon + 5.2 \sqrt{c_0} e^{-\frac{38\pi}{81\cdot32 s^2}} \right)\cdot \sqrt{c_0} e^{\frac{17\pi}{32}u^2}
    \end{multline*}
    Finally, Assumption \eqref{eq:sofepsilon} is equivalent to 
    $    e^{-\frac{38\pi}{81\cdot 32s^2}} \le \sqrt{\varepsilon}$. Hence, 
    $$
    |\E[G](p,u)| \le 8.6 (1+5.2 \sqrt{c_0})^2 \cdot \sqrt\varepsilon \cdot e^{\frac{17\pi}{32}u^2},
    $$
    as desired.
    \end{proof}

    Next, we provide sufficient conditions for the set of admissible matrices to be non-empty as well as an upper bound for the objective function of a minimizer.
    \begin{proposition}\label{prop:feasibilityaprioribd}
        Let $f\in L^2(\R)$, let $\varepsilon \in (0,1)$, $s>0$ and let $\Omega\subseteq s\Z^2$.
        Moreover let $\Gamma \subseteq \mathfrak{a}\Z^2$ such that 
        \begin{equation}\label{eq:reqGamma}
            \dist\left(\Omega, \mathfrak{a}\Z^2\setminus \Gamma\right) \ge \sqrt{\frac2\pi \ln \left(\frac{53 \|\spect f\|_{L^\infty}}\varepsilon \right)} + \frac1{2\sqrt2}.
        \end{equation}
        Then it holds that $Adm_\varepsilon(f,\Gamma,\Omega)\neq \emptyset$ and every solution $A$ of the ACP satisfies 
        $$
        \max_{\lambda\in\Gamma} A_{\lambda,\lambda} \le 1.6 \|\spect f\|_{L^\infty(\R^2)}.
        $$
    \end{proposition}

\begin{proof}
    Let $\psi$ denote the canonical dual of the Gabor frame $(\pi(\lambda)\varphi)_{\lambda\in\mathfrak{a}\Z^2}$.
    Moreover, let us define $a\in \C^\Gamma$ by 
    $$
    a_\lambda = \langle f, \pi(\lambda)\psi\rangle, \quad \lambda\in\Gamma
    $$
    and set $A=a\otimes \bar{a}$. We prove the statement by showing that 
    $$ \text{i)}\, A\in Adm_\varepsilon(f,\Gamma,\Omega)\quad\text{and}\quad
    \text{ii)} \,\max_{\lambda\in\Gamma} A_{\lambda,\lambda}\le 1.6 \|\spect f\|_{L^\infty}.$$

     To prove i) let us consider $\xi\in\Omega$ arbitrary but fixed.
    We need to show that 
    $$
    |F_A(\xi) - \spect f(\xi)| \le \varepsilon.
    $$
    Let $R$ denote the quantity on the right hand side of \eqref{eq:reqGamma} and note that by that assumption 
    $$
    \mathfrak{a}\Z^2 \setminus \Gamma \subseteq \mathfrak{a}\Z^2 \setminus B_R(\xi).
    $$
    With $g:= \sum_{\lambda\in\Gamma} a_\lambda \pi(\lambda)\varphi$ we have by the reconstruction formula \eqref{eq:dualwindowrec} that 
    $$
    \G [f-g] = \sum_{\lambda\in \mathfrak{a}\Z^2\setminus \Gamma} \langle f,\pi(\lambda)\psi\rangle
    \cdot \G [\pi(\lambda)\varphi].
    $$
    We can estimate 
    \begin{align}
    |\G [f-g](\xi)| &\le \sup_{\lambda\in\mathfrak{a}\Z^2\setminus \Gamma}|\langle f,\pi(\lambda)\psi\rangle| \cdot 
    \sum_{\lambda\in\mathfrak{a}\Z^2\setminus \Gamma} |\G [\pi(\lambda)\varphi](\xi)| \nonumber\\
    & \le \sup_{\lambda\in\R^2}|\langle f,\pi(\lambda)\psi\rangle| 
    \cdot 
    \sum_{\lambda\in\mathfrak{a}\Z^2\setminus B_R(\xi)} e^{-\frac\pi2 |\lambda-\xi|^2}.\label{est:feasbility}
    \end{align}

    We use that the Gabor transform is unitary and that $|\G [\pi(\lambda)\psi](p)|= |\G \psi(p-\lambda)|$ in order to bound
    \begin{multline}\label{est:fpsiip}
        |\langle f,\pi(\lambda)\psi\rangle|         = 
        |\langle \G f, \G [\pi(\lambda)\psi]\rangle_{L^2(\R^2)} |\\
        \le 
        \int_{\R^2} |\G f(y)| |\G \psi(y-\lambda)|\,\mbox{d}y
        \le \|\G f\|_{L^\infty} \cdot \|\G \psi\|_{L^1}
        \le 1.23 \|\G f\|_{L^\infty}
    \end{multline}
    where the last inequality follows from Lemma \ref{lem:bdspsi}.

    We proceed with estimating the sum of gaussians.
    A simple computation shows that 
    $$
    x\mapsto e^{-\frac\pi2 |x-\xi|^2}, \quad x\in \R^2\setminus B_{\sqrt{\frac2\pi}} (\xi)
    $$
    is subharmonic. For every $\lambda\in\mathfrak{a}\Z^2\setminus B_R(\xi)$ we have that 
    $$
    B_{\frac{\mathfrak{a}}2}(\lambda) \cap B_{\sqrt{\frac2\pi}}(\xi) = \emptyset,
    $$
    since $|\lambda-\xi|\ge R \ge 1.2 > \frac{\mathfrak{a}}2 + \sqrt{\frac2\pi}$.
    By subharmonicity we have for all such $\lambda$ that 
    $$
    e^{-\frac\pi2|\lambda-\xi|^2} \le \frac{1}{|B_{\mathfrak{a}/2}(\lambda)|} \int_{B_{\mathfrak{a}/2}(\lambda)} e^{-\frac\pi2 |x-\xi|^2}\,\mbox{d}x
    =
   \frac8\pi \int_{B_{\mathfrak{a}/2}(\lambda)} e^{-\frac\pi2 |x-\xi|^2}\,\mbox{d}x.
    $$
    Since we have that 
    $$
    \bigcup_{\lambda\in \mathfrak{a}\Z^2 \setminus B_R(\xi)} B_{\mathfrak{a}/2}(\lambda) \subseteq 
    \R^2 \setminus B_{R-\frac{\mathfrak{a}}2}(\xi),
    $$
    and since all these disks are pairwise disjoint, we can estimate 
    \begin{align*}
        \sum_{\lambda\in\mathfrak{a}\Z^2\setminus B_R(\xi)} e^{-\frac\pi2 |\lambda-\xi|^2} 
        & \le \frac8\pi \int_{\R^2 \setminus B_{R-\frac{\mathfrak{a}}2}(\xi)} e^{-\frac\pi2 |x-\xi|^2}\,\mbox{d}x \\
        & = \frac8\pi \int_{\R^2 \setminus B_{R-\frac{\mathfrak{a}}2}(0)} e^{-\frac\pi2 |x|^2}\,\mbox{d}x
        = \frac{16}\pi e^{-\frac\pi2 (R-\mathfrak{a}/2)^2}.
    \end{align*}
    Thus, by injecting this and \eqref{est:fpsiip} into \eqref{est:feasbility}, we obtain 
    $$
    |\G [f-g](\xi)| \le 6.3 \cdot \frac{16}{\pi} \|\G f\|_{L^\infty} e^{-\frac\pi2 (R-\mathfrak{a}/2)^2}.
    $$
    With this we can bound
    \begin{align*}
    |\spect f(\xi) - F_A(\xi)| &=  ||\G f(\xi)| - |\G g(\xi)|| \cdot  (|\G f(\xi)| + |\G g(\xi)|)\\
    &\le |\G [f-g](\xi)| \cdot (2 |\G f(\xi)| + |\G [f-g](\xi)|)\\
    &\le 6.3 \left( 2 + 6.3\right)  \|\spect f\|_{L^\infty(\R^2)} e^{-\frac\pi2 (R-\mathfrak{a}/2)^2} \\
    &\le 53 \|\spect f\|_{L^\infty(\R^2)} e^{-\frac\pi2 (R-\mathfrak{a}/2)^2}
    \end{align*}
    which is bounded by $\varepsilon$ according to the choice of $R$.\\

    Part ii) now follows directly from estimate \eqref{est:fpsiip}: For arbitrary $\lambda\in\Gamma$ we have that
    $$
    A_{\lambda,\lambda} = |a_\lambda|^2 = |\langle f, \pi(\lambda)\psi\rangle|^2 \le 1.23^2 \|\spect f\|_{L^\infty},
    $$
    which --  as $1.23^2 < 1.6$ -- implies the claimed inequality.
\end{proof}

\section{Proof of Theorem \ref{thm:main}}
\label{sec:proofthm1}

\begin{proof}[Proof of Theorem \ref{thm:main}]
    We break up the proof into various steps and show that
    \begin{enumerate}[a)]
    \item Step 1 is feasible, and any solution $A$ satisfies $A\in Adm_{\frac{3\varepsilon}2}(f,\Gamma,\Omega)$ and 
     $$\max_{\lambda\in\Gamma} A_{\lambda,\lambda}\le 1.6.$$
    \item For all $(\lambda,\lambda') \in \mathcal{P}$ it holds that 
    \begin{equation}\label{eq:estTerror}
    |T_{\lambda',\lambda} - \Tcal_{\lambda'-\lambda} [\G f](\lambda)| \le\varepsilon' = (3.1\cdot  10^4) \sqrt\varepsilon e^{\frac{17\pi}{32}r^2}.
    \end{equation}
    \item There exists a feasible $Y$ for the CP in step 2  (see \eqref{eq:convfeasprob}).
    \item Any such $Y$ has a simple largest eigenvalue, and the corresponding eigenvector $v$ obeys \eqref{eq:bounddelift}.
    \end{enumerate}

    \smallskip

To establish a) we first observe that if  $A\in \mathfrak{A}_+(\Gamma)$ is feasible for step 1, then by the triangle inequality 
$$
\|F_A - \spect f\|_{\ell^\infty(\Omega)}\le \|F_A-\sigma\|_{\ell^\infty(\Omega)} + \|\sigma-\spect f\|_{\ell^\infty(\Omega)} \le  \varepsilon + \frac\varepsilon2,
$$
i.e.
$A\in Adm_\varepsilon(f,\Gamma,\Omega)$.
It remains to show that there exists $A\in Adm_{\frac\varepsilon2}(f,\Gamma,\Omega)$ which further satisfies the inequality $\max_{\lambda\in\Gamma}A_{\lambda,\lambda}\le 1.6$, as by the triangle inequality such a matrix is then feasible for step 1.\\

    By Proposition \ref{prop:feasibilityaprioribd} (applied to $\frac\varepsilon2$) 
    we only need to  show that 
    \begin{equation}\label{eq:condstepa}
        \dist\left(\Omega, \mathfrak{a}\Z^2\setminus \Gamma \right) \ge \sqrt{\frac2\pi \ln\left(\frac{2\cdot 53}{\varepsilon}\right)} + \frac1{2\sqrt2}.
    \end{equation}
    Note that for every $\varepsilon\in (0,1)$ the right hand side can be bounded according to 
    \begin{multline}
        \sqrt{\frac2\pi \ln\left(\frac{106}{\varepsilon}\right)} + \frac1{2\sqrt2} = 
        \sqrt{\frac2\pi \ln(106) - \frac2\pi \ln\varepsilon} + \frac1{2\sqrt2}\\
        \le \sqrt{\frac2\pi \ln(106)}+ \left(-\sqrt\frac2\pi \ln\varepsilon \right)^{1/2} + \frac1{2\sqrt2} 
        <
        2.1 + 0.9 \sqrt{\ln \frac1\varepsilon}
    \end{multline}
   By construction of the sets and with condition \eqref{eq:condR} we have 
    $$
    \dist\left(\Omega, \mathfrak{a}\Z^2\setminus \Gamma \right)\ge R \ge  2.1 + 0.9 \sqrt{\ln \frac1\varepsilon},
    $$
    and we are done.

    \smallskip

    We proceed with the proof of statement b).
    Let $(\lambda,\lambda')\in\mathcal{P}$ be arbitrary but fixed, and denote $p=\lambda$ and $u=\lambda'-\lambda$. 
    According to Proposition \ref{prop:accuracytensorest} (with $1.5 \varepsilon$ instead of $\varepsilon$), it holds that 
    \begin{multline}\label{est:stepb}
        |T_{\lambda,\lambda'}- \Tcal_{\lambda'-\lambda}[\G f](\lambda)| =
        |\mathcal{E}[F_A](p,u)-\Tcal_u[\G f](p)|\\
        \le 
        8.6 (1+5.2\sqrt{c_0})^2\cdot\sqrt{1.5\varepsilon}\cdot e^{\frac{17\pi}{32}u^2}
    \end{multline}
    with (recalling that $A$ is a solution of the CP in step $1$)
    \begin{equation}\label{eq:boundc0}
    c_0 =\|\spect f\|_{L^\infty(\R^2)}+64.1(\max_{\lambda\in\Gamma} A_{\lambda,\lambda}) \le 1+64.1\cdot 1.6 < 104
    \end{equation}
    provided that the following three conditions hold:  \begin{enumerate}[i)]
    \item $1.5 \varepsilon < e^{-1}$,
    \item $s \le \frac{\sqrt{38\pi}}{36} \left(\ln\frac2{3\varepsilon} \right)^{-\frac12}$, and
    \item $p+\frac12 u \in \mathcal{A}(s,\Omega)$.
    \end{enumerate}
    As $\frac{\sqrt{38\pi}}{36}=0.30\ldots$, condition ii) follows directly from assumption \eqref{eq:conds}.
    Since $\|\spect f\|_{L^\infty(\R^2)} \le 1$, assumption \eqref{eq:condepsilon} implies that 
    $$
    \varepsilon \le \left[\frac{e^{-\frac{17\pi}{32}r^2}}{1.33\cdot 10^5} \frac{\|\G f\|_{\ell^2(\Lambda)}^2}{|\Lambda|^2} \right]^2 \le (1.33\cdot 10^5)^{-2}, 
    $$
    which implies i).\\
    To verify iii), recall that $|u|\le r$ and that $p\in\Lambda\subseteq[-T,T]\times[-S,S]$. As 
    $$
    s\Z^2\setminus \Omega \subseteq \R^2 \setminus ([-T-R,T+R]\times [-S-R,S+R]),
    $$
    we have with \eqref{eq:condR} that 
    $$
    \dist \left(p+\frac12 u, s\Z^2\setminus \Omega \right) \ge R - \frac{r}2 \ge \frac1{2s} > \frac{\sqrt{19}}{9}s^{-1},
    $$
    which shows that indeed $p+\frac12 u \in\mathcal{A}(s,\Omega)$.
    Finally, combining \eqref{est:stepb} and \eqref{eq:boundc0} implies that 
    $$
    |T_{\lambda,\lambda'}-\Tcal_{\lambda'-\lambda}[\G f](\lambda)|\le (3.1\cdot 10^4)\cdot  \sqrt{\varepsilon} e^{\frac{17\pi}{32}u^2}
    $$
    
    \smallskip

    It follows directly from the estimate \eqref{eq:estTerror} that 
    $$
    Y=\big(\G f(\lambda)\overline{\G f(\lambda')} \big)_{\lambda,\lambda'\in \Lambda}
    $$
    meets the constraints of step 2. Thus, statement c) holds.\\

    \smallskip 

    To prove part d) we apply Corollary \ref{cor:delifting}, which states that 
    \begin{equation}\label{eq:asstcor}
    \min_{\theta\in\R}\left|\sqrt{\trace(Y)}v - e^{i\theta} (\G f(\lambda))_{\lambda\in\Lambda} \right|
    \le 
    \left(1+2\sqrt6 \sqrt{\frac{B}{\lambda_2(f,\Lambda,r)}} \right) \sqrt{\varepsilon'},
    \end{equation}
    provided that 
    \begin{equation}\label{eq:condepsprime1}
        \varepsilon' \le \min\left\{ 1, \frac{|\Lambda|^2 \lambda_2(f,\Lambda,r)}{12 B} \right\} \times \frac{\|\G f\|_{\ell^2(\Lambda)}^2}{|\Lambda|^2}
    \end{equation}
    where $B=\sum_{(u,v)\in \Lambda} |\mathcal{L}_{u,v}|$ with $\mathcal{L}$ the Laplacian of the signal associated graph.
    By the way the graph is defined, we have that the maximal degree (i.e., the maximal number of neighbors a vertex can have) is bounded from above by
    $$
    \big| \overline{B_r(0)}\cap \mathfrak{a}\Z^2| = |\overline{B_{\sqrt2 r}(0)} \cap \Z^2 | \le 
    (2\sqrt2 r)^2 = 8r^2.
    $$
    Recall that the Laplacian is given by 
    $$
    \mathcal{L}_{\lambda,\mu} =
    \begin{cases}
      \sum_{\lambda'\sim \lambda} |\G f(\lambda')|^2,  \quad &\mu=\lambda\\
      - |\G f(\lambda)| |\G f(\mu)| , \quad & \mu\sim \lambda\\
      0 &\text{otw.}
    \end{cases}
    $$
    Since $|xy|\le \dfrac{1}{2} (|x|^2+|y|^2)$ we can estimate 
    \begin{align*}
    B &\le \sum_{\lambda\in\Lambda} \left(\sum_{\substack{\lambda'\in\Lambda\\\lambda'\sim \lambda}} |\G f(\lambda')|^2 + 
     \dfrac{1}{2}\sum_{\substack{\mu\in\Lambda\\ \mu \sim \lambda}} (|\G f(\lambda)|^2 + |\G f(\mu)|^2)
    \right)\\
    & = \sum_{\lambda\in\Lambda} \sum_{\substack{\lambda'\in\Lambda\\ \lambda'\sim\lambda}} |\G f(\lambda')|^2 + 
     \dfrac{1}{2} \left(\sum_{\lambda \in\Lambda} |\G f(\lambda)|^2 \right) \left( \sum_{\substack{\mu\in\Lambda\\\mu\sim\lambda}} 1\right)
     +  \dfrac{1}{2}\sum_{\lambda\in\Lambda} \sum_{\substack{\mu\in\Lambda\\\mu\sim\lambda}} |\G f(\mu)|^2\\
     &= 2 \left(\sum_{\lambda \in\Lambda} |\G f(\lambda)|^2 \right) \left( \sum_{\substack{\mu\in\Lambda\\\mu\sim\lambda}} 1\right) \\     
     &\le 16 r^2 \|\G f\|_{\ell^2(\Lambda)}^2
    \end{align*}
    which implies that 
    $$
    \frac{|\Lambda|^2 \lambda_2(f,\Lambda,r)}{12B} \ge \frac{|\Lambda|^2 \lambda_2(f,\Lambda,r)}{192 r^2 \|\G f\|_{\ell^2(\Lambda)}^2}.
    $$
    Together with \eqref{eq:condepsilonprime} and \eqref{eq:condepsilon} we get that 
    \begin{align*}
        \varepsilon' &= (3.1\times 10^4)e^{\frac{17\pi}{32}r^2} \sqrt\varepsilon\\
        &\le \min\left\{\frac{\|\G f\|_{\ell^2(\Lambda)}^2}{|\Lambda|^2}, \frac{ \lambda_2(f,\Lambda,r)}{192 r^2 }\right\}\\
        &= \min\left\{ 
        1 , \frac{|\Lambda|^2 \lambda_2(f,\Lambda,r)}{192 r^2 \|\G f\|_{\ell^2(\Lambda)^2}} 
        \right\}\times \frac{\|\G f\|_{\ell^2(\Lambda)}^2}{|\Lambda|^2},
    \end{align*}
    which implies 
     \eqref{eq:condepsprime1}

    Rewriting the right hand side of \eqref{eq:asstcor} in terms of $\varepsilon$ and making use of the estimate for $B$, further implies 
    \begin{align*}
        \min_{\theta\in\R}\Big|\sqrt{\trace(Y)}v &- e^{i\theta} (\G f(\lambda))_{\lambda\in\Lambda} \Big|\\
    &\le 
    \left(1+2\sqrt6 \sqrt{\frac{16 r^2 \|\G f\|_{\ell^2(\Lambda)}^2}{\lambda_2(f,\Lambda,r)}} \right) 
    \times \sqrt{3.1\cdot  10^4} \,  e^{\frac{17\pi}{64}r^2} \sqrt[4]{\varepsilon}\\
    &\le 177 \cdot e^{0.84 r^2} \cdot \left(1+20 r \sqrt{\frac{\|\G f\|_{\ell^2(\Lambda)}^2}{\lambda_2(f,\Lambda,r)}} \right)
    \cdot \sqrt[4]\varepsilon,
    \end{align*}
    which finishes the proof.
\end{proof}

\section*{Acknowledgements}
Martin Rathmair was supported by the Erwin–Schr{\"o}dinger Program (J-4523) of
the Austrian Science Fund (FWF).

Ce travail a bénéficié d'une aide de l'\'Etat attribu\'e \`a l'Universit\'e de Bordeaux en
tant qu'Initiative d'excellence, au titre du plan France 2030.

\bibliography{references} 
\bibliographystyle{plain}

\end{document}